\documentclass{article}
\usepackage{graphicx} 
\usepackage{amssymb}
\usepackage{amsthm}
\usepackage{bbm}
\usepackage{enumerate}
\usepackage{mathtools}
\usepackage{mathrsfs}
\usepackage[english]{babel}
\usepackage{fullpage}
\usepackage{xcolor}
\usepackage{tikz}
\usetikzlibrary{decorations.markings, arrows.meta}
\newtheorem{thm}{Theorem}[section]
\newtheorem{lm}[thm]{Lemma}

\newtheorem{defn}[thm]{Definition}
\newtheorem{prop}[thm]{Proposition}

\newtheorem{conj}[thm]{Conjecture}

\newtheorem{rmk}[thm]{Remark}
\numberwithin{equation}{section}

\newcommand{\balpha}{\boldsymbol{\alpha}}
\newcommand{\bB}{\mathbf{B}}
\newcommand{\bbeta}{\boldsymbol{\beta}}
\newcommand{\bseta}{\boldsymbol{\eta}}
\newcommand{\bh}{\boldsymbol{h}}
\newcommand{\bn}{\boldsymbol{n}}

\newcommand{\bt}{\boldsymbol{t}}
\newcommand{\btau}{\boldsymbol{\tau}}
\newcommand{\bunion}{\sqcup}
\newcommand{\bx}{\boldsymbol{x}}
\newcommand{\bxi}{\boldsymbol{\xi}}
\newcommand{\bz}{\boldsymbol{z}}

\newcommand{\cD}{\mathcal{K}}
\newcommand{\cH}{\mathcal{H}^{\mathrm{UT}}}

\newcommand{\complexC}{\mathbb{C}}

\newcommand{\dD}{D}
\newcommand{\dist}{\mathrm{dist}}
\newcommand{\dP}{Q_1}
\newcommand{\dQ}{Q_2}
\newcommand{\FGUE}{\mathrm{F}_{\mathrm{GUE}}}
\newcommand{\ddbar}[2]{\frac{{\mathrm d}#1}{2\pi {\mathrm i}#2}}

\newcommand{\FT}{\mathrm{T}}
\newcommand{\ii}{\mathrm{i}}
\newcommand{\inn}{\mathrm{in}}

\newcommand{\LL}{\mathrm{L}}

\newcommand{\out}{\mathrm{out}}
\newcommand{\pD}{\mathcal{\hat K}}

\newcommand{\prob}{\mathbb{P}}
\newcommand{\rC}{\mathrm{C}}

\newcommand{\rd}{\mathrm{d}}
\newcommand{\realR}{\mathbb{R}}

\newcommand{\rh}{\mathrm{h}}
\newcommand{\rH}{\mathrm{H^{KPZ}}}
\newcommand{\HH}{\mathbb{H}_L}
\newcommand{\RR}{\mathrm{R}}
\newcommand{\rt}{\mathrm{t}}
\newcommand{\rx}{\mathrm{x}}

\renewcommand{\Im}{\mathrm{Im}}
\renewcommand{\Re}{\mathrm{Re}}
\newcommand{\Bts}{\mathbf{B}_{\mathrm{ts}}}

\title{An upper tail field of the KPZ fixed point}
\author{Zhipeng Liu\footnote{Department of Mathematics, University of Kansas, Lawrence, KS 66045. Email: \texttt{zhipeng@ku.edu}} \and Ruixuan Zhang\footnote{Department of Mathematics, University of Utah, Salt Lake City, UT 84112. Email: \texttt{ray.zhang@math.utah.edu}}}
\date{}

\begin{document}

\maketitle

\begin{abstract}
    The KPZ fixed point is a (1+1)-dimensional space-time random field  conjectured to be the universal limit for models within the Kardar-Parisi-Zhang (KPZ) universality class. We consider the KPZ fixed point with the narrow-wedge initial condition, conditioning on a large value at a specific point. By zooming in the neighborhood of this high point appropriately, we obtain a limiting random field, which we call an upper tail field of the KPZ fixed point. Different from the KPZ fixed point, where the time parameter has to be nonnegative, the upper tail field is defined in the full $2$-dimensional space. Especially, if we zoom out the upper tail field appropriately, it behaves like a Brownian-type field in the negative time regime, and the KPZ fixed point in the positive time regime. One main ingredient of the proof is an upper tail estimate of the joint tail probability functions of the KPZ fixed point near the given point, which generalizes the well known one-point upper tail estimate of the GUE Tracy-Widom distribution.
\end{abstract}
\section{Introduction}
\subsection{Background and motivation}
The KPZ fixed point is a (1+1)-dimensional space-time random field which has been proven or conjectured to be the universal limiting space-time field of a large class of interface growth models in the Kardar-Parisi-Zhang universality class \cite{BDJ, Jo00, TW08, TW09, ACQ,Borodin-Corwin13, borodin2016stochastic,MQR,QS22,wu2023kpz,DZ24,ACH24}. It was first rigorously constructed in \cite{MQR}, and could be viewed as a marginal of the directed landscape \cite{DOV} which is a universal random metric for the Kardar-Parisi-Zhang universality class. There have been many studies on the properties of the KPZ fixed point \cite{Corwin-Quastel-Remenik11,10.1214/22-EJP898,DAUVERGNE2024109550,quastel2022kp,baik2023differential,Liu-Wang22,ferrari2024exactdecaypersistenceprobability,DDV24,DT24}. In this paper, we mainly focus on the limiting behaviors of the KPZ fixed point when the height function at a point becomes extremely large. More explicitly, if we denote $\rH(x,t)$ the KPZ fixed point with the narrow-wedge initial condition, what does $\rH(x,t)$ look like conditioning on $\rH(0,1)\to\infty$?

This question was partly answered in \cite{Liu-Wang22,NZ22}. It turns out that, before the high point $\rH(0,1)=L$, there is a strip of size $O(L^{-1/4})$ along the line between $(0,0)$ and $(0,1)$. Within this strip, the KPZ fixed point fluctuates of $O(L^{1/4})$ and the limiting fluctuation is given by the minimum of two independent Brownian bridges. More explicitly, it is proved in \cite{Liu-Wang22} that 
\begin{equation}
\label{eq:LiuWang}
\left\{ \frac{\rH\left(\frac{x}{\sqrt{2}L^{1/4}},t\right)-t\rH(0,1)}{\sqrt{2}L^{1/4}} \, \Bigg|\, \rH(0,1) = L\right\} \to \min\left\{\mathbb{B}_1(t)+x,\mathbb{B}_2(t)-x\right\}
\end{equation}
in the sense of convergence of finite dimensional distributions when $L$ goes to infinity. Here $x\in\realR, t\in (0,1)$, and  $\mathbb{B}_1$ and $\mathbb{B}_2$ are two independent Brownian bridges. The conditional distribution should be understood as a limit of the distribution conditioned on $\{\rH(0,1) \in (L-\epsilon,L+\epsilon)\}$ as $\epsilon\to 0$. On the other hand, after the high point, the KPZ fixed point returns to an unconditioned KPZ fixed point with the narrow-wedge initial condition. This was proved for the convergence of the one-point distribution in \cite{NZ22} as below
\begin{equation}
\label{eq:NZ}
\prob \left(\rH(x,1+t)-\rH(0,1) \le h \mid \rH(0,1)=L\right)\to \FGUE\left(\frac{h}{t^{1/3}}+\frac{x^2}{t^{4/3}}\right)= \prob \left(  \mathrm{\tilde H}^{\mathrm{KPZ}}(x,t) \le h\right)
\end{equation}
as $L$ goes to infinity. Here $h, x\in\realR, t\in (0,\infty)$, $\mathrm{\tilde H}^{\mathrm{KPZ}}$ denotes a new KPZ fixed point with the narrow wedge initial condition which is independent of $\rH$, and $\FGUE$ is the GUE Tracy-Widom distribution.

The two results above imply that $\rH(x,t)$ has two different limiting behaviors before and after the conditioned high point. Both the scaling exponents and the limiting fields are totally different in these two time regimes. Therefore, it is very interesting to understand how the transition occurs. The main goal of this paper is to investigate the limiting behaviors of $\rH(x,t)$ near the high point $(0,1)$, as $\rH(0,1)$ goes to infinity. Our main result is a limit theorem for the conditional field $\{\rH(x,t)\mid \rH(0,1)\ge L\}$ near the point $(0,1)$ as $L\to\infty$. We are also able to show that the limiting field, which we call the upper tail field of the KPZ fixed point, interpolates the Brownian-like field \eqref{eq:LiuWang} and the KPZ fixed point, as we expect in the discussions above. It is a new field to the best of our knowledge, and is different from some other known transitional field from Gaussian Universality to KPZ Universality, such as those reviewed in \cite{IC16}.

There are also recent large deviation results in the upper tail regime which are relevant to this paper. For example, \cite{GH22} considered the one-point limit shape of the (solution to) KPZ equation at the time $t=1$ when the solution becomes large at the point $(0,1)$. \cite{Gaudreau_Lamarre_2023} also studied the KPZ equation and obtained the large deviation rate function of the solution in the weak noise regime under the deep upper tail condition. \cite{lin2023spacetime} proved the $n$-point fixed-time large deviation principle and characterized the space-time limit shape of the KPZ equation in the upper tail. \cite{ganguly2023brownian} proved the Brownian bridge limit of the geodesic in the directed landscape and continuum directed random
polymer in the upper tail region. Most recently, \cite{DDV24} and \cite{DT24} studied the upper tail large deviations of the directed landscape and the associated marginals. There are also related results for the periodic KPZ fixed point. \cite{baik2024pinchedup} obtained the limiting conditional field before a conditioned large height location for the periodic KPZ fixed point. Note that when the period goes to infinity, the periodic KPZ fixed point converges to the KPZ fixed point \cite{Baik-Liu-Silva22}. Hence, their result can be viewed as a generalization of \eqref{eq:LiuWang} when the period does not necessarily go to infinity.

On the other hand, it is a natural question to ask the limiting behaviors of $\{\rH(x,t) \mid \rH(0,1) \le  -L\}$ as  $L \to \infty$, which corresponds to the lower tail regime. The KPZ fixed point in the upper tail and lower tail regimes behaves very different. For the KPZ equation with a narrow wedge initial condition, predictions for the one-point limit and large deviations have been proposed in the physics literature \cite{kamenev2016short,meerson2016large,kamenev2016short}. Recently, \cite{lin2022lower} established the most probable limit shape for the KPZ equation before a conditioned extremely negative value at the time $t = 2$. However, the precise characterization of the lower-tail limit for the KPZ equation or the KPZ fixed point remains an open problem.

\subsection{Main results}
The main result of this paper is about the limit of the rescaled KPZ fixed point near a conditioned high point. The limit $\cH(\alpha,\tau)$, which we call the upper tail local limiting field or the upper tail field for short if there is no confusion, will be defined in Section \ref{sec:def_cH} via its finite-dimensional distributions, and 
\begin{equation}
    \label{eq:def_cH0}
    \cH_0(\alpha,\tau):= \cH(\alpha,\tau)-\cH(0,0),\quad (\alpha,\tau)\in\realR^2.
    \end{equation}
We also recall that $\rH(x,t)$ denotes the KPZ fixed point with the narrow-wedge initial condition.   The main theorem is as follows.

\begin{thm}
\label{thm:main}
    Assume that $\hat\alpha, \hat\tau$ and $\hat\beta$ are constants. 
   \begin{enumerate}[(a)]
   \item
    Conditioned on $\rH(\hat\alpha  L^{-1},1 + \hat\tau L^{-3/2}) \ge L + \hat\beta L^{-1/2},$ 
    \begin{equation}
    \label{eq:main_result_general}
        \sqrt{L}\left(\rH(\alpha L^{-1},1+\tau L^{-3/2})-L\right) \rightarrow \hat\beta+ \cH(\alpha- \hat\alpha,\tau- \hat\tau),\qquad (\alpha,\tau)\in \realR^2
    \end{equation}
    in the sense of convergence of finite-dimensional distributions as $L\to \infty$.  Especially, conditioned on $\rH(0,1) \ge L$,
    \begin{equation}
    \label{eq:main_result}
        \sqrt{L}\left(\rH(\alpha L^{-1},1+\tau L^{-3/2})-L\right) \rightarrow \cH(\alpha,\tau),\qquad (\alpha,\tau)\in \realR^2
    \end{equation}
    in the sense of convergence of finite-dimensional distributions as $L\to \infty$.
    \item 
    Conditioned on $\rH(\hat\alpha  L^{-1},1 + \hat\tau L^{-3/2}) = L + \hat\beta L^{-1/2}$, 
    \begin{equation}
    \label{eq:main_result_general2}
        \sqrt{L}\left(\rH(\alpha L^{-1},1+\tau L^{-3/2})-L \right) \rightarrow \hat\beta+ \cH_0(\alpha- \hat\alpha,\tau- \hat\tau),\qquad (\alpha,\tau)\in \realR^2
    \end{equation}
    in the sense of convergence of finite-dimensional distributions as $L\to \infty.$ Here the conditional probability $\prob\left(\cdot \mid \rH(x,t)=h\right)$ should be understood as $\lim_{\epsilon\to 0+} \prob\left(\cdot \mid \rH(x,t)\in (h,h+\epsilon)\right)$.      
    Especially, conditioned on $\rH(0,1) = L$,
    \begin{equation}
    \label{eq:main_result2}
        \sqrt{L}\left(\rH(\alpha L^{-1},1+\tau L^{-3/2})-L\right) \rightarrow \cH_0(\alpha,\tau),\qquad (\alpha,\tau)\in \realR^2
    \end{equation}
    in the sense of convergence of finite-dimensional distributions as $L\to \infty$.
    \end{enumerate}
\end{thm}

\begin{rmk}
Note the smaller scaling exponents in the equations \eqref{eq:main_result_general} and \eqref{eq:main_result_general2} than those in the formulas \eqref{eq:LiuWang} and \eqref{eq:NZ}. Heuristically, we need to zoom in the small neighborhood of $(0,1)$ to see how the two fields in \eqref{eq:LiuWang} and \eqref{eq:NZ} transit to each other near the point $(0,1)$.
\end{rmk}
\begin{rmk}
The two limiting fields $\cH$ and $\cH_0$ are slightly different when conditioning on $\rH(\hat\alpha  L^{-1},1 + \hat\tau L^{-3/2}) \ge L + \hat\beta L^{-1/2} $ and $\rH(\hat\alpha  L^{-1},1 + \hat\tau L^{-3/2}) = L + \hat\beta L^{-1/2} $ respectively. Intuitively, it means that the scaling is so small that the conditional field $\sqrt{L}\left(\rH(\alpha L^{-1},1+\tau L^{-3/2})-L\right)$ is sensitive to a small change of $\sqrt{L}\left(\rH(\hat\alpha L^{-1},1+\hat \tau L^{-3/2})-L\right)$.

On the other hand, one can heuristically derive that the two statements in the theorem are equivalent. For simplicity, we assume that $\hat\alpha=\hat\tau=0$. We show how one expects the second statement from the first and the other direction is similar. Assuming the first statement, for points $(\alpha_i,\tau_i)\ne (0,0)$, $1\le i\le m$,
\begin{equation}
\begin{split}
&\lim_{L\to\infty}\lim_{\epsilon\to 0} \prob \left(\bigcap_{i=1}^m\left\{\rH(\alpha_i L^{-1},1+\tau_i L^{-3/2})-L -L^{-1/2}\hat\beta \ge L^{-1/2}\beta_i\right\}\,\Big|\, L^{1/2}(\rH(0,1)-L)\in (\hat\beta,\hat\beta+\epsilon) \right)\\
&=\lim_{\epsilon\to 0}\lim_{L\to\infty} \frac{\prob\left(\bigcap_{i=1}^m\left\{\rH(\alpha_i L^{-1},1+\tau_i L^{-3/2})-L \ge L^{-1/2}(\beta_i+\hat\beta)\right\} \bigcap \left\{L^{1/2}(\rH(0,1)-L)\in (\hat\beta,\hat\beta+\epsilon)\right\}\right)}{\prob\left(L^{1/2}(\rH(0,1)-L)\in (\hat\beta,\hat\beta+\epsilon)\right)}\\
&= \lim_{\epsilon\to 0} \prob\left(\bigcap_{i=1}^m\left\{\cH(\alpha_i,\tau_i)\ge \beta_i\right\}\, \big| \, \cH(0,0)\in (0,0+\epsilon)\right)\\
&=\lim_{\epsilon\to 0} \frac{\prob\left(\bigcap_{i=1}^m\left\{\cH_0(\alpha_i,\tau_i)+\cH(0,0)\ge \beta_i\right\}\bigcap \left\{\cH(0,0)\in (0,0+\epsilon)\right\}\right)}{\prob\left(\cH(0,0)\in (0,0+\epsilon)\right)}=\prob\left(\bigcap_{i=1}^m\left\{\cH_0(\alpha_i,\tau_i)\ge \beta_i\right\}\right),
\end{split}
\end{equation}
where the last equation comes from a few properties of $\cH$ and $\cH_0$ which will be proved later: $\cH(0,0)$ is independent of $\cH_0(\alpha_i,\tau_i)=\cH(\alpha_i,\tau_i)-\cH(0,0)$, and all the random variables $\cH(\alpha_i,\tau_i)$ and $\cH(0,0)$ have a sufficiently good joint tail probability function which is differentiable. This reasoning is still heuristic due to the change of the order of limits in the first step, which we don't have a short argument to justify. Instead, we prove the two parts of the theorem separately using the steepest descent method.
\end{rmk}

\begin{rmk}
\label{rmk:generalization_thm}
Assume $(x,t)$ is a given point in $\realR\times\realR_+$. Note that the KPZ fixed point $\rH$ has the following invariance property (see \cite[Lemma 10.2]{DOV} for example)
\begin{equation}
\begin{split}
&\rH(x+\alpha L^{-1}, t+\tau L^{-3/2}) + \frac{1}{t+\tau L^{-3/2}} (x+\alpha L^{-1})^2\\
&\stackrel{d}{=} \, t^{1/3}\rH (\alpha t^{-2/3}L^{-1}, 1+ \tau t^{-1}L^{-3/2}) + \frac{1}{t+\tau L^{-3/2}} (\alpha L^{-1})^2,
\end{split}
\end{equation}
where $\stackrel{d}{=}$ denotes the equation in distribution. Thus, the convergence \eqref{eq:main_result_general} can be generalized as follows. Conditioned on $\rH(x+ \hat \alpha L^{-1}, t+ \hat\tau L^{-3/2}) + t^{-1}x^2 \ge tL + \hat\beta L^{-1/2}$, 
\begin{equation}
\label{eq:general_law}
\sqrt{L} \left( \rH(x+\alpha L^{-1}, t+ \tau L^{-3/2}) +t^{-1}x^2 -tL \right) \rightarrow \hat\beta + \cH(\alpha- \hat \alpha, \tau-\hat \tau), \quad (\alpha,\tau)\in \realR^2
\end{equation}
in the sense of convergence of finite-dimensional distributions as $L\to \infty$. And similarly, conditioned on $\rH(x+ \hat \alpha L^{-1}, t+ \hat\tau L^{-3/2}) + t^{-1}x^2 = tL + \hat\beta L^{-1/2}$, 
\begin{equation}
\label{eq:general_law2}
\sqrt{L} \left( \rH(x+\alpha L^{-1}, t+ \tau L^{-3/2}) +t^{-1}x^2 -tL \right) \rightarrow  \hat\beta+ \cH_0(\alpha- \hat \alpha, \tau-\hat \tau), \quad (\alpha,\tau)\in \realR^2,
\end{equation}
in the sense of convergence of finite-dimensional distributions as $L\to \infty$.
\end{rmk}

The proof of Theorem \ref{thm:main} is provided in Section \ref{sec:proof_mainthm}. 

\bigskip

The upper tail field $\cH(\alpha,\tau)$ obtained in Theorem \ref{thm:main} has the following properties. 
\begin{prop}
\label{prop:properties_cH}
The upper tail field $\cH(\alpha,\tau)$ satisfies:
\begin{enumerate}[(a)]
\item $\cH(0,0)$ is an exponential random variable of parameter $2$.
\item For all $x,\tau, \beta \in \realR$, we have
     \begin{equation}
     \prob\left(\cH(\alpha,\tau)\ge \beta\right) = e^{\frac{2}{3}\tau -2\beta} \prob\left(\cH(-\alpha,-\tau)\ge -\beta\right).
     \end{equation}
\item The field $\cH_0(\alpha,\tau) =\cH(\alpha,\tau)-\cH(0,0)$ is independent of $\cH(0,0)$.
\item At time $\tau=0$, the spatial process $\cH_0(\alpha,\tau=0)$ has the same distributions as $\Bts(2\alpha)-2|\alpha|$, where $\Bts$ denotes a two-sided Brownian motion with $\Bts(0)=0$.
\end{enumerate}
\end{prop}
\begin{rmk}
It is not surprising that $\cH(0,0)$ is an exponential random variable. In fact, the upper tail estimate of $\FGUE$ (see \eqref{eq:KPZ_GUE} and \eqref{eq:upper_tail_FGUE} for example) implies
\begin{equation}
    \prob\left(\rH(0,1)\ge L + \beta L^{-1/2}\mid \rH(0,1) \ge L \right)=\frac{1- \FGUE(L+\beta L^{-1/2})}{1- \FGUE(L)}\to e^{-2\beta}
\end{equation}
for fixed $\beta \ge 0$ when the large parameter $L\to\infty$. 
\end{rmk}
\begin{rmk}
The last property is due to Duncan Dauvergne. Through personal communication, Dauvergne told us the property from the perspective of the Airy line ensemble techniques. We verified the property using the formulas obtained in this paper. These properties in Proposition \ref{prop:properties_cH} also imply the following simple result: If $X$ is an exponential random variable of parameter $2$, and $Z$ is a standard Gaussian random variable independent of $X$, then $\prob(X+\sqrt{2\alpha}Z-2\alpha \ge \beta) = e^{-2\beta}\prob(X+\sqrt{2\alpha}Z-2\alpha \ge -\beta)$ holds for any $\alpha>0$ and $\beta\in\realR$. It is an elementary exercise to verify this identity, and we skip the proof since it is not needed for our argument.
\end{rmk}

The proof of Proposition \ref{prop:properties_cH} is given in Section \ref{sec:proof_property}.

\bigskip

If we zoom out the fields $\cH$ or $\cH_0$, we are able to see the Brownian-like (for the negative times) and KPZ type (for the positive times) behaviors. In other words, they are random fields on $\realR\times\realR$ that interpolate a Brownian-like field and the KPZ fixed point. More precisely, we have

\begin{prop}[Large scale limits of $\cH$ and $\cH_0$]
\label{prop:crossover}
Both $\cH$ and $\cH_0$ have the following large scale limits.
\begin{enumerate}[(a)]
\item In the negative time regime, 
\begin{equation}
\label{eq:negative_time_limit}
\frac{1}{\sqrt{2\lambda}} \left(\cH\left(\frac{\lambda^{1/2}\rx}{\sqrt{2}},\lambda \rt\right)- \lambda \rt \right), \frac{1}{\sqrt{2\lambda}} \left(\cH_0\left(\frac{\lambda^{1/2}\rx}{\sqrt{2}},\lambda \rt\right)- \lambda \rt \right) 
\end{equation}
both converge to
\begin{equation}
 \min\left\{\mathbf{B}_1(-\rt)+\rx,\mathbf{B}_2(-\rt)-\rx \right\}, \quad (\rx,\rt)\in \realR\times(-\infty,0)
\end{equation}
in the sense of convergence of finite-dimensional distributions as $\lambda\to \infty$. Here $\mathbf{B}_1$ and $\mathbf{B}_2$ are two independent standard Brownian motion.
\item In the positive time regime, both
\begin{equation}
\label{eq:positive_time_limit}
\lambda^{-1/3}   \cH(\lambda^{2/3}\rx,\lambda \rt), \lambda^{-1/3}   \cH_0(\lambda^{2/3}\rx,\lambda \rt)   \to  \rH(\rx,\rt), \quad (\rx,\rt)\in \realR\times(0, \infty)
 \end{equation}
 in the sense of convergence of finite-dimensional distributions as $\lambda\to \infty$. Here $\rH$ denotes the KPZ fixed point with the narrow-wedge initial condition.
\end{enumerate}
\end{prop}
\begin{rmk}
Both results can be extended to $\rt=0$. In fact, using Proposition \ref{prop:properties_cH} (d), we have
\begin{equation}
\frac{1}{\sqrt{2\lambda}} \cH_0\left(\frac{\lambda^{1/2}\rx}{\sqrt{2}},0\right)\stackrel{d}{=} \frac{1}{\sqrt{2\lambda}} \left(\Bts(\sqrt{2\lambda}\rx) -\sqrt{2\lambda}|\rx|\right) \to -|\rx| = \min\{\rx,-\rx\}
\end{equation}
and
\begin{equation}
\lambda^{-1/3}   \cH_0(\lambda^{2/3}\rx,0)\stackrel{d}{=}\lambda^{-1/3} \left(\Bts(2\lambda^{2/3}\rx) - 2\lambda^{2/3}|\rx|\right)\to -\infty\mathbf{1}_{\rx\ne0} = \rH(\rx,0)
\end{equation}
as $\lambda\to\infty$.
\end{rmk}
\begin{rmk}
One can easily recover the $1:2:3$ scaling invariance of the KPZ fixed point (with the narrow-wedge initial condition) from \eqref{eq:positive_time_limit}. In fact, both $\rH(\rx,\rt)$ and $c^{-1/3}\rH(c^{2/3}\rx,c\rt)$ are the limits of $$(c\lambda)^{-1/3}\cH((c\lambda)^{2/3}\rx,c\lambda \rt)= c^{-1/3}\cdot \lambda^{-1/3}\cH(\lambda^{2/3}\cdot c^{2/3}\rx,\lambda\cdot c \rt)$$ as $\lambda\to\infty$, hence they have the same finite-dimensional distributions.
\end{rmk}

We can also see that Proposition \ref{prop:crossover} is consistent with the known results \eqref{eq:LiuWang} and \eqref{eq:NZ}. In fact, since Brownian bridges locally behave like Brownian motions, we have
\begin{equation}
\label{eq:Brownian_near_edge}
\epsilon^{-1/2}\left(\mathbb{B}_1(1+\epsilon \rt) + \sqrt{\epsilon}\rx\right),\quad \epsilon^{-1/2}\left(\mathbb{B}_2(1+\epsilon \rt) - \sqrt{\epsilon} \rx\right),
\end{equation}
converges to $\mathbf{B}_1(-\rt)+\rx,\mathbf{B}_2(-\rt)-\rx$ jointly on $(\rx,\rt)\in \realR\times(-\infty,0)$, as $\epsilon\to 0$. On the other hand, the KPZ fixed point enjoys the $1:2:3$ scaling invariance, hence
\begin{equation}
\epsilon^{-1/2}  \rH(\epsilon \rx, \epsilon^{3/2}\rt) \stackrel{d}{=} \rH(\rx,\rt)
\end{equation}
for $(\rx,\rt)\in\realR\times(0,\infty)$. Thus, the fields $\cH$ and $\cH_0$ provide the transition between \eqref{eq:LiuWang} and \eqref{eq:NZ}.

\bigskip

The proof of Proposition \ref{prop:crossover} is given in Section \ref{sec:crossover}.

\bigskip

The fields $\cH$ and $\cH_0$ are new to our best knowledge. Since they are the limiting fields of the KPZ fixed point near a conditioned high point, and the KPZ fixed point is expected to be universal in the Kardar-Parisi-Zhang universality class, we expect $\cH$ and $\cH_0$ are universal local limits for all the models in the Kardar-Parisi-Zhang universality class near a conditioned high point. In an upcoming work, we are planning to verify it for the totally asymmetric simple exclusion process, which is one of the simplest models in the Kardar-Parisi-Zhang universality class. It might be possible to verify the same limits for the discrete polynuclear growth model, or equivalently,  the discrete totally asymmetric simple exclusion process, using the formulas in \cite{johansson2020long, Liao22}. However, for other models, it would be more difficult to consider the analogous upper tail limits due to the lack of exact formulas of the multipoint distribution of the height function.

We conjecture that $\cH$ and $\cH_0$ will also appear in the periodic KPZ fixed point conditioning on the upper tail event. The heuristic reason is that the scaling window for the upper tail field is so small that the periodicity might not be visable unless the period also shrink to the size of the scaling window. We leave this verification as a future project.

Finally, by the same reason, we conjecture that $\cH$ and $\cH_0$ do not depend on the initial condition. Below we heuristically show the reason why the information of the initial condition would disappear in the same scaling window for $\cH$ and $\cH_0$. If we change the narrow-wedge initial condition to the flat initial condition, the limiting conditional random field of the KPZ fixed point before the conditioned high point was also obtained in \cite{Liu-Wang22}. It turns out that the limiting field analogous to \eqref{eq:LiuWang} 
is 
\begin{equation}
\label{eq:pre_peak_flat}
\min\left\{\mathbb{B}_1(t)+x+ \frac{(1-t)Z}{\sqrt{2}}, \mathbb{B}_2(t)-x- \frac{(1-t)Z}{\sqrt{2}} \right\},\qquad (x,t)\in \realR\times (0,1)
\end{equation}
where $\mathbb{  B}_1$ and $\mathbb{  B}_2$ are two independent Brownian bridges and $Z$ is a standard normal random variable independent of $\mathbb{  B}_1$ and $\mathbb{  B}_2$. Near $t=1+\epsilon \rt \approx 1$, 
\begin{equation}
\epsilon^{-1/2}\left(\mathbb{ B}_1(1+\epsilon \rt) + \sqrt{\epsilon}\rx -\epsilon \frac{\rt Z}{\sqrt{2}}\right),\quad \epsilon^{-1/2}\left(\mathbb{ B}_2(1+\epsilon \rt) - \sqrt{\epsilon}\rx +\epsilon \frac{\rt Z}{\sqrt{2}}\right)\end{equation}
still converges to $\mathbf{B}_1(-\rt)+\rx,\mathbf{B}_2(-\rt)-\rx$ jointly on $\realR\times(-\infty,0)$, as $\epsilon\to 0$. The impact of the initial condition (which appears as an extra drift $Z$) disappears when we zoom in the field near $t\approx 1$ from below with the $1:2$ scaling. We conjecture that the same argument holds for a general initial condition. In other words, the initial condition will not affect the limiting behaviors of the KPZ fixed point near the conditioned high point $(0,1)$. More precisely, the limit is always $\min\{\mathbf{B}_1(-\rt)+\rx,\mathbf{B}_2(-\rt)-\rx\}$ when we zoom in the conditional field near the conditioned point from below. Another observation is that at the point $(0,1)$, we have the following asymptotics for any fixed $\beta\ge 0$ and sufficiently large $L$,
\begin{equation}
\begin{split}
\prob\left(\mathrm{H_{flat}^{KPZ}}(0,1)\ge L +\beta L^{-1/2} \mid \mathrm{H_{flat}^{KPZ}}(0,1)\ge L\right)&= \frac{1-\mathrm{F_{GOE}}(2^{2/3}(L +\beta L^{-1/2}))}{1-\mathrm{F_{GOE}}(2^{2/3}L)} \\
&\approx \frac{e^{-\frac{2}{3}(2^{2/3}(L+\beta L^{-1/2}))^{3/2}}}{e^{-\frac{2}{3}(2^{2/3}L)^{3/2}}}\\
&\approx e^{-2\beta}
\end{split}
\end{equation}
which is the same as the tail probability of $\cH(0,0)$. Here $\mathrm{H_{flat}^{KPZ}}$ denotes the KPZ fixed point with the flat initial condition and $\mathrm{F_{GOE}}$ is the GOE Tracy-Widom distribution. We also used the right tail asymptotics of $\mathrm{F_{GOE}}$: $1-\mathrm{F_{GOE}}(L)\approx \frac{1}{4\sqrt{\pi}L^{3/2}}e^{-\frac23 L^{3/2}}$ as $L\to\infty$. See \cite[Equations (25) and (26)]{BBD08} or \cite{DV13} for the right tail of the GOE Tracy-Widom distribution.

The above discussions suggest the following conjecture. Note that the assumption on the growth rate of the initial condition is needed to ensure the existence of the KPZ fixed point at the point $(0,1)$. 
\begin{conj}
Theorem \ref{thm:main} holds for the KPZ fixed point with a general initial condition which grows sufficiently slower than the function $f(x)=x^2$ when $|x|$ becomes large, and the upper tail field of the KPZ fixed point is independent of the initial condition.
\end{conj}

\subsection{Multipoint upper tail estimate of the KPZ fixed point}
\label{sec:main_result2}
The proof of Theorem \ref{thm:main} relies on an upper tail estimate of the KPZ fixed point on a cluster of space-time points near a given point. In order to introduce the result, we first define an order $\prec$ on $\realR^2$.
\begin{defn}
\label{defn:prec}
We say $(\alpha,\tau) \prec (\alpha',\tau')$ for two points $(\alpha,\tau), (\alpha',\tau')\in\realR^2$, if either one of the following two conditions are satisfied:
\begin{enumerate}[(i)]
\item $\tau<\tau'$;
\item $\tau=\tau'$ and $\alpha<\alpha'$.
\end{enumerate}
\end{defn}
Note that $\prec$ is a total order. Any two points on $\realR^2$ are comparable by the order $\prec$ and there is a unique way to arrange a set of distinct points on $\realR^2$ by this order.

Now we introduce the multipoint upper tail estimate. For simplicity, we consider a cluster of points near $(0,1)$. The general case can be deduced to this case using the same argument as in the Remark \ref{rmk:generalization_thm}.

Define the following rescaled KPZ fixed point near $(0,1)$
\begin{equation}
\label{eq:def_HH}
\HH(\alpha,\tau):= \sqrt{L} \left(\rH(\alpha L^{-1}, 1+ \tau  L^{-3/2}) - L\right)
\end{equation}
for $(x,\tau)\in\realR^2$ and $L>0$ satisfying $1+\tau L^{-3/2}>0$.

\begin{prop}
\label{prop:main}
Assume that $(\alpha_1,\tau_1)\prec \cdots \prec (\alpha_m,\tau_m)$ are $m$ points on the plane $\realR^2$, and $\bbeta=(\beta_1,\cdots,\beta_m)\in\realR^m$. We have
\begin{equation}
\label{eq:prop_main}
\begin{split}
&16\pi L^{3/2} e^{\frac{4}{3}L^{3/2}}\prob \left( \bigcap_{\ell=1}^m \left\{ \HH(\alpha_\ell,\tau_\ell)\ge \beta_\ell  \right\}\right)
\to \FT(\bbeta;(\alpha_1,\tau_1),\cdots,(\alpha_m,\tau_m))
\end{split}
\end{equation}
and
\begin{equation}
    \label{eq:prop_main_derivative}
    \begin{split}
    &16\pi L^{3/2} e^{\frac{4}{3}L^{3/2}} \frac{\partial}{\partial \beta_k}\prob \left( \bigcap_{\ell=1}^m \left\{ \HH(\alpha_\ell,\tau_\ell)\ge \beta_\ell  \right\}\right)
    \to \frac{\partial}{\partial \beta_k}\FT(\bbeta;(\alpha_1,\tau_1),\cdots,(\alpha_m,\tau_m)),\quad 1\le k\le m,
    \end{split}
    \end{equation}
as $L\to\infty$. The function $\FT$ is defined in Definition \ref{defn:FT}. 
\end{prop}

One special case is that when $m=1$,   
$\FT(\beta;(\alpha,\tau)) = e^{\frac{2}{3}\tau-2\beta}$. The above result implies 
\begin{equation}
\label{eq:tail}
\prob\left(\HH(\alpha,\tau)\ge \beta\right) \approx e^{\frac{2}{3}\tau-2\beta} \cdot \frac{e^{-\frac{4}{3}L^{3/2}}}{16\pi L^{3/2}}
\end{equation}
as $L\to\infty$. If we further assume that $(\alpha,\tau)=(0,0)$ and $\beta=0$, this is the well known upper tail estimate of the GUE Tracy-Widom distribution (see \cite{tracy1994level,BBD08} for example). Intuitively, the upper tail of the KPZ fixed point at a set of points near $(0,1)$ should have the same leading order as the one point upper tail at $(0,1)$, as long as the set of points are enough close to $(0,1)$ and the bounds of the heights are also enough close. The Proposition \ref{prop:main} justifies this intuition and further provides the proper scaling under which each point in the cluster affects the approximation in a nontrivial way.

The proof of Proposition \ref{prop:main} is given in Section \ref{sec:proof_proposition}.

\subsection{Strategies and structure of the proofs}

There are two main technical parts in the paper. The first one is an asymptotic analysis of the joint tail probability function with the upper tail scaling. The asymptotics of the joint distribution function before the conditioned hight point was performed in \cite{Liu-Wang22}. However, the analysis was only performed for distinct times, while the case of equal times was handled using a probabilistic argument based on the continuity of the limiting field since there was a difficulty analyzing the formula when times are equal. Another related paper \cite{NZ22} considered the asymptotics of the conditional distribution after the conditioned hight point. However, only the one-point distribution function was analyzed due to the complexity of the formula. In this paper, we did a finer analysis with a different scaling and the analysis works for the multipoint distribution case with general space-time locations, including possibly equal times. Including the general space-time locations brings extra convergence issues of the integrals in the asymptotics analysis. We handled them by using different types of contours on the two half planes (which breaks the symmetry of the formula) and some careful estimate of the integrand along these contours. Another related analysis was performed for the upper tail conditional limit of the periodic KPZ fixed point before the high point \cite{baik2024pinchedup}, which was analogous to \cite{Liu-Wang22} and the asymptotic analysis was also limited to the case with distinct times. 

The second main technical part is to show that the limit of the conditional joint tail probability functions, the functions $\hat\FT$, actually defines a nontrivial random field. We need to verify the consistency conditions of the functions $\hat\FT$ for applying the Kolmogorov extension theorem. The verification is nontrivial. We were able to prove the consistency combing the techniques from the asymptotic analysis and some probabilistic arguments. 

\bigskip

Below is the structure of the paper.

In Section \ref{sec:def_scrH}, we first define the function $\FT$, the limit of the multipoint tail function of the KPZ fixed point in the upper tail scaling. In subsection \ref{sec:def_cH}, we use $\FT$ to define the multipoint tail probability functions $\hat \FT$, and then in Proposition \ref{prop:consistency} we show that $\hat\FT$ satisfy the consistency requirements for the Kolmogorov extension theorem, hence they define a random field $\cH$. Part of the proof of Proposition \ref{prop:consistency} relies on a tail estimate of the function $\FT$ which is postponed in Subsection \ref{sec:Properties_FT}, see Proposition \ref{prop:FT_tail}. Besides, we discuss some properties  the function $\FT$ and the field $\cH_0=\cH-\cH(0,0)$, such as the differentiability and the changes of the formulas under a shift of parameters in Subsections \ref{sec:Properties_FT} and \ref{sec:Properties_cH}.

Section \ref{sec:proof_proposition} is the proof of Proposition \ref{prop:main} about the asymptotics of the multipoint tail probability function with the upper tail scaling. We first prove a formula of the joint tail probability functions for the KPZ fixed point, see Proposition \ref{prop:KPZ_tail_prob}. Then we provide the proof of Proposition \ref{prop:main} in Subsection \ref{sec:proof_proposition_main_sketch}. The technical details involving the asymptotic analysis are postponed to Subsection \ref{sec:asympt}.

In Section \ref{sec:proof_mainthm}, we prove our main result, Theorem \ref{thm:main}, using Proposition \ref{prop:main} and some properties of the $\FT$ function proved in Subsection \ref{sec:Properties_FT}.

Finally, in Section \ref{sec:crossover} and Section \ref{sec:proof_property}  we prove the properties of the fields $\cH$ and $\cH_0$ listed in Proposition \ref{prop:crossover} and Proposition \ref{prop:properties_cH} respectively.

\section*{Acknowlegements}

The authors would like to thank Duncan Dauvergne, Daniel Remenik, Tejaswi Tripathi, and Yizao Wang for their comments and suggestions. Both authors were partially supported by NSF DMS-1953687 and DMS-2246683.

\section{The function $\FT$ and the upper tail field $\cH$}
\label{sec:def_scrH}

\subsection{Definition of $\FT$}
\label{sec:def_FT}

The function $\FT(\bbeta;(\alpha_1,\tau_1),\cdots,(\alpha_m,\tau_m))$ is defined for $\bbeta=(\beta_1,\cdots,\beta_m)\in\realR^m$ and $(\alpha_1,\tau_1),\cdots (\alpha_m,\tau_m)\in\realR^2$ satisfying $(\alpha_1,\tau_1) \prec \cdots \prec(\alpha_m,\tau_m)$, where the order $\prec$ was defined in Definition \ref{defn:prec}.

When $m\ge 2$, the definition involves a sum of contour integrals. There are $4(m-1)$ contours appearing in the definition. Let 
\begin{equation}
\Gamma_{m,\LL}^{\inn},\cdots,\Gamma_{2,\LL}^{\inn}, \Gamma_{2,\LL}^{\out},\cdots,\Gamma_{m,\LL}^\out
\end{equation}
be $2m-2$ contours, ordered from left to right, on the left half plane $\{u: \Re(u)<0\}$, each of which goes from $\infty e^{-\ii 2\pi/3}$ to $\infty e^{\ii 2\pi/3}$. Moreover, the point $-1$ lies between the two contours $\Gamma_{2,\LL}^\inn$ and $\Gamma_{2,\LL}^\out$. We similarly let 
\begin{equation}
\Gamma_{m,\RR}^{\inn},\cdots,\Gamma_{2,\RR}^{\inn},\Gamma_{2,\RR}^{\out},\cdots,\Gamma_{m,\RR}^\out
\end{equation}
be $2m-2$ contours, ordered from right to left, on the right half plane $\{v: \Re(v)>0\}$, each of which goes from $\infty e^{-\ii  \pi/5 }$ to $\infty e^{\ii  \pi/5}$. Moreover, the point $1$ lies between the two contours $\Gamma_{2,\RR}^\inn$ and $\Gamma_{2,\RR}^\out$. The symbols $\LL$ and $\RR$ appearing in the subscripts of these contours indicate whether the contour lies on the left half plane or the right half plane. The symbols $\inn$ and $\out$ appearing in the superscripts indicate the relative locations of these contours to infinity ($-\infty$ for the contours on the left half plane and $\infty$ for those on the right half plane). The angles of these contours are chosen to guarantee the convergence of the integrals along these contours in the definition. Especially, we would like to point out that when the times are strictly ordered $\tau_1<\tau_2<\cdots<\tau_m$, the angles of the right contours could be chosen to be the $\pm\pi/3$. However, to ensure the formula is still valid for the case of possible points with an equal time, we need to bend the contours on one side (the side depends on how one chooses the order of the spatial parameters). See \cite{Liu19} for more discussions on this issue.

See Figure \ref{fig:contours1} for an illustration of the contours when $m=3$.

\begin{figure}
    \centering
    \begin{tikzpicture}[scale=1.5, decoration={
  markings,
  mark=at position 0.75 with {\arrow{>}}}
  ] 
  \draw[-latex] (-3,0) -- (3,0) node[right] {$\Re$};
  \draw[-latex] (0,-2) -- (0,2) node[above] {$\Im$};

  \draw[postaction={decorate},blue,thick] plot[smooth] coordinates {(-2.5,-1) (-2,-0.5) (-1.5,0) (-2,0.5) (-2.5,1)};
  \draw[postaction={decorate},blue,thick] plot[smooth] coordinates {(-2,-1) (-1.5,-0.5) (-1,0) (-1.5,0.5) (-2,1)};
  \draw[postaction={decorate},blue,thick] plot[smooth] coordinates {(-1.5,-1) (-1,-0.5) (-0.5,0) (-1,0.5) (-1.5,1)};
  \draw[postaction={decorate},blue,thick] plot[smooth] coordinates {(-1,-1) (-0.5,-0.5) (-0.1,0) (-0.5,0.5) (-1,1)};

  \draw[postaction={decorate},red,thick] plot[smooth] coordinates {(1,-0.5) (0.5,-0.25) (0.1,0) (0.5,0.25) (1,0.5)};
  \draw[postaction={decorate},red,thick] plot[smooth] coordinates {(1.5,-0.5) (1,-0.25) (0.5,0) (1,0.25) (1.5,0.5)};
  \draw[postaction={decorate},red,thick] plot[smooth] coordinates {(2,-0.5) (1.5,-0.25) (1,0) (1.5,0.25) (2,0.5)};
  \draw[postaction={decorate},red,thick] plot[smooth] coordinates {(2.5,-0.5) (2,-0.25) (1.5,0) (2,0.25) (2.5,0.5)};

  \filldraw (-0.75,0) circle (0.5pt) node[below] {$-1$};
  \filldraw (0.75,0) circle (0.5pt) node[below] {$1$};

  \node at (-2.8,1.2) {$\Gamma_{3,\LL}^{\inn}$};
  \node at (-2.1,1.2) {$\Gamma_{2,\LL}^{\inn}$};
  \node at (-1.4,1.2) {$\Gamma_{2,\LL}^\out$};
  \node at (-0.7,1.2) {$\Gamma_{3,\LL}^\out$};

  \node at (2.8,0.6) {$\Gamma_{3,\RR}^{\inn}$};
  \node at (2.1,0.6) {$\Gamma_{2,\RR}^{\inn}$};
  \node at (1.4,0.6) {$\Gamma_{2,\RR}^\out$};
  \node at (0.7,0.6) {$\Gamma_{3,\RR}^\out$};
\end{tikzpicture}
    \caption{Illustration of the $\Gamma$-contours in the definition of $\FT$ when $m=3$.}
    \label{fig:contours1}
\end{figure}

We also introduce some notations below.

Assume $W=(w_1,\cdots,w_n)\in\complexC^n$ is a vector, we denote
\begin{equation}
aW+b= (aw_1+b,\cdots,aw_n+b)\in\complexC^n
\end{equation}
for any $a,b\in\complexC$. We also denote the concatenation of two vectors $W=(w_1,\cdots,w_n)\in\complexC^n$ and $W'=(w'_1,\cdots,w'_{n'})\in\complexC^{n'}$
\begin{equation}
W\bunion W':= (w_1,\cdots,w_n,w'_1,\cdots,w'_{n'}) \in \complexC^{n+n'}.
\end{equation}
Especially, when $n'=0$ we write $W\bunion\emptyset=W$, and when $n'=1$, we write, for $w'\in\complexC$,
\begin{equation}
W\bunion w'=(w_1,\cdots,w_n,w'),\quad \text{and}\quad 
w'\bunion W=(w',w_1,\cdots,w_n).
\end{equation}

If $W=(w_1,\cdots,w_n)\in \complexC^n$ and $\tilde W=(\tilde w_1,\cdots, \tilde w_n)\in \complexC^n$ satisfying $w_i\ne \tilde w_j$ for all $1\le i,j\le n$, denote the Cauchy determinant
\begin{equation}
\label{eq:def_rC}
\rC(W; \tilde W) = \det\left[ \frac{1}{w_i- \tilde w_j}\right]_{i,j=1}^n = (-1)^{n(n-1)/2}\frac{\prod_{1\le i<j\le n}(w_j-w_j)(\tilde w_j-\tilde w_i)}{\prod_{1\le i,j\le n}(w_i-\tilde w_j)}.
\end{equation}
Note that the dimensions of $W$ and $\tilde W$ need to match in the above Cauchy determinant. A simple calculation also implies
\begin{equation}
\label{eq:id_rC}
\begin{split}
\rC(W\bunion W';\tilde W\bunion\tilde W')
&= \det  
         \begin{bmatrix}
         & \vdots & &\vdots&\\
         \cdots& \frac{1}{w_i-\tilde w_j}& \cdots& \frac{1}{w_i -\tilde w'_{j'}} & \cdots\\
         & \vdots & &\vdots&\\
         \cdots& \frac{1}{w'_{i'}-\tilde w_j}& \cdots& \frac{1}{w'_{i'} -\tilde w'_{j'}} & \cdots\\
                  & \vdots & &\vdots&
         \end{bmatrix}_{\substack{1\le i,j\le n\\ 1\le i',j'\le n'}}\\
         &=  \rC(W;\tilde W)\rC(W';\tilde W')\cdot \prod_{i=1}^n \prod_{i'=1}^{n'} \frac{(w_i-w'_{i'})(\tilde w_i-\tilde w'_{i'})}{(w_i-\tilde w'_{i'}) (\tilde w_i- w'_{i'})}
\end{split}
\end{equation}
for any $W=(1_1,\cdots,w_n), \tilde W=(\tilde w_1,\cdots,\tilde w_n) \in \complexC^n$, $W'=(w'_1,\cdots,w'_{n'}), \tilde W'=(\tilde w'_1,\cdots,\tilde w'_{n'}) \in \complexC^{n'}$ satisfying $w_i\ne \tilde w_j$, $w'_{i'}\ne \tilde w_j$, $w_i\ne \tilde w'_{j'}$, and $w'_{i'}\ne \tilde w'_{j'}$ for all $1\le i, j\le n$ and $1\le i',j'\le n'$.

For the Cauchy determinant, we have a very simple inequality, which we state in Lemma \ref{lm:bounds_Cauchy_det} below. Before we state the inequality, we introduce a notation. For two sets $A,B\subset\complexC$, denote
\begin{equation}
\dist(A;B):=\inf\{|a-b|: a\in A, b\in B\}.
\end{equation}
For simplification we also use the same notation for the distance of the coordinates in two vectors
\begin{equation}
\dist(W;W')=\min \{|w_i-w'_{i'}|: 1\le i\le n, 1\le i'\le n'\}
\end{equation}
for any vectors $W=(w_1,\cdots,w_n)\in\complexC^n$ and $W'=(w'_1,\cdots,w'_{n'})\in \complexC^{n'}$. Note this is not the usual distance of two vectors. For example, this function $\dist(W,W')$ is invariant under permutations of the coordinates of $W$ or $W'$.

\begin{lm}
\label{lm:bounds_Cauchy_det}
Assume the two vectors $W=(w_1,\cdots,w_n), \tilde W =(\tilde w_1,\cdots,\tilde w_n)\in\complexC^n$ satisfy $\dist(W;\tilde W)>0$, then
\begin{equation}
|\rC(W;W')| \le n^{n/2}  \dist(W;\tilde W)^{-n}.
\end{equation}
If we further have $W'=(w'_1,\cdots,w'_{n'}), \tilde W'=(\tilde w'_1,\cdots,\tilde w'_{n'})\in \complexC^{n'}$ satisfying $\dist(W\bunion W';\tilde W\bunion \tilde W')>0$, then
\begin{equation}
|\rC(W\bunion W';\tilde W\bunion \tilde W')| \le n^{n/2}(n')^{n'/2} 2^{(n+n')/2}\dist(W\bunion W';\tilde W\bunion \tilde W')^{-n-n'}.
\end{equation}
\end{lm}
\begin{proof}[Proof of Lemma \ref{lm:bounds_Cauchy_det}]
The first inequality follows from the Hadamard's inequality
\begin{equation}
|\rC(W;W')|= \left|\det\left( \frac{1}{w_i-\tilde w_j}\right)_{i,j=1}^n\right| \le \prod_{j=1}^n \sqrt{\sum_{i=1}^n |w_i-\tilde w_j|^{-2}} \le \left( n \cdot \dist(W;\tilde W)^{-2}\right)^{n/2}.
\end{equation}
The second inequality follows from the first inequality and the following simple inequality
\begin{equation}
\left(\frac{n+n'}{2}\right)^{(n+n')/2} \le n^{n/2} (n')^{n'/2}
\end{equation}
since the function $x\ln x$ is a convex function on $(0,\infty)$.
\end{proof}

\bigskip

Finally, for fixed $\bbeta=(\beta_1,\cdots,\beta_m)$ and $(\alpha_\ell,\tau_\ell)$, $1\le \ell \le m$, we introduce the function
\begin{equation}
\label{eq:def_f}
f_\ell(w)=f_\ell(w;\bbeta;(\alpha_1,\tau_1),\cdots,(\alpha_m,\tau_m)):= 
\begin{dcases}
e^{-\frac{1}{3}\tau_{1}w^3 + \alpha_{1}w^2 + \beta_{1} w}, & \ell=1,\\
e^{-\frac{1}{3}(\tau_{\ell}-\tau_{\ell-1})w^3 + (\alpha_{\ell}-\alpha_{\ell-1})w^2 +(\beta_{\ell}-\beta_{\ell-1})w}, &2\le \ell \le m.
\end{dcases}
\end{equation}
It is direct to see that if $\ell\ge 2$ and $(\alpha_{\ell-1},\tau_{\ell-1}) \prec (\alpha_{\ell},\tau_{\ell})$, then $f_\ell(w)$ decays super-exponentially fast to $0$ as $w\to\infty$ along the directions $e^{\pm \ii 2\pi/3}$. In fact, the real part of the exponent of $f_\ell$ is approximately $-\frac13(\tau_{\ell}-\tau_{\ell-1}) \Re(w^3)$ if $\tau_{\ell-1}<\tau_{\ell}$, or $(\alpha_{\ell}-\alpha_{\ell-1}) \Re(w^2)$ if $\tau_{\ell-1}=\tau_{\ell}$ and $\alpha_{\ell-1}<\alpha_{\ell}$ when $w$ grows to infinity along the directions $e^{\pm \ii 2\pi/3}$. In both cases, the real part goes to $-\infty$ (at a rate of $|w|^3$ or $|w|^2$). Thus $f_\ell(w)$ decays super-exponentially fast to $0$ along these two directions. Similarly, when $\ell\ge 2$ the function $f_\ell(w)$ grows super-exponentially fast to $\infty$ as $w\to\infty$ along the directions $e^{\pm \ii  \pi/5}$, i.e., $1/f_\ell(w)$ decays super-exponentially fast to $0$ along the directions $e^{\pm \ii  \pi/5}$.

\bigskip
Now we are ready to define $\FT(\bbeta;(\alpha_1,\tau_1),\cdots,(\alpha_m,\tau_m))$.

\begin{defn}[Definition of $\FT$]
\label{defn:FT}

Assume that $\bbeta=(\beta_1,\cdots,\beta_m)\in\realR^m$, and the $m$ points $(\alpha_\ell,\tau_\ell)\in\realR^2$, $1\le \ell\le m$, satisfy $(\alpha_1,\tau_1)\prec \cdots \prec (\alpha_m,\tau_m)$. $\bz=(z_1,\cdots,z_{m-1})$ is a vector in $(\complexC\setminus\{0,1\})^{m-1}$ if $m\ge 2$.
\begin{enumerate}[(i)]
\item If $m=1$, we define
\begin{equation}
\label{eq:def_FT_1}
\FT(\beta_1;(\alpha_1,\tau_1)) = e^{\frac23\tau_1-2\beta_1}.
\end{equation}
\item If $m\ge 2$, we define
\begin{equation}
    \begin{split}
\label{eq:def_FT_2}
&\FT(\bbeta;(\alpha_1,\tau_1),\cdots,(\alpha_m,\tau_m))\\
& = (-1)^m \oint_{>1}\cdots \oint_{>1} \sum_{\substack{n_\ell \ge 1\\ 2\le \ell\le m}} \frac{1}{ (n_2!\cdots n_{m-1}!)^2} \cD_{\bn}(\bbeta;\bz) \prod_{\ell=1}^{m-1}\frac{\rd z_\ell}{2\pi\ii z_\ell(1-z_\ell)}
    \end{split}
\end{equation}
where $\bn=(n_1=1,n_2,\cdots,n_m)$, $\bz=(z_1,\cdots,z_{m-1})$, $\oint_{>1}$ denotes the integral around a counterclockwise oriented circle with radius larger than $1$ and centered at the origin, 
\begin{equation}
\label{eq:def_cdnz}
        \begin{split}
             \cD_{\bn}(\bbeta;\bz)&= \cD_{\bn}(\bbeta;\bz;(\alpha_1,\tau_1),\cdots,(\alpha_m,\tau_m))\\
             &=2\prod_{\ell=1}^{m-1}(1-z_{\ell})^{n_{\ell}}(1-z_{\ell}^{-1})^{n_{\ell+1}}
             \cdot \prod_{\ell=2}^{m} \prod_{i_\ell=1}^{n_\ell} \left(\frac{1}{1-z_{\ell-1}}\int_{\Gamma_{\ell,\LL}^\inn}\ddbar{u_{i_\ell}^{(\ell)}}{}-\frac{z_{\ell-1}}{1-z_{\ell-1}}\int_{\Gamma_{\ell,\LL}^\out}\ddbar{u_{i_\ell}^{(\ell)}}{}\right)\\
             &\quad \cdot \prod_{\ell=2}^{m} \prod_{i_\ell=1}^{n_\ell} \left(\frac{1}{1-z_{\ell-1}}\int_{\Gamma_{\ell,\RR}^\inn}\ddbar{v_{i_\ell}^{(\ell)}}{}-\frac{z_{\ell-1}}{1-z_{\ell-1}}\int_{\Gamma_{\ell,\RR}^\out}\ddbar{v_{i_\ell}^{(\ell)}}{}\right) \prod_{\ell=2}^{m} \prod_{i_\ell=1}^{n_\ell} \frac{f_{\ell}(u_{i_\ell}^{(\ell)})}{f_{\ell}(v_{i_\ell}^{(\ell)})} \cdot \frac{f_1(-1)}{f_1(1)}\\
        & \quad \cdot \rC(-1\bunion V^{(2)}; 1\bunion U^{(2)}) 
        \cdot \prod_{\ell=2}^{m-1} \rC(U^{(\ell)}\bunion V^{(\ell+1)}; V^{(\ell)}\bunion U^{(\ell+1)})
        \cdot \rC(U^{(m)};V^{(m)})
        \end{split}
    \end{equation}
    and the vectors $U^{(\ell)}=(u_{1}^{(\ell)},\cdots,u_{n_\ell}^{(\ell)})$, $V^{(\ell)}=(v_{1}^{(\ell)},\cdots,v_{n_\ell}^{(\ell)})$ for each $2\le \ell \le m$. The functions $f_\ell$, $1\le \ell\le m$, are defined in \eqref{eq:def_f}.
\end{enumerate}
\end{defn}

We first note that the case when $m=1$, the function $\FT(\beta_1;(\alpha_1,\tau_1))$ equals to $-2\rC(-1;1)f_1(-1)/f_1(1)$ which can be viewed as the degenerated form of \eqref{eq:def_FT_2} if we set all the empty product to be $1$ and set $U^{(2)}=V^{(2)}=\emptyset$ in \eqref{eq:def_cdnz}.

It is also easy to see that  when $m\ge 2$, all the terms $\cD_{\bn}(\bbeta;\bz)$ are well defined since the integrand decays super-exponentially fast along the integration contours due to the factors of the $f_\ell$ functions, see the discussions after \eqref{eq:def_f}. 

We also need to verify that the integrals and summations in \eqref{eq:def_FT_2} are well defined. In fact, note that \eqref{eq:def_cdnz} is independent of the specific choices of the $\Gamma$-contours (as long as they satisfy the nesting assumption and have the desired angles approaching to infinity described at the beginning of this subsection). We could select these integral contours such that the distance between any two contours and the distance between any contour to $\{1,-1\}$ are at least $d>0$. Then we apply Lemma \ref{lm:bounds_Cauchy_det} and obtain
\begin{equation}
    \begin{split}
        & \left|\cD_{\bn}(\bz;\bbeta)\right|\\
        & \le 2 \left| \frac{f_1(-1)}{f_1(1)}\right| \cdot \prod_{\ell=2}^{m} \frac{(1+|z_{\ell-1}|)^{2n_\ell}}{|z_{\ell-1}|^{n_\ell} |1-z_{\ell-1}|^{n_\ell -n_{\ell-1}}} \cdot 2^{\frac12+\sum_{\ell=2}^{m}n_\ell} \cdot \frac{\prod_{\ell=2}^{m} n_\ell^{n_\ell}}{d^{1+2\sum_{\ell=2}^{m}} n_\ell}\\
        &\quad  \cdot \prod_{\ell=2}^{m} \left| \int_{\Gamma_{\ell,\LL}^\inn\cup \Gamma_{\ell,\LL}^\out} |f_{\ell+1}(u)| \frac{|\rd u|}{2\pi} \int_{\Gamma_{\ell,\RR}^\inn\cup \Gamma_{\ell,\RR}^\out} \frac1{|f_{\ell+1}(v)|} \frac{|\rd v|}{2\pi}\right|^{n_\ell}
        \\
        &\le \prod_{\ell=2}^{m} \frac{(1+|z_{\ell-1}|)^{2n_\ell}}{|z_{\ell-1}|^{n_\ell} |1-z_{\ell-1}|^{n_\ell -n_{\ell-1}}}  C^{\sum_{\ell=2}^{m}n_\ell} \prod_{\ell=2}^{m}n_\ell^{n_\ell}
    \end{split}
\end{equation}
when $m\ge 2$, where $C$ is a constant independent of $\bn$ and $\bz$, and we used the fact that the functions $f_\ell(w)$ decays super-exponentially fast along the left contours and grows super-exponentially fast along the right contours as $w\to\infty$ and hence each integral in the middle of the above inequality is finite. Thus, we see that the right hand sides of \eqref{eq:def_FT_2} is absolutely convergent, and the function $\FT$ is well defined. 

\bigskip

Finally, we remark that the $z_1$ integral in \eqref{eq:def_FT_2} actually can be evaluated explicitly. In fact, noting \eqref{eq:def_cdnz}, the $z_1$ integral is (after we suppress other terms independent of $z_1$ in the $\cdots$ part)
\begin{equation}
    \begin{split}
    &\oint_{>1} (1-z_1)(1-z_1^{-1})^{n_2} \prod_{i_2=1}^{n_2} \left(\frac{1}{1-z_{1}}\int_{\Gamma_{2,\LL}^\inn}\ddbar{u_{i_2}^{(2)}}{}-\frac{z_{1}}{1-z_{1}}\int_{\Gamma_{2,\LL}^\out}\ddbar{u_{i_2}^{(2)}}{}\right)\\
    &\quad \prod_{i_2=1}^{n_2}
    \left(\frac{1}{1-z_{1}}\int_{\Gamma_{2,\RR}^\inn}\ddbar{u_{i_2}^{(2)}}{}-\frac{z_{1}}{1-z_{1}}\int_{\Gamma_{2,\RR}^\out}\ddbar{v_{i_2}^{(2)}}{}\right) \cdots \frac{\rd z_1}{2\pi \ii z_1(1-z_1)}.
\end{split}
\end{equation}
If we deform the $z_1$ contour to infinity, we see that any term involving $\int_{\Gamma_{2,\LL}^\inn}$ or $\int_{\Gamma_{2,\RR}^\inn}$ vanishes, and we end with, after integrating $z_1$,
\begin{equation}
    \prod_{i_2=1}^{n_2} \int_{\Gamma_{2,\LL}^\out}\ddbar{u_{i_2}^{(2)}}{}\int_{\Gamma_{2,\RR}^\out}\ddbar{v_{i_2}^{(2)}}{} \cdots
\end{equation}
Therefore, we obtain the following result.
\begin{prop}
    \label{prop:FT_alt}
    When $m\ge 2$, we have the following formula for $\FT(\bbeta;(\alpha_1,\tau_1),\cdots,(\alpha_m,\tau_m))$.
\begin{equation}
    \begin{split}
\label{eq:def_FT_alt}
&\FT(\bbeta;(\alpha_1,\tau_1),\cdots,(\alpha_m,\tau_m))
 = (-1)^m \oint_{>1}\cdots \oint_{>1} \sum_{\substack{n_\ell \ge 1\\ 2\le \ell\le m}} \frac{1}{ (n_2!\cdots n_{m-1}!)^2} \tilde\cD_{\bn}(\bbeta;\tilde\bz) \prod_{\ell=2}^{m-1}\frac{\rd z_\ell}{2\pi\ii z_\ell(1-z_\ell)},
    \end{split}
\end{equation}
where $\tilde\bz=(z_2,\cdots,z_{m-1})$, and
\begin{equation}
    \label{eq:def_cdnz_alt}
            \begin{split}
                 &\tilde\cD_{\bn}(\bbeta;\tilde\bz)\\
                 &=2\prod_{\ell=2}^{m-1}(1-z_{\ell})^{n_{\ell}}(1-z_{\ell}^{-1})^{n_{\ell+1}}
                 \cdot \prod_{\ell=3}^{m} \prod_{i_\ell=1}^{n_\ell} \left(\frac{1}{1-z_{\ell-1}}\int_{\Gamma_{\ell,\LL}^\inn}\ddbar{u_{i_\ell}^{(\ell)}}{}-\frac{z_{\ell-1}}{1-z_{\ell-1}}\int_{\Gamma_{\ell,\LL}^\out}\ddbar{u_{i_\ell}^{(\ell)}}{}\right)\prod_{i_2=1}^{n_2}\int_{\Gamma_{2,\LL}^{\out}}\ddbar{u_{i_2}^{(2)}}{}\\
                 &\quad  \prod_{\ell=3}^{m} \prod_{i_\ell=1}^{n_\ell} \left(\frac{1}{1-z_{\ell-1}}\int_{\Gamma_{\ell,\RR}^\inn}\ddbar{v_{i_\ell}^{(\ell)}}{}-\frac{z_{\ell-1}}{1-z_{\ell-1}}\int_{\Gamma_{\ell,\RR}^\out}\ddbar{v_{i_\ell}^{(\ell)}}{}\right)\prod_{i_2=1}^{n_2}\int_{\Gamma_{2,\RR}^{\out}}\ddbar{v_{i_2}^{(2)}}{}\cdot \prod_{\ell=2}^{m} \prod_{i_\ell=1}^{n_\ell} \frac{f_{\ell}(u_{i_\ell}^{(\ell)})}{f_{\ell}(v_{i_\ell}^{(\ell)})} \cdot \frac{f_1(-1)}{f_1(1)}\\ 
            & \quad \cdot \rC(-1\bunion V^{(2)}; 1\bunion U^{(2)}) 
            \cdot \prod_{\ell=2}^{m-1} \rC(U^{(\ell)}\bunion V^{(\ell+1)}; V^{(\ell)}\bunion U^{(\ell+1)})
            \cdot \rC(U^{(m)};V^{(m)}).
            \end{split}
\end{equation}
\end{prop}

We could have used this formula as the definition of $\FT$ function. The reason we prefer the formula \eqref{eq:def_FT_2} in the definition is that the structure of \eqref{eq:def_FT_2} is more symmetric and aligns better with the pre-limit formula of the KPZ fixed point.

\subsection{Definition of the upper tail field $\cH$}
\label{sec:def_cH}

Now we are ready to define the upper tail field $\cH$ using the Kolmogorov extension theorem and its joint tail probability functions. We denote
\begin{equation}
    \label{eq:def_hat_FT}
    \hat\FT(\beta_1,\cdots,\beta_m; (\alpha_1,\tau_1),\cdots,(\alpha_m,\tau_m))= \prob\left( \bigcap_{\ell=1}^m \left\{\cH(\alpha_\ell,\tau_\ell) \ge \beta_\ell \right\} \right)
\end{equation}
for any $m\ge 1$, and $m$ distinct points $(\alpha_1,\tau_1),\cdots,(\alpha_m,\tau_m)\in\realR^2$ and $\beta_1,\cdots,\beta_m\in\realR$. The function $\hat\FT$ is explicitly given below.

\begin{defn}
    \label{def:cH_by_tail}
Let $m\ge 1$. Assume $(\alpha_1,\tau_1),\cdots,(\alpha_m,\tau_m)$ are $m$ distinct points on $\realR^2$, and $\beta_1,\cdots,\beta_m\in\realR$. 
\begin{enumerate}[(i)]
    \item If $(\alpha_1,\tau_1)\prec (\alpha_2,\tau_2)\prec\cdots \prec (\alpha_m,\tau_m)$, and $(\alpha_k,\tau_k)=(0,0)$ for some $1\le k\le m$, then
        \begin{equation}
            \begin{split}
            &\hat\FT(\beta_1,\cdots,\beta_m; (\alpha_1,\tau_1),\cdots,(\alpha_m,\tau_m))\\
            &=\FT(\beta_1,\cdots,\beta_{k-1},\max\{\beta_k,0\},\beta_{k+1},\cdots,\beta_m; (\alpha_1,\tau_1),\cdots,(\alpha_m,\tau_m)).
            \end{split}
        \end{equation}
    \item If $(\alpha_1,\tau_1)\prec \cdots \prec (\alpha_{k-1},\tau_{k-1})\prec(0,0)\prec (\alpha_{k},\tau_{k})\prec\cdots \prec (\alpha_m,\tau_m)$ for some $1\le k\le m+1$, then
        \begin{equation}
            \begin{split}
            &\hat\FT(\beta_1,\cdots,\beta_m; (\alpha_1,\tau_1),\cdots,(\alpha_m,\tau_m))\\
            &=\FT(\beta_1,\cdots,\beta_{k-1},0,\beta_{k},\cdots,\beta_m; (\alpha_1,\tau_1),\cdots,(\alpha_{k-1},\tau_{k-1}),(0,0),(\alpha_k,\tau_k),\cdots,(\alpha_m,\tau_m)).
            \end{split}
        \end{equation}
        \item More generally, assume that the points $(\alpha_1,\tau_1),\cdots,(\alpha_m,\tau_m)$ are not necessarily ordered, and $\sigma\in S_m$ is the unique permutation such that $(\alpha_{\sigma_1},\tau_{\sigma_1})\prec \cdots \prec (\alpha_{\sigma_m},\tau_{\sigma_m})$, then
         \begin{equation}
            \hat\FT(\beta_1,\cdots,\beta_m; (\alpha_1,\tau_1),\cdots,(\alpha_m,\tau_m))=\hat\FT(\beta_{\sigma_1},\cdots,\beta_{\sigma_m}; (\alpha_{\sigma_1},\tau_{\sigma_1}),\cdots,(\alpha_{\sigma_m},\tau_{\sigma_m}))
         \end{equation} 
         where the right hand side is defined in the previous two cases.
 \end{enumerate}
\end{defn}

Recall that $\prec$ is a total order defined in Definition \ref{defn:prec}. Thus, Definition \ref{def:cH_by_tail} gives the joint tail probability functions of $\cH$ for arbitrarily ordered points. We also note a special case when $m=1$ and $(\alpha_1,\tau_1)=(0,0)$, we have (recalling Definition \ref{defn:FT})
\begin{equation}
    \label{eq:cH_00}
    \hat\FT(\beta;(0,0))=\prob\left( \cH(0,0)\ge \beta\right) = e^{-2\max\{\beta,0\}}.
\end{equation}
Moreover, $\hat\FT$ has the following property by the equation \eqref{eq:aux_12_31_01} which will be proved later,
\begin{equation}
\label{eq:aux_12_31_02}
\hat\FT(\beta;(\alpha,\tau)) = e^{\frac{2}{3}\tau -2\beta}\hat\FT(-\beta;(-\alpha,-\tau)),\quad \text{ if } (0,0)\prec (\alpha,\tau).
\end{equation}

We still need to verify that \eqref{eq:def_hat_FT} actually defines a random field by checking the joint tail probability functions satisfy the consistency conditions so that the Kolmogorov extension theorem applies. Note that the Kolmogorov extension theorem works for the joint tail probability functions in the same way as for the joint cumulative distribution functions, since the joint tail probability functions of the field $\cH$ are the same as the joint cumulative distribution functions of $-\cH$. The consistency of the joint probability tail functions is proved in the following proposition. We note that the proof relies on Proposition \ref{prop:main} and the first part of Theorem \ref{thm:main}. They imply (also see \eqref{eq:need_to_show})
\begin{equation}
    \label{eq:cond_limit}
    \lim_{L\to\infty} \prob \left( \bigcap_{\ell=1}^m \left\{ \HH(\alpha_\ell,\tau_\ell)\ge \beta_\ell \right\} \mid \HH(0,0)\ge 0 \right) = \hat\FT(\beta_1,\cdots,\beta_m;(\alpha_1,\tau_1),\cdots,(\alpha_m,\tau_m))
\end{equation}
    for any $\beta_k\in\realR$, $1\le k\le m$, as long as $(\alpha_k,\tau_k)$, $1\le k\le m$, are distinct points on $\realR^2$, where $\HH$ is the rescaled KPZ fixed point $\HH(\alpha,\tau):= \sqrt{L} \left(\rH(\alpha L^{-1}, 1+ \tau  L^{-3/2}) - L\right)$
    as defined in \eqref{eq:def_HH}. We emphasize that the proofs of Proposition \ref{prop:main} and the first part of Theorem \ref{thm:main}, including \eqref{eq:cond_limit}, does not depend on the existence of the random field $\cH$ but only uses the explicit formulas of $\HH$ and $\FT$.

\begin{prop}
    \label{prop:consistency}
    The function $\hat\FT$ has the following properties.
    \begin{enumerate}[(a)]
    \item $\hat\FT(\beta;(\alpha,\tau))$ is a tail probability function for each $(\alpha,\tau)\in\realR^2$, i.e., 
    \begin{equation}
    \lim_{\beta\to-\infty}\hat\FT(\beta;(\alpha,\tau))=1,\quad \lim_{\beta\to\infty}\hat\FT(\beta;(\alpha,\tau))=0.
    \end{equation}
    \item For $m\ge 2$, we have
    \begin{equation}
    \label{eq:consistency_right}
    \lim_{\beta_k\to \infty}\hat\FT(\beta_1,\cdots,\beta_m;(\alpha_1,\tau_1),\cdots,(\alpha_m,\tau_m))=0
    \end{equation}
    and
    \begin{equation}
    \label{eq:consistency_left}
    \begin{split}
    &\lim_{\beta_k\to-\infty}\hat\FT(\beta_1,\cdots,\beta_m;(\alpha_1,\tau_1),\cdots,(\alpha_m,\tau_m))\\
    &= \hat\FT(\beta_1,\cdots,\beta_{k-1},\beta_{k+1},\cdots,\beta_m;(\alpha_1,\tau_1),\cdots,(\alpha_{k-1},\tau_{k-1}),(\alpha_{k+1},\tau_{k+1}),\cdots,(\alpha_m,\tau_m))
    \end{split}
    \end{equation}
    for all $1\le k\le m$.
    \end{enumerate}
    \end{prop}
    
    \begin{proof}[Proof of Proposition \ref{prop:consistency}]
    
    We first prove (b) using (a), then prove (a).
    
    Assume Proposition \ref{prop:consistency} (a) holds. Note that \eqref{eq:cond_limit} implies
    \begin{equation}
    \hat\FT(\beta_1,\cdots,\beta_m;(\alpha_1,\tau_1),\cdots,(\alpha_m,\tau_m))\le \hat\FT(\beta_k;(\alpha_k,\tau_k)).
    \end{equation}
    Combining with (a) we immediately obtain \eqref{eq:consistency_right}.
    
    On the other hand, \eqref{eq:cond_limit} also implies
    \begin{equation}
    \begin{split}
    0\le &\hat\FT(\beta_1,\cdots,\beta_{k-1},\beta_{k+1},\cdots,\beta_m;(\alpha_1,\tau_1),\cdots,(\alpha_{k-1},\tau_{k-1}),(\alpha_{k-1},\tau_{k-1}),\cdots,(\alpha_m,\tau_m))\\
    &-\hat\FT(\beta_1,\cdots,\beta_m;(\alpha_1,\tau_1),\cdots,(\alpha_m,\tau_m))\\
    =&\lim_{L\to\infty} \prob \left( \bigcap_{\ell\ne k} \left\{ \HH(\alpha_\ell,\tau_\ell)\ge \beta_\ell\right\} \bigcap \left\{\HH(\alpha_k,\tau_k)< \beta_k\right\} \mid \HH(0,0)\ge 0\right)\\
    \le& \lim_{L\to\infty} \prob \left( \HH(\alpha_k,\tau_k)< \beta_k\mid \HH(0,0)\ge 0\right)\\
    =&1-\hat\FT(\beta_k;(\alpha_k,\tau_k)).
    \end{split}
    \end{equation}
    Combining with (a) we have \eqref{eq:consistency_left}.

    \bigskip
    
    It remains to show Proposition \ref{prop:consistency} (a). The large $\beta$ limit is easier. We use \eqref{eq:cond_limit} and obtain
    \begin{equation}
    \hat\FT(\beta;(\alpha,\tau))= \lim_{L\to\infty} \prob \left(\HH(\alpha,\tau)\ge \beta\mid \HH(0,0)\ge 0\right) 
    \le \lim_{L\to\infty} \frac{\prob\left(\HH(\alpha,\tau)\ge \beta\right)}{\prob\left(\HH(0,0)\ge 0\right)}
    =e^{\frac{2}{3}\tau -2\beta}
    \end{equation}
    where we used the asymptotics \eqref{eq:tail} in the last equation. The above inequality implies that
    \begin{equation}
    \lim_{\beta\to\infty}\hat\FT(\beta;(\alpha,\tau))=0.
    \end{equation}
    
    The other limit when $\beta\to-\infty$ is more complicated. We consider four different cases.
    
    Case 1: $\alpha=\tau=0$. In this case by \eqref{eq:cH_00} we have $\hat\FT(\beta;(0,0))=e^{-2\max\{\beta,0\}}$ which is $1$ when $\beta<0$.
    
    Case 2: $\tau>0$. We need an inequality in the notation of the directed landscape \cite{DOV}. The directed landscape $\mathcal{L}(y,s;x,t)$, $0<s<t$, $x,y\in\realR$, is a ``random metric'' which has the following relation to the KPZ fixed point $\mathcal{L}(0,0;x,t) \stackrel{d}{=} \rH(x,t)$, where $\stackrel{d}{=}$ means an equation in law. It also has the following triangle inequality
    \begin{equation}
    \mathcal{L}(0,0;x,t)\ge \mathcal{L}(0,0;y,s)+\mathcal{L}(y,s;x,t)
    \end{equation}
    for any $s<t$ and $x,y\in\realR$. Moreover, $\mathcal{L}(0,0;y,s)$, and $\mathcal{L}(y,s;x,t)$ are independent.  Finally, the one point distribution of $\mathcal{L}(y,s;x,t)$ is given by the GUE Tracy-Widom distribution
    \begin{equation}
    \prob\left( \mathcal{L}(y,s;x,t)\le h \right) = \FGUE\left( \frac{h}{(t-s)^{1/3}} +\frac{(x-y)^2}{(t-s)^{4/3}}\right).
    \end{equation}
    
    Using the language of the directed landscape, we have, for any $\tau_1<\tau_2$,
    \begin{equation}
    \label{eq:aux_11_30}
    \begin{split}
    &\prob\left(\HH(\alpha_1,\tau_1) \ge \beta_1, \HH(\alpha_2,\tau_2) < \beta_2\right)\\
    &=\prob\left(\rH(\alpha_1 L^{-1},1+\tau_1 L^{-3/2})\ge L +\beta_1 L^{-1/2}, \rH(\alpha_2 L^{-1},1+\tau_2 L^{-3/2})< L +\beta_2 L^{-1/2}\right)\\
    &\le \prob\left( 
    \begin{array}{c}
    \mathcal{L}(0,0;\alpha_1 L^{-1},1+\tau_1 L^{-3/2})\ge L +\beta_1 L^{-1/2}\\
    \mathcal{L}(\alpha_1 L^{-1},1+\tau_1 L^{-3/2};\alpha_2 L^{-1},1+\tau_2 L^{-3/2})<(\beta_2-\beta_1) L^{-1/2} 
    \end{array}
    \right)\\
    &=\prob\left(\HH(\alpha_1,\tau_1) \ge \beta_1\right) \cdot \FGUE\left(\frac{\beta_2-\beta_1}{(\tau_2-\tau_1)^{1/3}}+\frac{(\alpha_2-\alpha_1)^2}{(\tau_2-\tau_1)^{4/3}}\right).
    \end{split}
    \end{equation}
    As a special case we have
    \begin{equation}
    \prob\left(\HH(0,0)\ge 0, \HH(\alpha,\tau) < \beta\right)\le \prob\left(\HH(0,0)\ge 0\right)\cdot \FGUE\left(\frac{\beta}{\tau^{1/3}}+\frac{\alpha^2}{\tau^{4/3}}\right) 
    \end{equation}
    when $\tau>0$, and hence
    \begin{equation}
    \prob\left(\HH(\alpha,\tau) < \beta\mid  \HH(0,0)\ge 0\right)\le \FGUE\left(\frac{\beta}{\tau^{1/3}}+\frac{\alpha^2}{\tau^{4/3}}\right).
    \end{equation}
    By taking $L\to\infty$ and applying \eqref{eq:cond_limit}, we get
    \begin{equation}
     1-  \FGUE\left(\frac{\beta}{\tau^{1/3}}+\frac{\alpha^2}{\tau^{4/3}}\right)\le \hat\FT(\beta;(\alpha,\tau)) \le 1.
    \end{equation}
    We further take $\beta\to-\infty$ and obtain the desired limit $\lim_{\beta\to-\infty}\hat\FT(\beta;(\alpha,\tau))=1$.
    
    Case 3: $\tau<0$. Note in this case we have, using Definition \ref{def:cH_by_tail},
    \begin{equation}
        \hat\FT(\beta;(\alpha,\tau)) = \FT(\beta,0;(\alpha,\tau),(0,0))
    \end{equation}
    which goes to $1$ as $\beta\to -\infty$  by  Proposition \ref{prop:FT_tail} which will be proved in the next subsection as a property of the function $\FT$.
    
    Case 4: $\tau=0$ and $\alpha\ne 0$. We first apply \eqref{eq:aux_11_30} and obtain
    \begin{equation}
    \prob\left(\HH(\alpha,0)<\beta,\HH(\alpha,-1)\ge \beta/2\right) \le \prob\left(\HH(\alpha,-1)\ge \beta/2\right)\FGUE\left(\beta/2\right).
    \end{equation}
    By taking $L\to\infty$ and applying \eqref{eq:tail}, we have
    \begin{equation}
    \lim_{L\to\infty}\frac{\prob\left(\HH(\alpha,0)<\beta,\HH(\alpha,-1)\ge \beta/2\right)}{\prob\left(\HH(0,0)\ge 0\right)} \le e^{-\frac{2}{3}-\beta} \FGUE\left(\beta/2\right).
    \end{equation}
    This further implies
    \begin{equation}
    \lim_{L\to\infty}\prob\left(\HH(\alpha,0)<\beta,\HH(\alpha,-1)\ge \beta/2 \mid \HH(0,0)\ge 0\right)\le e^{-\frac{2}{3}-\beta} \FGUE\left(\beta/2\right)\to 0 
    \end{equation}
    as $\beta\to-\infty$, where we used the left tail estimate of the GUE Tracy-Widom distribution $\FGUE(\beta)\approx e^{-(\frac1{24}+o(1))|\beta|^{3}}$, see \cite{BBD08}. On the other hand, the third case above implies
    \begin{equation}
    \lim_{\beta\to-\infty}\lim_{L\to\infty}\prob\left(\HH(\alpha,-1)< \beta/2 \mid \HH(0,0)\ge 0\right)=\lim_{\beta\to-\infty} (1-\hat\FT(\beta/2;(\alpha,-1)))=0.
    \end{equation}
     Finally we also note that
    \begin{equation}
    \begin{split}
    &\prob\left(\HH(\alpha,0)<\beta \mid \HH(0,0)\ge 0\right) \\
    &\le \prob\left(\HH(\alpha,0)<\beta,\HH(\alpha,-1)\ge \beta/2 \mid \HH(0,0)\ge 0\right) + \prob\left(\HH(\alpha,-1)< \beta/2 \mid \HH(0,0)\ge 0\right).
    \end{split}
    \end{equation}
    Combining the above argument, we obtain $\lim_{\beta\to-\infty}\lim_{L\to\infty}\prob\left(\HH(\alpha,0)<\beta \mid \HH(0,0)\ge 0\right)=0$, which implies $\lim_{\beta\to-\infty}\hat\FT(\beta;(\alpha,0))=1$. This completes the proof.
\end{proof}

\subsection{Some properties of $\FT$}
\label{sec:Properties_FT}

In this subsection, we discuss several properties of the function $\FT$.

We first note that the parameters $\beta_\ell$, $\alpha_\ell$, $\tau_\ell$, $1\le \ell \le m$, only appear in the functions $f_\ell$ in the definition of $\FT$, and these functions are exponential functions which are differentiable arbitrarily many times with respect to the parameters. One might wonder whether this property passes to the function $\FT$. This actually is true, and we state it in Proposition \ref{prop:differentiability}. It follows from the following lemma about the change of the order of integral/summation and the differentiation.

\begin{lm}
    \label{lm:exchange_differentiation_sum}
    Suppose $\mu$ is a complex measure on a space $\Omega$, and $x\in\realR$ is a fixed number. If 
    \begin{equation}
        \int_\Omega \left| F(W)e^{y g(W)} \right| |\rd \mu (W)|<\infty, \quad  \int_\Omega \left| F(W)g(W)e^{y g(W)} \right| |\rd \mu (W)|<\infty
    \end{equation}
    for $y\in(x-\delta,x+\delta)$ for some $\delta>0$, then 
    \begin{equation}
    \frac{\rd }{\rd y}\Bigg|_{y=x}\int_\Omega F(W)e^{y g(W)} \rd \mu(W)= \int_\Omega F(W)g(W)e^{x g(W)} \rd \mu(W).
    \end{equation}
\end{lm} 

\begin{proof}[Proof of Lemma \ref{lm:exchange_differentiation_sum}]
    We need the following simple inequality
    \begin{equation}
        \label{eq:simple_inequality_exp}
        \left| \frac{e^w -1 }{w} \right| \le e\cdot (|e^w| +1),\quad w\in\complexC.    
       \end{equation}
    Now assuming this inequality holds, we prove the lemma. We write
    \begin{equation}
        \label{eq:exchange_differentiation_sum2}
        \begin{split}
        \frac{\rd }{\rd y}\Bigg|_{y=x}\int_\Omega F(W)e^{y g(W)} \rd \mu(W)
        &=\lim_{\epsilon\to 0} \int_\Omega F(W)\frac{e^{(x+\epsilon) g(W)}-e^{x g(W)}}{\epsilon} \rd \mu(W)\\
        &=\lim_{\epsilon\to 0} \int_\Omega F(W)g(W)\frac{e^{(x+\epsilon) g(W)}-e^{x g(W)}}{\epsilon g(W)} \rd \mu(W).
        \end{split}
    \end{equation}
    Note that for any $\epsilon \in (-\delta/2,\delta/2)$, we have, by using \eqref{eq:simple_inequality_exp},
       \begin{equation}
        \begin{split}
        \left| \frac{e^{(x+\epsilon) g(W)}-e^{x g(W)}}{\epsilon g(W)} \right| &\le e\cdot \left|e^{x g(W)}\right| \left(|e^{\epsilon g(W)}| +1\right)\\
        & \le e\cdot \left|e^{x g(W)}\right| \left(|e^{\frac{\delta}{2} g(W)}| +|e^{-\frac{\delta}{2} g(W)}|+1\right).
        \end{split}
       \end{equation}
    Moreover, the assumption in the lemma implies
    \begin{equation}
        \int_\Omega |F(W)g(W)| e\cdot \left|e^{x g(W)}\right| \left(|e^{\frac{\delta}{2} g(W)}| +|e^{-\frac{\delta}{2} g(W)}|+1\right)  |\rd \mu(W)|<\infty.
    \end{equation}   
    Therefore the dominated convergence theorem applies on the right hand side of \eqref{eq:exchange_differentiation_sum2} and we can take the differentiation within the integral. This proves the lemma.

    It remains to prove \eqref{eq:simple_inequality_exp}. We prove it by considering two different cases.
    The first case is  $|w|\ge 1$. In this case, we have
\begin{equation}
    \left| \frac{e^w -1 }{w} \right| \le \left|e^w-1\right| \le |e^w|+1.
\end{equation}
The second case is that $|w|\le 1$. Note that $(e^w-1)/w$ is analytic. Hence, its maximal norm within the disk is obtained at the boundary
\begin{equation}
    \left| \frac{e^w-1}{w}\right| \le \max_{|w'|=1} \left| \frac{e^{w'}-1}{w'}\right|= \max_{|w'|=1} |e^{w'} -1| \le e+1 \le e+ |e^{w+1}|.
\end{equation}
Combining the two cases we get the desired inequality.
\end{proof}

A simple corollary of Lemma \ref{lm:exchange_differentiation_sum} is that we can take the derivatives with respect to the parameters of $\FT$ inside the summand and integrand, more explicitly, the derivatives can be taken to the function $\prod_{\ell=2}^m \prod_{i_\ell=1}^{n_m} \frac{f_\ell(u_{i_\ell}^{(\ell)})}{f_\ell(v_{i_\ell}^{(\ell)})}\cdot \frac{f_1(-1)}{f_1(1)}$ which is the only part containing the parameters. Note the assumptions of Lemma \ref{lm:exchange_differentiation_sum} are satisfied as long as the parameters are within the interior of the domain $\{(\alpha_1,\tau_1)\prec \cdots\prec (\alpha_m,\tau_m)\}$. 

\begin{prop}[Differentiability]
    \label{prop:differentiability}
The function $\FT(\bbeta;(\alpha_1,\tau_1),\cdots,(\alpha_m,\tau_m))$ is a smooth function on all the parameters $\beta_\ell$, $\alpha_\ell$, $\tau_\ell$, $1\le\ell\le m$, within the domain $\{(\alpha_1,\tau_1)\prec \cdots\prec (\alpha_m,\tau_m)\}$. Moreover, the derivatives can be taken inside $\cD_{\bn}$ in \eqref{eq:def_FT_2} when $m\ge 2$. More explicitly, let $\mathrm{D}$ denote the differentiation operator with respect to any of these parameters, and $\mathrm{D}^k$ denote the $k$-th differentiation which could be with respect to different parameters, we have
\begin{equation}
    \begin{split}
    &\mathrm{D}^k \FT(\bbeta;(\alpha_1,\tau_1),\cdots,(\alpha_m,\tau_m))\\
    &=(-1)^m \oint_{>1}\cdots \oint_{>1} \sum_{\substack{n_\ell \ge 1\\ 2\le \ell\le m}} \frac{1}{ (n_2!\cdots n_{m-1}!)^2} \left(\mathrm{D}^k\cD_{\bn}(\bbeta;\bz) \right)\prod_{\ell=1}^{m-1}\frac{\rd z_\ell}{2\pi\ii z_\ell(1-z_\ell)},
    \end{split}
\end{equation}
where $\mathrm{D}^k\cD_{\bn}(\bbeta;\bz)$ has the same formula as for $\cD_{\bn}(\bbeta;\bz)$ in \eqref{eq:def_cdnz}, except that we need to replace the factor $\prod_{\ell=2}^m \prod_{i_\ell=1}^{n_m} \frac{f_\ell(u_{i_\ell}^{(\ell)})}{f_\ell(v_{i_\ell}^{(\ell)})}\cdot \frac{f_1(-1)}{f_1(1)}$ by $\mathrm{D}^k\prod_{\ell=2}^m \prod_{i_\ell=1}^{n_m} \frac{f_\ell(u_{i_\ell}^{(\ell)})}{f_\ell(v_{i_\ell}^{(\ell)})}\cdot \frac{f_1(-1)}{f_1(1)}$ in the integrand of \eqref{eq:def_cdnz}.
\end{prop}

\bigskip

The next property of $\FT$ involves the shift of parameters. The functions $f_\ell$ have a nice invariance property under the shift of parameters, which passes to the function $\FT$. This property is explicitly stated below.

\begin{prop}[Shift on parameters]
    \label{prop:shift_parameters}
    Assume that $\beta_1,\cdots,\beta_m\in\realR$, and $(\alpha_1,\tau_1)\prec \cdots\prec(\alpha_m,\tau_m)$ are $m$ ordered points on $\realR^2$. Then we have
    \begin{equation}
        \label{eq:shift_parameter}
        \FT(\bbeta;(\alpha_1,\tau_1),\cdots,(\alpha_m,\tau_m))
        = e^{\frac{2}{3}\hat\tau -2\hat\beta} \FT(\beta_1-\hat\beta,\cdots,\beta_m-\hat\beta; (\alpha_1-\hat\alpha,\tau_1-\hat\tau),\cdots,(\alpha_m-\hat\alpha,\tau_m-\hat\tau))
    \end{equation}
    for any $\hat\beta,\hat\alpha,\hat\tau \in \realR$. Moreover, if $\mathrm{D}^k$ denotes the $k$-th differentiation which could be with respect to different parameters among $\{\alpha_1,\cdots,\alpha_m,\beta_1,\cdots,\beta_m,\tau_1,\cdots,\tau_m\}$, then $\mathrm{D}^k\FT$ has the same shift property as $\FT$. More explicitly, if 
    \begin{equation}
        \mathcal{T}(\bbeta;(\alpha_1,\tau_1),\cdots,(\alpha_m,\tau_m)) = \mathrm{D}^k\FT(\bbeta;(\alpha_1,\tau_1),\cdots,(\alpha_m,\tau_m)),
    \end{equation}
    then
    \begin{equation}
        \label{eq:shift_parameter1}
        \mathcal{T}(\bbeta;(\alpha_1,\tau_1),\cdots,(\alpha_m,\tau_m))
        = e^{\frac{2}{3}\hat\tau -2\hat\beta} \mathcal{T}(\beta_1-\hat\beta,\cdots,\beta_m-\hat\beta; (\alpha_1-\hat\alpha,\tau_1-\hat\tau),\cdots,(\alpha_m-\hat\alpha,\tau_m-\hat\tau)).
    \end{equation}
\end{prop}

We remark that for the second part of this proposition, if we take 
\begin{equation}
    \label{eq:shift_parameter_special1}
    \mathcal{T}(\bbeta;(\alpha_1,\tau_1),\cdots,(\alpha_m,\tau_m)) = \frac{\partial^m}{\partial \beta_1\cdots\partial\beta_m}\FT(\bbeta;(\alpha_1,\tau_1),\cdots,(\alpha_m,\tau_m)),
\end{equation}
then we have
\begin{equation}
    \label{eq:shift_parameter_special2}
    \mathcal{T}(\bbeta;(\alpha_1,\tau_1),\cdots,(\alpha_m,\tau_m)) = e^{\frac23 \tau_k -2\beta_k} 
    \mathcal{T}(\beta_1-\beta_k,\cdots,\beta_m-\beta_k; (\alpha_1-\alpha_k,\tau_1-\tau_k),\cdots,(\alpha_m-\alpha_k,\tau_m-\tau_k))
\end{equation}
for any $1\le k\le m$.
\begin{proof}[Proof of Proposition \ref{prop:shift_parameters}]
We first consider the proof of \eqref{eq:shift_parameter}. It is a direct check when $m=1$. For $m\ge 2$, note that the functions $f_\ell$, $2\le \ell\le m$, are invariant if we shift the parameters $\beta_\ell$ (or $\alpha_\ell$, $\tau_\ell$), $1\le \ell \le m$, by the same constant. See the definition of $f_\ell$ in \eqref{eq:def_f}. Hence the shift of parameters only affects the function $f_1$ in \eqref{eq:def_cdnz} and we have
\begin{equation}
    \begin{split}
        \cD_{\bn}(\bbeta-\hat\beta;\bz;(\alpha_1-\hat\alpha,\tau_1-\hat\tau),\cdots,(\alpha_m-\hat\alpha,\tau_m-\hat\tau))
        = e^{-\frac23\hat\tau+2\hat\beta}\cD_{\bn}(\bbeta;\bz;(\alpha_1,\tau_1),\cdots,(\alpha_m,\tau_m))
    \end{split}
\end{equation}
where $\bbeta=(\beta_1,\cdots,\beta_m)$ and $\bbeta-\hat\beta= (\beta_1-\hat\beta,\cdots,\beta_m-\hat\beta)$. Inserting it to \eqref{eq:def_FT_2} we get the desired property.

The proof of \eqref{eq:shift_parameter1} is similar. We apply Proposition \ref{prop:differentiability} and the differentiation can be passed to the exponential function 
\begin{equation}
    \mathrm{D}^k\prod_{\ell=2}^m \prod_{i_\ell=1}^{n_m} \frac{f_\ell(u_{i_\ell}^{(\ell)})}{f_\ell(v_{i_\ell}^{(\ell)})}\cdot \frac{f_1(-1)}{f_1(1)}
     = H\cdot \prod_{\ell=2}^m \prod_{i_\ell=1}^{n_m} \frac{f_\ell(u_{i_\ell}^{(\ell)})}{f_\ell(v_{i_\ell}^{(\ell)})}\cdot \frac{f_1(-1)}{f_1(1)}
\end{equation}
where $H$ is some function independent of the parameters $\alpha_\ell,\beta_\ell$ and $\tau_\ell$, $1\le \ell\le m$. Thus the above expression has the same shift property as that without the differentiation. \eqref{eq:shift_parameter1} follows immediately.
\end{proof}

\bigskip

Finally we are interested in the tail behavior of the function $\FT$. When $m=1$, $\FT$ is an exponential function. When $m=2$, we will see that $\FT$ has some nontrivial tail behaviors. We can assume that $\alpha_1=\beta_1=\tau_1=0$ or $\alpha_2=\beta_2=\tau_2=0$ by the Proposition \ref{prop:shift_parameters}.

\begin{prop}
    \label{prop:FT_tail}
      Assume that $\tau>0$ and $\alpha\in\realR$ are both fixed. Then we have
      \begin{equation}
        \label{eq:FT_tail1}
        \begin{split}
            &\FT(0,\beta;(0,0),(\alpha,\tau)) - e^{\frac{2}{3}\tau -2\beta}\\
            &= \begin{dcases}
                \displaystyle -\frac{\tau^{3/4}}{\sqrt{\pi}\beta^{5/4}} e^{-\frac{2}{3}\frac{\left(\beta+\frac{\alpha^2}{\tau}\right)^{3/2}}{\tau^{1/2}}+|\alpha|\left(\frac{\beta}{\tau}-1+\frac{2\alpha^2}{3\tau^2}\right) +\frac{1}{3}\tau -\beta} \cdot \left(1+ O(\beta^{-3/4})\right),& \alpha\ne 0,\\
                \displaystyle -\frac{2\tau^{3/4}}{\sqrt{\pi}\beta^{5/4}} e^{-\frac{2}{3}\frac{\beta^{3/2}}{\tau^{1/2}}+\frac{1}{3}\tau -\beta} \cdot \left(1+ O(\beta^{-3/4})\right),& \alpha= 0,
            \end{dcases}
        \end{split}
      \end{equation}
      and
      \begin{equation}
        \label{eq:FT_tail2}
        \begin{split}
            &\FT(-\beta,0;(-\alpha,-\tau),(0,0)) - 1\\
            &= \begin{dcases}
                \displaystyle -\frac{\tau^{3/4}}{\sqrt{\pi}\beta^{5/4}} e^{-\frac{2}{3}\frac{\left(\beta+\frac{\alpha^2}{\tau}\right)^{3/2}}{\tau^{1/2}}+|\alpha|\left(\frac{\beta}{\tau}-1+\frac{2\alpha^2}{3\tau^2}\right) -\frac{1}{3}\tau +\beta} \cdot \left(1+ O(\beta^{-3/4})\right),& \alpha\ne 0,\\
                \displaystyle -\frac{2\tau^{3/4}}{\sqrt{\pi}\beta^{5/4}} e^{-\frac{2}{3}\frac{\beta^{3/2}}{\tau^{1/2}}-\frac{1}{3}\tau +\beta} \cdot \left(1+ O(\beta^{-3/4})\right),& \alpha= 0,
            \end{dcases}
        \end{split}
      \end{equation}
      when $\beta\to\infty$.
\end{prop}

\begin{proof}[Proof of Proposition \ref{prop:FT_tail}]
    Note that the shift property in Proposition \ref{prop:shift_parameters} implies that
    \begin{equation}
    \label{eq:aux_12_31_01}
    \FT(-\beta,0;(-\alpha,-\tau),(0,0)) = e^{-\frac{2}{3}\tau +2\beta} \FT(0,\beta;(0,0),(\alpha,\tau)),\quad \text{for any } (0,0)\prec (\alpha,\tau).
    \end{equation}
    Hence the two formulas \eqref{eq:FT_tail1} and \eqref{eq:FT_tail2} are equivalent. Below we only show \eqref{eq:FT_tail1}.

    We apply Proposition \ref{prop:FT_alt} and write
    \begin{equation}
        \label{eq:FT0_initial}
        \begin{split}
           &\FT(0,\beta;(0,0),(\alpha,\tau))\\
           &= \sum_{n\ge 1} \frac{1}{(n!)^2} \prod_{i=1}^n \int_{\Gamma_{2,\LL}^\out} \frac{\rd u_i}{2\pi\ii} \int_{\Gamma_{2,\RR}^\out} \frac{\rd v_i}{2\pi\ii} 2\rC(-1\bunion V;1\bunion U)\rC(U;V) \prod_{i=1}^n \frac{f_2(u_i)}{f_2(v_i)}\\
           &=\sum_{n\ge 1}\frac{(-1)^{n+1}}{(n!)^2}\prod_{i=1}^n \int_{\Gamma_{2,\LL}^\out} \frac{\rd u_i}{2\pi\ii} \int_{\Gamma_{2,\RR}^\out} \frac{\rd v_i}{2\pi\ii} \frac{\prod_{1\le i<j\le n} (u_i-u_j)^2 (v_i-v_j)^2}{ \prod_{1\le i,j\le n}(u_i-v_j)^2} \prod_{i=1}^n \frac{(u_i-1)(v_i+1)}{(u_i+1)(v_i-1)} \prod_{i=1}^n \frac{f_2(u_i)}{f_2(v_i)}
        \end{split}
    \end{equation}
where $U=(u_1,\cdots,u_n)$ and $V=(v_1,\cdots,v_n)$ and we applied the Cauchy determinant formula \eqref{eq:def_rC} in the second equation. The function $f_2(w) = e^{-\frac{\tau}{3}w^3 +\alpha w^2 +\beta w}$ as defined in \eqref{eq:def_f}. 

Now we deform the contours $\Gamma_{2,\LL}^\out$ and $\Gamma_{2,\RR}^\out$ to $\Gamma_{2,\LL}^\inn$ and $\Gamma_{2,\RR}^\inn$ respectively. Note that the $u_i$-contour passes the pole $-1$ and the $v_i$-contour passes the pole $1$ during the deformation. Hence the deformation can be expressed as 
\begin{equation}
    \int_{\Gamma_{2,\LL}^\out} \cdots \frac{\rd u_i}{2\pi\ii} = \int_{\Gamma_{2,\LL}^\inn} \cdots \frac{\rd u_i}{2\pi\ii} + \mathrm{Res}(\cdots;u_i=-1)
\end{equation}
where $\cdots$ is the suppressed integrand for simplification and $\mathrm{Res}(f;z=a)$ is the residue of $f$ at the point $a$
\begin{equation}
    \mathrm{Res}(f;z=a) = \oint_a f(z) \frac{\rd z}{2\pi\ii},
\end{equation}
and similarly
\begin{equation}
    \int_{\Gamma_{2,\RR}^\out} \cdots \frac{\rd v_i}{2\pi\ii} = \int_{\Gamma_{2,\RR}^\inn} \cdots \frac{\rd v_i}{2\pi\ii} - \mathrm{Res}(\cdots;v_i=1),
\end{equation}
here we have the $-$ sign in front of the residue since the orientation of the contours $\Gamma_{2,\RR}^\out$ and $\Gamma_{2,\RR}^\inn$ are from $\infty e^{-\ii 2\pi/5 }$ to $\infty e^{\ii 2\pi/5}$. See the beginning of Section \ref{sec:def_FT}. Note the factor $\prod_{1\le i<j\le n}(u_i-u_j)^2(v_i-v_j)^2$. So if two deformations on the left contours end at evaluating the residue, say at $u_i=1$ and $u_j=1$, then the integrand vanishes. Similarly we at most evaluate the residue at $1$ once when we deform the right contours. Also note the integrand is symmetric in $u_1,\cdots,u_n$ and in $v_1,\cdots,v_n$. We can write \eqref{eq:FT0_initial} as a sum of four terms
\begin{equation}
    \label{eq:FT_rt_split}
    \FT(0,\beta;(0,0),(\alpha,\tau)) = \FT_0 +\FT_1+\FT_{-1} +\FT_{2}
\end{equation}
where $\FT_0$ is the term obtained without evaluating any residues
\begin{equation}
    \FT_0 = \sum_{n\ge 1}\frac{(-1)^{n+1}}{(n!)^2}\prod_{i=1}^n \int_{\Gamma_{2,\LL}^\inn} \frac{\rd u_i}{2\pi\ii} \int_{\Gamma_{2,\RR}^\inn} \frac{\rd v_i}{2\pi\ii} \frac{\prod_{1\le i<j\le n} (u_i-u_j)^2 (v_i-v_j)^2}{ \prod_{1\le i,j\le n}(u_i-v_j)^2} \prod_{i=1}^n \frac{(u_i-1)(v_i+1)}{(u_i+1)(v_i-1)} \prod_{i=1}^n \frac{f_2(u_i)}{f_2(v_i)},
\end{equation}
$\FT_1$ is the term with one $v_i$ integral converted to the residue at $1$. There are $n$ such terms and by symmetry we assume $v_n$-integral becomes the residue at $v_n=1$. After simplification, we have
\begin{equation}
    \begin{split}
        \FT_1
        &=2e^{\frac{1}{3}\tau -\alpha-\beta}\sum_{n\ge 1} \frac{(-1)^{n}}{((n-1)!)^2n} \prod_{i=1}^n \int_{\Gamma_{2,\LL}^\inn} \frac{\rd u_i}{2\pi\ii} \prod_{j=1}^{n-1}\int_{\Gamma_{2,\RR}^\inn} \frac{\rd v_j}{2\pi\ii} \\
        &\quad\frac{\prod_{1\le i<j\le n}(u_i-u_j)^2 \prod_{1\le i<j\le n-1}(v_i-v_j)^2}{\prod_{1\le i\le n}\prod_{1\le j\le n-1}(u_i-v_j)^2} \cdot \frac{\prod_{j=1}^{n-1}(v_j-1)(v_j+1)}{\prod_{i=1}^n(u_i-1)(u_i+1)} \cdot \frac{\prod_{i=1}^n f_2(u_i)}{\prod_{j=1}^{n-1}f_2(v_j)}.
    \end{split}
\end{equation}
Similarly, $\FT_{-1}$ is the term with one $u_i$ integral converted to the residue at $-1$
\begin{equation}
    \begin{split}
        \FT_{-1}
        &=2e^{\frac{1}{3}\tau +\alpha-\beta}\sum_{n\ge 1} \frac{(-1)^{n}}{((n-1)!)^2n} \prod_{i=1}^{n-1} \int_{\Gamma_{2,\LL}^\inn} \frac{\rd u_i}{2\pi\ii} \prod_{j=1}^{n}\int_{\Gamma_{2,\RR}^\inn} \frac{\rd v_j}{2\pi\ii} \\
        &\quad\frac{\prod_{1\le i<j\le n-1}(u_i-u_j)^2 \prod_{1\le i<j\le n}(v_i-v_j)^2}{\prod_{1\le i\le n-1}\prod_{1\le j\le n}(u_i-v_j)^2} \cdot \frac{\prod_{i=1}^{n-1}(u_i-1)(u_i+1)}{\prod_{j=1}^{n}(v_j-1)(v_j+1)} \cdot \frac{\prod_{i=1}^{n-1} f_2(u_i)}{\prod_{j=1}^{n}f_2(v_j)}.
    \end{split}
\end{equation}
And finally, $\FT_2$ is the term with one $u_i$ integral converted to the residue at $-1$ and one $v_j$ integral converted to the residue at $1$
\begin{equation}
    \begin{split}
        \FT_{2}
        &=-e^{\frac{2}{3}\tau -2\beta}\sum_{n\ge 1} \frac{(-1)^{n}}{((n-1)!)^2} \prod_{i=1}^{n-1} \int_{\Gamma_{2,\LL}^\inn} \frac{\rd u_i}{2\pi\ii} \prod_{j=1}^{n-1}\int_{\Gamma_{2,\RR}^\inn} \frac{\rd v_j}{2\pi\ii} \\
        &\quad\frac{\prod_{1\le i<j\le n-1}(u_i-u_j)^2 \prod_{1\le i<j\le n-1}(v_i-v_j)^2}{\prod_{1\le i\le n-1}\prod_{1\le j\le n-1}(u_i-v_j)^2} \cdot \prod_{i=1}^{n-1}\frac{(v_i-1)(u_i+1)}{(u_i-1)(v_i+1)} \cdot\prod_{i=1}^{n-1} \frac{ f_2(u_i)}{f_2(v_i)}.
    \end{split}
\end{equation}
Now we apply the steepest descent analysis for each term as $\beta$ becomes large. Recall that $f_2(w) = e^{-\frac{\tau}{3}w^3+\alpha w^2 +\beta w}$. Note that when $\alpha=0$ and $\tau=1$, the integral of $f_2(w)$ along the contour $\Gamma_{2,\LL}^\inn$ is the Airy function of $\beta$. It is a classical textbook exercise to analyze the asymptotics when $\beta\to\infty$. The same exercise (after a shift and rescale of the $w$ variable) actually implies
\begin{equation}
    \int_{\Gamma_{2,\LL}^\inn} g(u) f_2(u) \frac{\rd u}{2\pi\ii} = \frac{g(-\sqrt{\beta/\tau})}{2\sqrt{\pi}(\beta\tau)^{1/4}}e^{-\frac{2}{3}\frac{\left(\beta +\frac{\alpha^2}{\tau}\right)^{3/2}}{\tau^{1/2}} + \frac{2}{3}\frac{\alpha^3}{\tau^2} +\frac{\alpha\beta}{\tau}} \left(1+ O(\beta^{-3/4})\right)
\end{equation} 
when $\beta\to\infty$ provided $g(u)$ is an analytic function to the left of $\Gamma_{2,\LL}^\inn$ independent of $\beta$ and it at most grows polynomially when $u\to\infty$. Similarly,
\begin{equation}
    \int_{\Gamma_{2,\RR}^\inn} g(v) \frac{1}{f_2(v)} \frac{\rd v}{2\pi\ii} = \frac{g(\sqrt{\beta/\tau})}{2\sqrt{\pi}(\beta\tau)^{1/4}}e^{-\frac{2}{3}\frac{\left(\beta +\frac{\alpha^2}{\tau}\right)^{3/2}}{\tau^{1/2}} - \frac{2}{3}\frac{\alpha^3}{\tau^2} -\frac{\alpha\beta}{\tau}} \left(1+ O(\beta^{-3/4})\right)
\end{equation} 
when $\beta\to\infty$ provided $g(v)$ is an analytic function to the right of $\Gamma_{2,\RR}^\inn$ independent of $\beta$ and it at most grows polynomially when $v\to\infty$. Note that both expressions decay like $e^{-\frac{2}{3}\beta^{3/2}/\tau^{1/2}}$ which is super-exponentially small when $\beta$ becomes large. This implies that in the summations of $\FT_i$ functions, $i=0,\pm1,2$, the leading term comes from the term with the least number of integrals, i.e., the summand when $n=1$. Inserting the above asymptotics in the $\FT_i$ functions when $n=1$, we get
\begin{equation}
    \begin{split}
    \FT_0 &= O\left(e^{-\frac43 \frac{\left(\beta +\frac{\alpha^2}{\tau}\right)^{3/2}}{\tau^{1/2}} }\right),\\
    \FT_1 & = -\frac{\tau^{3/4}}{\sqrt{\pi}\beta^{5/4}}e^{-\frac{2}{3}\frac{\left(\beta +\frac{\alpha^2}{\tau}\right)^{3/2}}{\tau^{1/2}} + \frac{2}{3}\frac{\alpha^3}{\tau^2} +\frac{\alpha\beta}{\tau}+\frac{1}{3}\tau-\alpha-\beta}\left(1+ O(\beta^{-3/4})\right),\\
    \FT_{-1}&=-\frac{\tau^{3/4}}{\sqrt{\pi}\beta^{5/4}}e^{-\frac{2}{3}\frac{\left(\beta +\frac{\alpha^2}{\tau}\right)^{3/2}}{\tau^{1/2}} - \frac{2}{3}\frac{\alpha^3}{\tau^2} -\frac{\alpha\beta}{\tau}+\frac{1}{3}\tau+\alpha-\beta}\left(1+ O(\beta^{-3/4})\right),\\
    \FT_{2}&=e^{\frac{2}{3}\tau-2\beta} + O\left(e^{-\frac43 \frac{\left(\beta +\frac{\alpha^2}{\tau}\right)^{3/2}}{\tau^{1/2}} }\right).
    \end{split}
\end{equation}
Inserting these asymptotics to \eqref{eq:FT_rt_split} we arrive at \eqref{eq:FT_tail1}.
\end{proof}

\subsection{Properties of $\cH_0$}
\label{sec:Properties_cH}
    
Recall the random field $\cH_0(\alpha,\tau) = \cH(\alpha,\tau) - \cH(0,0)$. We first prove that the field $\cH_0$ is independent of $\cH(0,0)$. It follows from the following lemma and the fact that $\cH$ has continuous tail probability functions by the definition (thus the non-differentiability at $\beta=0$ of the tail probability functions does not affect the independence).

\begin{lm}
    \label{lm:joint_density}
    Assume $m\ge k\ge 1$, and $(\alpha_1,\tau_1)\prec \cdots \prec (\alpha_{k-1},\tau_{k-1}) \prec (0,0) \prec (\alpha_k,\tau_k) \prec \cdots \prec(\alpha_m,\tau_m)$ on the plane $\realR^2$. The joint density function of $\cH_0(\alpha_\ell,\tau_\ell)$ and $\cH(0,0)$ exists and is given by
     \begin{equation}
        \label{eq:joint_density_1}
        \begin{split}
            &(-1)^{m+1}\frac{\partial^{m+1}}{\partial \beta_1 \cdots \partial \beta_m \partial \beta} \prob\left(\bigcap_{\ell=1}^m \left\{ \cH_0(\alpha_\ell,\tau_\ell) \ge \beta_k \right\} \bigcap \left\{\cH(0,0)\ge \beta\right\}
            \right)
            \\
            &=\dP(\beta_1,\cdots,\beta_m;(\alpha_1,\tau_1),\cdots,(\alpha_m,\tau_m))\cdot \dQ(\beta),\quad \beta \ne 0,
        \end{split}
     \end{equation}
     where $\dP$ is a function independent of $\beta$
     \begin{multline}
            \dP(\beta_1,\cdots,\beta_m;(\alpha_1,\tau_1),\cdots,(\alpha_m,\tau_m))=\frac12 \frac{(-1)^{m+1} \partial^{m+1}}{\partial \beta_1 \cdots \partial \beta_m\partial \hat\beta}\Bigg|_{\hat\beta=0} \FT(\beta_1,\cdots,\beta_{k-1},\hat\beta,\beta_{k},\cdots,\beta_m; \\
             (\alpha_1,\tau_1),\cdots,(\alpha_{k-1},\tau_{k-1}),(0,0),(\alpha_k,\tau_k),\cdots,(\alpha_m,\tau_m))
     \end{multline}
     and $\dQ$ is given by
     \begin{equation}
        \dQ(\beta) = \begin{dcases}
            2e^{-2\beta},& \beta>0,\\
            0,& \beta<0.
        \end{dcases}
     \end{equation}
\end{lm}
\begin{proof}
    We first note that the function $\dP$ is well defined by Proposition \ref{prop:differentiability}.

    In order to show \eqref{eq:joint_density_1}, it is sufficient to show that, by the change of variables, the joint density function of $\cH(\alpha_1,\tau_1)$, $\cdots$, $\cH(\alpha_m,\tau_m)$, and $\cH(0,0)$, exists and is given by
    \begin{equation}
        \label{eq:joint_density_2}
        \begin{split}
            &(-1)^{m+1}\frac{\partial^{m+1}}{\partial \beta_1 \cdots \partial \beta_m \partial \beta} \prob\left(\bigcap_{\ell=1}^m \left\{ \cH(\alpha_\ell,\tau_\ell) \ge \beta_k \right\} \bigcap \left\{\cH(0,0)\ge \beta\right\}
            \right)
            \\
            &=\dP(\beta_1-\beta,\cdots,\beta_m-\beta;(\alpha_1,\tau_1),\cdots,(\alpha_m,\tau_m))\cdot \dQ(\beta)
        \end{split}
     \end{equation}
     when $\beta\ne 0$. Note that when $\beta<0$, the probability $\prob\left(\bigcap_{\ell=1}^m \left\{ \cH(\alpha_\ell,\tau_\ell) \ge \beta_k \right\} \bigcap \left\{\cH(0,0)\ge \beta\right\}
            \right)$ is independent of $\beta$ by Definition \ref{def:cH_by_tail}, hence the left hand side of \eqref{eq:joint_density_2} is zero which matches its right hand side. When $\beta>0$, the left hand side of \eqref{eq:joint_density_2} equals to, by the definition \ref{def:cH_by_tail},
            \begin{multline}
                (-1)^{m+1}\frac{\partial^{m+1}}{\partial \beta_1 \cdots \partial \beta_m \partial \beta} \FT(\beta_1,\cdots,\beta_{k-1},\beta,\beta_k,\cdots,\beta_m; \\
                (\alpha_1,\tau_1),\cdots,(\alpha_{k-1},\tau_{k-1}),(0,0),(\alpha_k,\tau_k),\cdots,(\alpha_m,\tau_m)).
            \end{multline}
            Note the definitions of $\dP$ and $\dQ$. Thus \eqref{eq:joint_density_2} is a property of the shift of parameters of the derivative of the function $\FT$, which follows from Proposition \ref{prop:shift_parameters}, or more explicitly the equations \eqref{eq:shift_parameter_special1} and \eqref{eq:shift_parameter_special2}. 
\end{proof}

\begin{rmk}
    A simple corollary of this proposition is that the joint density function of $\cH_0(\alpha_1,\tau_1)$, $\cdots$, $\cH_0(\alpha_m,\tau_m)$ exists and is given by $\dP$ since $\dQ$ is the density of $\cH(0,0)$.
\end{rmk}

Below we also provide a formula of the joint tail probability functions of $\cH_0$.
\begin{prop}
    \label{prop:finite_dimensional_cH_0}
    Let $m\ge 2$. Assume that $(\alpha_1,\tau_1)\prec \cdots \prec (\alpha_{m},\tau_{m})$ are $m$ points satisfying $(\alpha_k,\tau_k)=(0,0)$ for some $1\le k\le m$, and $\beta_1,\cdots,\beta_{k-1},\beta_{k+1},\cdots,\beta_{m}\in \realR$ are fixed. Then
    \begin{equation}
        \prob\left( \bigcap_{\substack{1\le \ell \le m\\ \ell \ne k}}\left\{\cH_0(\alpha_\ell,\tau_\ell) \ge \beta_\ell \right\}\right) = -\frac12 \frac{\partial}{\partial \beta_k}\Bigg|_{\beta_k=0} \FT(\beta_1,\cdots,\beta_{m};(\alpha_1,\tau_1),\cdots,(\alpha_{m},\tau_{m})).
    \end{equation}     
\end{prop}
\begin{proof}[Proof of Proposition \ref{prop:finite_dimensional_cH_0}]
    Note that $\cH_0$ is independent of $\cH(0,0)$ as proved at the beginning of this subsection. Thus
    \begin{equation}
        \begin{split}
            \prob\left( \bigcap_{\substack{1\le \ell \le m\\ \ell \ne k}}\left\{\cH_0(\alpha_\ell,\tau_\ell) \ge \beta_\ell \right\}\right)
            &= \prob\left( \bigcap_{\substack{1\le \ell \le m\\ \ell\ne k }}\left\{\cH_0(\alpha_\ell,\tau_\ell) \ge \beta_\ell \right\}\, \Big| \, \cH(\alpha_k=0,\tau_k=0)=\hat\beta\right)\\
            &=\prob\left( \bigcap_{\substack{1\le \ell \le m\\ \ell\ne k }}\left\{\cH(\alpha_\ell,\tau_\ell) \ge \beta_\ell+\hat\beta \right\}\, \Big| \, \cH(\alpha_k=0,\tau_k=0)=\hat\beta\right)\\
            &=\frac{\frac{\partial}{\partial \beta_k}\Big|_{\beta_k=0}\prob\left( \bigcap_{1\le \ell \le m}\left\{\cH(\alpha_\ell,\tau_\ell)  \ge \beta_\ell +\hat\beta \right\} \right) }{\frac{\partial}{\partial \beta_k}\Big|_{\beta_k=0} \prob\left(\cH(0,0)\ge \beta_k+\hat\beta\right)}
        \end{split}
    \end{equation}
    for any $\hat\beta>0$. Recall Definition \ref{def:cH_by_tail}. The denominator of the right hand side is given by
    \begin{equation}
        \frac{\partial}{\partial \beta_k}\Big|_{\beta_k=0}\FT(\beta_k+\hat\beta;(0,0))=\frac{\partial}{\partial \beta_k}\Big|_{\beta_k=0} e^{-2(\beta_k +\hat \beta)} =-2e^{-2\hat\beta}
    \end{equation}
    and the numerator is, by Proposition \ref{prop:shift_parameters},
    \begin{equation}
        \begin{split}
        &\frac{\partial}{\partial \beta_k}\Big|_{\beta_k=0} \FT(\beta_1+\hat\beta,\cdots,\beta_{m}+\hat\beta;(\alpha_1,\tau_1),\cdots,(\alpha_{m},\tau_{m})) \\
        &= e^{-2\hat\beta} \frac{\partial}{\partial \beta_k}\Big|_{\beta_k=0} \FT(\beta_1,\cdots,\beta_{m};(\alpha_1,\tau_1),\cdots,(\alpha_{m},\tau_{m})).
        \end{split}
    \end{equation}
    Combing the two equations we complete the proof.
\end{proof}

\section{Proof of Proposition \ref{prop:main}}
\label{sec:proof_proposition}
\subsection{Tail probability formula of the KPZ fixed point}
\label{sec:KPZ_tail_prob}

The starting point of the proof is an explicit formula of the multipoint distribution functions of the KPZ fixed point $\rH(x,t)$, with the narrow-wedge initial condition. If $t=t_0$ is fixed, the process $\rH(x,t=t_0)$ is equivalent to a (rescaled) parabolic Airy$_2$ process with finite-dimensional distribution functions expressed as a Fredholm determinant with an extended Airy kernel, see \cite{PS02,Jo03}. In this paper, we need the finite-dimensional distribution functions of the full field $\rH(x,t)$ in the space-time plane. While there are no explicit formulas for a general initial condition so far, the finite-dimensional distribution functions of $\rH(x,t)$, the KPZ fixed point with the narrow-wedge initial condition, were recently obtained in \cite{JR21} and \cite{Liu19}. We will use a variation of the formula in \cite{Liu19} to prove Proposition \ref{prop:main}. More precisely, we will use a joint tail probability formula instead of the regular one proved in \cite{Liu19}. Such a tail probability formula for the KPZ fixed point was not explicitly written before, although similar expressions of the same nature have appeared in \cite{Liu-Wang22} and \cite{baik2024pinchedup}. We state the formula in the following proposition, and include its proof for completeness.

\begin{prop}
\label{prop:KPZ_tail_prob}
Assume $\bx=(x_1,\cdots,x_m), \bh=(h_1,\cdots,h_m)$ are two vectors in $\realR^m$ and $\bt=(t_1,\cdots,t_m)\in \realR_+^m$. Moreover, the points $(x_\ell,t_\ell)$, $1\le \ell\le m$, are ordered in the half plane $\realR\times\realR_+$: $(x_1,t_1)\prec (x_2,t_2) \prec \cdots \prec (x_m,t_m)$, where the relation $\prec$ is defined in the Definition \ref{defn:prec}. We have
\begin{equation}
\label{eq:KPZ_tail_prob}
\prob\left(\bigcap_{\ell=1}^m \{ \rH(x_\ell, t_\ell) \ge h_\ell \}\right)
= (-1)^m\oint_{>1} \cdots \oint_{>1} \dD(\bh;\bz) \prod_{\ell=1}^{m-1}\ddbar{z_\ell}{z_\ell(1-z_\ell)},
\end{equation}
where 
 \begin{equation}
 \label{eq:dD_summation}
        \begin{split}
            \dD(\bh;\bz)=\dD(\bh;\bz;(x_1,t_1),\cdots,(x_m,t_m)):=\sum_{\substack{n_\ell\ge 1\\ \ell=1,\cdots,m}}\frac{1}{(n_1!\cdots n_{m}!)^2} \dD_{\bn}(\bz,\bh),
        \end{split}
    \end{equation}
    with
    \begin{equation}
    \label{eq:def_dD_bn}
        \begin{split}
            \dD_{\bn}(\bh;\bz)=&
            \dD_{\bn}(\bh;\bz;(x_1,t_1),\cdots,(x_m,t_m))
            \\
            :=&
            \prod_{\ell=1}^{m-1}(1-z_\ell)^{n_\ell}(1-z_\ell^{-1})^{n_{\ell+1}}\\
            &\cdot
            \prod_{\ell=2}^{m} \prod_{i_\ell=1}^{n_\ell} \left(\frac{1}{1-z_{\ell-1}}\int_{C_{\ell,\LL}^\inn}\ddbar{\xi_{i_\ell}^{(\ell)}}{}-\frac{z_{\ell-1}}{1-z_{\ell-1}}\int_{C_{\ell,\LL}^\out}\ddbar{\xi_{i_\ell}^{(\ell)}}{}\right)\prod_{i_1=1}^{n_1}\int_{C_{1,\LL}}\ddbar{\xi_{i_1}^{(1)}}{}\\
	    &\cdot \prod_{\ell=2}^{m} \prod_{i_\ell=1}^{n_\ell} \left(\frac{1}{1-z_{\ell-1}}\int_{C_{\ell,\RR}^\inn}\ddbar{\eta_{i_\ell}^{(\ell)}}{}-\frac{z_{\ell- 
            1}}{1-z_{\ell-1}}\int_{C_{\ell,\RR}^\out}\ddbar{\eta_{i_\ell}^{(\ell)}}{}\right)\prod_{i_1=1}^{n_1}\int_{C_{1,\RR}}\ddbar{\eta_{i_1}^{(1)}}{}\\ 
            &\quad \rC(\bseta^{(1)};\bxi^{(1)})\cdot \prod_{\ell=1}^{m} \rC(\bxi^{(\ell)}\bunion \bseta^{(\ell+1)};\bseta^{(\ell)}\bunion \bxi^{(\ell+1)})
            \cdot \prod_{\ell=1}^{m}\prod_{i_\ell=1}^{n_\ell}\frac{F_\ell(\xi_{i_\ell}^{(\ell)})}{F_\ell(\eta_{i_\ell}^{(\ell)})}
        \end{split} 
    \end{equation}
    where the notations $\bxi_\ell=(\xi_1^{(\ell)},\cdots,\xi_{n_\ell}^{(\ell)})$ and $\bseta_\ell = (\eta_{1}^{(\ell)},\cdots,\eta_{n_\ell}^{(\ell)})$. We also set $\bxi^{(m+1)}=\bseta^{(m+1)}=\emptyset$. The function $\rC$ represents the Cauchy determinant defined in \eqref{eq:def_rC}, and the functions
    \begin{equation}
    \label{eq:def_F}
    F_\ell(\zeta) = F_\ell(\zeta;\bh;(x_1,t_1),\cdots,(x_m,t_m)):= 
    \begin{cases}
        e^{-\frac13 t_1 \zeta^3 +x_1 \zeta^2 +h_1\zeta}, & \ell=1,\\
        e^{-\frac{1}{3}(t_{\ell}-t_{\ell-1})\zeta^3 + (x_\ell -x_{\ell-1})\zeta^2 + (h_\ell -h_{\ell-1})\zeta }, & 2\le \ell \le m.
    \end{cases}
    \end{equation}
    The integration contours are given as follows. The contours $C_{m,\LL}^{\inn},\cdots,C_{2,\LL}^{\inn},C_{1,\LL},C_{2,\LL}^{\out},\cdots, C_{m,\LL}^{\out}$ are contours on the left half plane $\{\zeta\in\complexC: \Re\zeta<0\}$ ordered from left to right. Each of these contours goes from $\infty e^{-\ii 2\pi/3}$ to $\infty e^{\ii 2\pi/3}$. Similarly, the contours $C_{m,\RR}^{\inn},\cdots,C_{2,\RR}^{\inn},C_{1,\RR},C_{2,\RR}^{\out},\cdots, C_{m,\RR}^{\out}$ are contours on the right half plane $\{\zeta\in\complexC: \Re\zeta>0\}$ ordered from right to left. Each of these contours goes from $\infty e^{-\ii  \pi/5}$ to $\infty e^{\ii  \pi/5}$. See Figure \ref{fig:contours2} for an illustration of the contours when $m=2$.
\end{prop}

   \begin{figure}
    \centering
    \begin{tikzpicture}[scale=1.8, decoration={
  markings,
  mark=at position 0.75 with {\arrow{>}}}
  ] 
  \draw[-latex] (-3,0) -- (3,0) node[right] {$\Re$};
  \draw[-latex] (0,-2) -- (0,2) node[above] {$\Im$};

  \draw[postaction={decorate},blue,thick] plot[smooth] coordinates {(-2.5,-1) (-2,-0.5) (-1.5,0) (-2,0.5) (-2.5,1)};
  \draw[postaction={decorate},blue,thick] plot[smooth] coordinates {(-2,-1) (-1.5,-0.5) (-1,0) (-1.5,0.5) (-2,1)};
  \draw[postaction={decorate},blue,thick] plot[smooth] coordinates {(-1.5,-1) (-1,-0.5) (-0.5,0) (-1,0.5) (-1.5,1)};

  \draw[postaction={decorate},red,thick] plot[smooth] coordinates {(1.5,-0.5) (1,-0.25) (0.5,0) (1,0.25) (1.5,0.5)};
  \draw[postaction={decorate},red,thick] plot[smooth] coordinates {(2,-0.5) (1.5,-0.25) (1,0) (1.5,0.25) (2,0.5)};
  \draw[postaction={decorate},red,thick] plot[smooth] coordinates {(2.5,-0.5) (2,-0.25) (1.5,0) (2,0.25) (2.5,0.5)};
  
  \node at (-2.75,1.2) {$C_{2,\LL}^{\inn}$};
  \node at (-2.0,1.2) {$C_{1,\LL}$};
  \node at (-1.25,1.2) {$C_{2,\LL}^\out$};

  \node at (1.25,0.6) {$C_{2,\RR}^\out$};
  \node at (2.0,0.6) {$C_{1,\RR}$};
  \node at (2.75,0.6) {$C_{2,\RR}^{\inn}$};
\end{tikzpicture}
    \caption{Illustration of the contours in $\dD_{\bn}(\bh;\bz)$ when $m=2$.}
    \label{fig:contours2}
\end{figure}

\begin{proof}[Proof of Proposition \ref{prop:KPZ_tail_prob}]
We will need a variation of the multi-dimensional distribution function formula in \cite{Liu19} shown below. This variation was first explicitly written down in \cite[Equation (2.4)]{Liu-Wang22}. 
\begin{equation}
\label{eq:aux_01}
\begin{split}
&\prob\left(\bigcap_{\ell=1}^{m-1} \{ \rH(x_\ell, t_\ell) \ge h_\ell \} \bigcap\{ \rH(x_m,t_m) \le h_m\}\right)\\
&=(-1)^{m-1}\oint_{>1} \cdots \oint_{>1}\sum_{\substack{n_\ell\ge 0\\ \ell=1,\cdots,m}}\frac{1}{(n_1!\cdots n_{m}!)^2} \dD_{\bn}(\bh;\bz) \prod_{\ell=1}^{m-1}\ddbar{z_\ell}{z_\ell(1-z_\ell)}.
\end{split}
\end{equation}
Note the summation allows some $n_\ell$ to be zero. The function $\dD_{\bn}(\bh;\bz)$ is defined in \eqref{eq:def_dD_bn}\footnote{The original formula has the usual choice of contours in the definition of $\dD_{\bn}(\bh;\bz)$ where the angles going to infinity are $\pm 2\pi/3$ (for $C_{\ell,\LL}$ contours) and $\pm \pi/3$ (for $C_{\ell,\RR}$ contours). When some times become equal, then the contours need to be bent according to the order of $(x_\ell,\tau_\ell)$ under $\prec$ to ensure the decay of the integrand (see the discussions after the Definition 2.25 in \cite{Liu19}). In this paper, we bend the contours at the beginning so that the integral is well defined in the definition of $\dD_{\bn}(\bh;\bz)$ for all ordered points $(x_\ell,t_\ell)$ under the ordering $\prec$.}, and when some $n_\ell=0$ we should view the corresponding empty product as $1$. One further observation is that the summand in \eqref{eq:aux_01} vanishes if any $1\le \ell\le m-1$ satisfies $n_\ell=0$ since the $z_{\ell}$ integral equals to zero (by deforming the $z_\ell$ contour to infinity). This observation was also made in \cite[Section 3.2]{Liu-Wang22} and \cite[Lemma 2.3]{baik2024pinchedup}. Thus the right hand side of \eqref{eq:aux_01} satisfies the following relation
\begin{equation}
\label{eq:aux_02}
\begin{split}
&(-1)^{m-1}\oint_{>1} \cdots \oint_{>1}\sum_{\substack{n_\ell\ge 0\\ \ell=1,\cdots,m}}\frac{1}{(n_1!\cdots n_{m}!)^2} \dD_{\bn}(\bh;\bz) \prod_{\ell=1}^{m-1}\ddbar{z_\ell}{z_\ell(1-z_\ell)} \\
&= (-1)^{m-1}\oint_{>1} \cdots \oint_{>1}\sum_{\substack{n_\ell\ge 1\\ \ell=1,\cdots,m-1}}\frac{1}{(n_1!\cdots n_{m-1}!)^2} \dD_{\hat\bn}(\hat\bh;\hat\bz;(x_1,t_1),\cdots,(x_{m-1},t_{m-1})) \prod_{\ell=1}^{m-2}\ddbar{z_\ell}{z_\ell(1-z_\ell)} \\
&\quad-(-1)^{m}\oint_{>1} \cdots \oint_{>1}\sum_{\substack{n_\ell\ge 1\\ \ell=1,\cdots,m}}\frac{1}{(n_1!\cdots n_{m}!)^2} \dD_{\bn}(\bh;\bz) \prod_{\ell=1}^{m-1}\ddbar{z_\ell}{z_\ell(1-z_\ell)}, 
\end{split}
\end{equation}
where the first sum comes from the case when $n_m=0$ and the second sum comes from the case when $n_m\ge 1$, and the notations $\hat\bn,\hat \bh,\hat \bz$ represent the vectors $\bn, \bh, \bz$ after removing the last coordinate respectively. Note that the left hand side of \eqref{eq:aux_01} satisfies the simple relation
\begin{equation}
\begin{split}
\label{eq:aux_last}
&\prob\left(\bigcap_{\ell=1}^{m-1} \{ \rH(x_\ell, t_\ell) \ge h_\ell \} \bigcap\{ \rH(x_m,t_m) \le h_m\}\right)\\
&= \prob\left(\bigcap_{\ell=1}^{m-1} \{ \rH(x_\ell, t_\ell) \ge h_\ell \} \}\right) -\prob\left(\bigcap_{\ell=1}^{m} \{ \rH(x_\ell, t_\ell) \ge h_\ell \}\right).
\end{split}
\end{equation}
Compare the two equations \eqref{eq:aux_02} and \eqref{eq:aux_last}. And note that Proposition \ref{prop:KPZ_tail_prob} is a well known formula for $1-\FGUE$ when $m=1$, see \cite[Equation (23)]{Liu19}  for an example. The case for general $m$ then follows by a simple induction on $m$ using \eqref{eq:aux_02} and \eqref{eq:aux_last}.
\end{proof}

\bigskip

We also need a formula of the derivative of the tail probability formula. Note that Lemma \ref{lm:exchange_differentiation_sum} allows us to take the derivative inside the integral and we obtain the following proposition. 
\begin{prop}
    \label{prop:derivative_multipoint_tail}
    Assuming the same notations as in Proposition \ref{prop:KPZ_tail_prob}, we have
    \begin{equation}
        \frac{\partial}{\partial h_k} \prob\left(\bigcap_{\ell=1}^m \{ \rH(x_\ell, t_\ell) \ge h_\ell \}\right)
        = (-1)^m\oint_{>1} \cdots \oint_{>1} \sum_{\substack{n_\ell\ge 1\\ \ell=1,\cdots,m}}\frac{1}{(n_1!\cdots n_{m}!)^2} \frac{\partial}{\partial h_k}\dD_{\bn}(\bz,\bh) \prod_{\ell=1}^{m-1}\ddbar{z_\ell}{z_\ell(1-z_\ell)},
    \end{equation}
    for any $1\le k\le m$, where $\frac{\partial}{\partial h_k}\dD_{\bn}(\bz,\bh)$ has the same formula as $\dD_{\bn}(\bz,\bh)$ in \eqref{eq:def_dD_bn}, except that we need to replace $\prod_{i_\ell=1}^{n_\ell}\frac{F_\ell(\xi_{i_\ell}^{(\ell)})}{F_\ell(\eta_{i_\ell}^{(\ell)})}$ by $\frac{\partial}{\partial h_k}\prod_{i_\ell=1}^{n_\ell}\frac{F_\ell(\xi_{i_\ell}^{(\ell)})}{F_\ell(\eta_{i_\ell}^{(\ell)})}$.
\end{prop}

\subsection{Proof of Proposition \ref{prop:main}}
\label{sec:proof_proposition_main_sketch}

Now we prove Proposition \ref{prop:main}. When $m=1$, it follows from the known result of the right tail estimate of the GUE Tracy-Widom distribution. In fact, it is known that the one point distribution function of $\rH$ is given by the GUE Tracy-Widom distribution
\begin{equation}
    \label{eq:KPZ_GUE}
\prob\left( \rH(x,t) \le h \right) =\FGUE\left(\frac{h}{t^{1/3}} + \frac{x^2}{t^{4/3}}\right).
\end{equation}
It is also well known that, see \cite[Equation (25)]{BBD08} and \cite[Equation (3.4)]{liu2022geodesic} for example, the GUE Tracy-Widom distribution satisfies the following upper tail estimate
\begin{equation}
\label{eq:upper_tail_FGUE}
1- \FGUE(L) = \frac{e^{-\frac{4}{3}L^{3/2}}}{16\pi L^{3/2}} \left(1+ O(L^{-3/2})\right),\quad  \FGUE'(L) = \frac{e^{-\frac{4}{3}L^{3/2}}}{8\pi L}\left(1+ O(L^{-3/2})\right)
\end{equation}
as $L\to\infty$. Combining the above two results, we immediately obtain Proposition \ref{prop:main} when $m=1$.

Throughout the proof below, we always assume  $m\ge 2$.

We write, using Proposition \ref{prop:KPZ_tail_prob} and Proposition \ref{prop:derivative_multipoint_tail},
\begin{equation}
\label{eq:pre_limit_prob}
\begin{split}
&16\pi L^{3/2} e^{\frac{4}{3}L^{3/2}}\prob \left( \bigcap_{\ell=1}^m \left\{ \HH(\alpha_\ell,\tau_\ell)\ge \beta_\ell  \right\}\right)\\
&=  (-1)^m\oint_{>1} \cdots \oint_{>1} \sum_{\substack{n_\ell\ge 1\\ \ell=1,\cdots,m}}\frac{1}{(n_1!\cdots n_{m}!)^2} 16\pi L^{3/2} e^{\frac{4}{3}L^{3/2}}\dD_{\bn}(\bh;\bz) \prod_{\ell=1}^{m-1}\ddbar{z_\ell}{z_\ell(1-z_\ell)}
\end{split}
\end{equation}
and 
\begin{equation}
    \label{eq:pre_limit_prob_derivative}
    \begin{split}
    &16\pi L^{3/2} e^{\frac{4}{3}L^{3/2}}\frac{\partial}{\partial \beta_k}\prob \left( \bigcap_{\ell=1}^m \left\{ \HH(\alpha_\ell,\tau_\ell)\ge \beta_\ell  \right\}\right)\\
    &=  (-1)^m\oint_{>1} \cdots \oint_{>1} \sum_{\substack{n_\ell\ge 1\\ \ell=1,\cdots,m}}\frac{1}{(n_1!\cdots n_{m}!)^2} 16\pi L^{3/2} e^{\frac{4}{3}L^{3/2}}\frac{\partial}{\partial \beta_k}\dD_{\bn}(\bh;\bz) \prod_{\ell=1}^{m-1}\ddbar{z_\ell}{z_\ell(1-z_\ell)}
    \end{split}
    \end{equation}
with the notations the same as in Proposition \ref{prop:KPZ_tail_prob} but with the parameters chosen as follows
\begin{equation}
\label{eq:parameter_scale}
x_\ell= \alpha_\ell L^{-1},\quad t_\ell = 1+\tau_\ell L^{-3/2},\quad h_\ell = L + \beta_\ell L^{-1/2},\qquad \ell=1,\cdots,m.
\end{equation}

It turns out that each term in the summand converges as $L\to\infty$. More precisely, for $n_1>1$ we define
\begin{equation}
\pD_{\bn}(\bbeta;\bz)=\pD_{\bn}(\bbeta;\bz;(\alpha_1,\tau_1),\cdots,(\alpha_m,\tau_m)):=0,
\end{equation}
and for $n_1=1$ we define
\begin{equation}
\label{eq:def_pD}
\begin{split}
             \pD_{\bn}(\bbeta;\bz)
             =&\pD_{\bn}(\bbeta;\bz;(\alpha_1,\tau_1),\cdots,(\alpha_m,\tau_m))\\
             :=&2\prod_{\ell=1}^{m-1} (1-z_\ell)^{n_\ell}(1-z_\ell^{-1})^{n_{\ell+1}}\\
             & \cdot \prod_{\ell=2}^{m} \prod_{i_\ell=1}^{n_\ell} \left(\frac{1}{1-z_{\ell-1}}\int_{\hat\Gamma_{\ell,\LL}^\inn}\ddbar{u_{i_\ell}^{(\ell)}}{}-\frac{z_{\ell-1}}{1-z_{\ell-1}}\int_{\hat\Gamma_{\ell,\LL}^\out}\ddbar{u_{i_\ell}^{(\ell)}}{}\right) \int_{\hat\Gamma_{\LL}}e^{u^2}\frac{\rd u}{\sqrt{\pi}\ii}\\
             &\cdot \prod_{\ell=2}^{m} \prod_{i_\ell=1}^{n_\ell} \left(\frac{1}{1-z_{\ell- 
            1}}\int_{\hat\Gamma_{\ell,\RR}^\inn}\ddbar{v_{i_\ell}^{(\ell)}}{}-\frac{z_{\ell-1}} 
        {1-z_{\ell-1}}\int_{\hat\Gamma_{\ell,\RR}^\out}\ddbar{v_{i_\ell}^{(\ell)}}{}\right) \int_{\ii\realR}e^{v^2}\frac{\rd v}{\sqrt{\pi}\ii}\\ 
        \\
       &\quad \rC(-1\bunion V^{(2)};1\bunion U^{(2)}) \cdot \prod_{\ell=2}^{m} \rC(U^{(\ell)}\bunion V^{(\ell+1)};V^{(\ell)}\bunion U^{(\ell+1)}) \cdot \prod_{\ell=2}^{m} \prod_{i_\ell=1}^{n_\ell} \frac{ f_{\ell}(u_{i_\ell}^{(\ell)})}{ f_{\ell}(v_{i_\ell}^{(\ell)})} \cdot \frac{f_1(-1)}{f_1(1)} 
        \end{split}
\end{equation}
where the functions $f_\ell$ are defined in \eqref{eq:def_f}, and for notation convenience we set $n_{m+1}=0$ and $U^{(m+1)}=V^{(m+1)}=\emptyset$. The integration contours are chosen specifically in the following way. Let $\hat\Gamma_{\LL}:= \{r e^{\pm\ii2\pi/3}: r\ge 0\}$ with orientation from $\infty e^{-\ii2\pi/3}$ to $\infty e^{\ii2\pi/3}$, and
\begin{equation}
\hat\Gamma_{\RR}:= \left\{y\ii: -\cot \frac{2\pi}{5} <y <\cot \frac{2\pi}{5} \right\} \cup \left\{ \ii\cdot \cot\frac{2\pi}{5} + re^{\ii \pi/5}: r\ge 0\right\} \cup \left\{ -\ii\cdot \cot \frac{2\pi}{5} + re^{-\ii \pi/5}: r\ge 0\right\}
\end{equation}
with orientation from $\infty e^{-\ii \pi/5}$ to $\infty e^{\ii \pi/5}$. Then we define
\begin{equation}
\hat\Gamma_{\ell,\LL}^\inn = -1-a_\ell +\hat\Gamma_{\LL}, \quad \hat\Gamma_{\ell,\LL}^\out = -1+a_\ell +\hat\Gamma_{\LL}, 2\le \ell\le m,
\end{equation}
and
\begin{equation}
\hat\Gamma_{\ell,\RR}^\inn = 1+a_\ell +\hat\Gamma_{\RR}, \quad \hat\Gamma_{\ell,\RR}^\out =  1-a_\ell +\hat\Gamma_{\RR}, 2\le \ell\le m,
\end{equation}
where $a_2,\cdots,a_m$ are constants satisfying $0<a_2<\cdots<a_m<1$. Note that the contours $\hat\Gamma_{\ell,\diamond}^{\star}$, $2\le \ell\le m$, $\diamond\in\{\LL,\RR\}$, $\star\in\{\inn,\out\}$ can also be viewed as the contours $\Gamma_{\ell,\diamond}^{\star}$ appeared in the Definition \ref{defn:FT} for the function $\FT$, with a more specific setting here to simplify our asymptotic analysis later. We also note that 
\begin{equation}
    \int_{\hat\Gamma_{\LL}}e^{u^2}\frac{\rd u}{\sqrt{\pi}\ii}=\int_{\ii\realR}e^{v^2}\frac{\rd v}{\sqrt{\pi}\ii}=1.
\end{equation}
Thus we have, by comparing  \eqref{eq:def_pD} with \eqref{eq:def_cdnz},
\begin{equation}
    \pD_{\bn}(\bbeta;\bz) =\cD_{\bn}(\bbeta;\bz)
\end{equation}
when $n_1=1$. Proposition \ref{prop:KPZ_tail_prob} hence follows from the following two lemmas, Proposition \ref{prop:differentiability}, and the dominated convergence theorem.

\begin{lm}
\label{lm:convergence_ptwise}
Assume that $\bn,\bz,\balpha,\btau,\bbeta$ are fixed, and $\bx,\bt,\bh$ depend on $L$ as in \eqref{eq:parameter_scale}. Then we have 
\begin{equation}
    \label{eq:aux_12_28_01}
16\pi L^{3/2} e^{\frac{4}{3}L^{3/2}}\dD_{\bn}(\bh,\bz;(x_1,t_1),\cdots,(x_m,t_m)) \to \pD_{\bn}(\bbeta;\bz;(\alpha_1,\tau_1),\cdots,(\alpha_m,\tau_m))
\end{equation}
and
\begin{equation}
    \label{eq:aux_12_28_02}
    16\pi L^{3/2} e^{\frac{4}{3}L^{3/2}}\frac{\partial}{\partial \beta_k}\dD_{\bn}(\bh,\bz;(x_1,t_1),\cdots,(x_m,t_m)) \to \frac{\partial}{\partial \beta_k}\pD_{\bn}(\bbeta;\bz;(\alpha_1,\tau_1),\cdots,(\alpha_m,\tau_m))
\end{equation}
as $L\to \infty$.
\end{lm}

\begin{lm}
\label{lm:uniform_bound}
For sufficiently large $L$, we have
\begin{equation}
\label{eq:aux_11_28_05}
\left| L^{3/2} e^{\frac{4}{3}L^{3/2}}\dD_{\bn}(\bn;\bz) \right| \le  C^{n_1+\cdots+n_m}\prod_{\ell=1}^{m-1} \frac{(1+|z_\ell|)^{n_{\ell+1}}}{|1-z_\ell|^{n_{\ell+1}-n_\ell}|z_\ell|^{n_{\ell+1}}}\cdot \prod_{\ell=1}^m n_\ell^{n_\ell}
\end{equation}
and
\begin{equation}
    \label{eq:aux_12_28}
    \left| L^{3/2} e^{\frac{4}{3}L^{3/2}}\frac{\partial}{\partial \beta_k}\dD_{\bn}(\bn;\bz) \right| \le  C^{n_1+\cdots+n_m}\prod_{\ell=1}^{m-1} \frac{(1+|z_\ell|)^{n_{\ell+1}}}{|1-z_\ell|^{n_{\ell+1}-n_\ell}|z_\ell|^{n_{\ell+1}}}\cdot \prod_{\ell=1}^m n_\ell^{n_\ell}
    \end{equation}
where $C$ is a positive constant independent of $\bz$. 
\end{lm}

The proof of these two lemmas are given in Section \ref{sec:asympt}.

\subsection{Asymptotic analysis}
\label{sec:asympt}
In this subsection, we perform the asymptotic analysis and prove Lemmas \ref{lm:convergence_ptwise} and \ref{lm:uniform_bound}.

We deform the contours in \eqref{eq:def_dD_bn} in the following way
\begin{equation}
\label{eq:contour_change1}
C_{\ell,\diamond}^{\star} = \sqrt{L} \hat\Gamma_{\ell,\diamond}^{\star},\quad \text{ for } 2\le \ell\le m, \diamond\in\{\LL,\RR\}, \star\in\{\inn,\out\},
\end{equation}
and
\begin{equation}
\label{eq:contour_change2}
C_{1,\LL} = -\sqrt{L} + L^{-1/4} \hat\Gamma_{\LL},\quad C_{1,\RR}= \sqrt{L}+ L^{1/2} \hat \Gamma_\RR = \sqrt{L} + L^{-1/4}(L^{3/4} \hat\Gamma_\RR).
\end{equation}
Note that the ordering of the $\hat\Gamma$-contours ensures that the $C$-contours are still of the same order and the deformation will not pass any poles of the integrand. The variables are changed accordingly
\begin{equation}
\label{eq:change_variables1}
\xi_{i_\ell}^{(\ell)} = \sqrt{L} u_{i_\ell}^{(\ell)}, \quad \eta_{i_\ell}^{(\ell)}=\sqrt{L} v_{i_\ell}^{(\ell)},\quad 1\le i_\ell \le n_\ell, \quad 2\le \ell \le m
\end{equation}
and
\begin{equation}
\label{eq:change_variables2}
\xi_{i_1}^{(1)}=-\sqrt{L}+ L^{-1/4} u_{i_1}^{(1)},\quad  \eta_{i_1}^{(1)} = \sqrt{L}+ L^{-1/4} v_{i_1}^{(1)},\quad 1\le i_1\le n_1.
\end{equation}
This change leads to 
\begin{equation}
\label{eq:prelimit_dD}
\begin{split}
&16\pi L^{3/2} e^{\frac{4}{3}L^{3/2}}\dD_{\bn}(\bh;\bz)\\
=&16\pi 
            \prod_{\ell=2}^{m} \prod_{i_\ell=1}^{n_\ell} \left(\frac{1}{1-z_{\ell-1}}\int_{\hat\Gamma_{\ell,\LL}^\inn}\ddbar{u_{i_\ell}^{(\ell)}}{}-\frac{z_{\ell-1}}{1-z_{\ell-1}}\int_{\hat\Gamma_{\ell,\LL}^\out}\ddbar{u_{i_\ell}^{(\ell)}}{}\right)\prod_{i_1=1}^{n_1}\int_{\hat\Gamma_{\LL}}\ddbar{u_{i_1}^{(1)}}{}\\
	    &\cdot \prod_{\ell=2}^{m} \prod_{i_\ell=1}^{n_\ell} \left(\frac{1}{1-z_{\ell-1}}\int_{\hat\Gamma_{\ell,\RR}^\inn}\ddbar{v_{i_\ell}^{(\ell)}}{}-\frac{z_{\ell- 
            1}}{1-z_{\ell-1}}\int_{\hat\Gamma_{\ell,\RR}^\out}\ddbar{v_{i_\ell}^{(\ell)}}{}\right)\prod_{i_1=1}^{n_1}\int_{L^{3/4}\hat\Gamma_{\RR}}\ddbar{v_{i_1}^{(1)}}{}\\ 
 &\quad L^{n_1/2}\rC(\bseta^{(1)};\bxi^{(1)})\cdot \prod_{\ell=1}^{m} L^{(n_\ell+n_{\ell+1})/2} \rC(\bxi^{(\ell)}\bunion \bseta^{(\ell+1)};\bseta^{(\ell)}\bunion \bxi^{(\ell+1)})\\
 &\quad \cdot \prod_{\ell=2}^{m}\prod_{i_\ell=1}^{n_\ell}\frac{f_\ell(u_{i_\ell}^{(\ell)})}{f_\ell(v_{i_\ell}^{(\ell)})} \cdot \prod_{i_1=1}^{n_1} \frac{e^{\frac{2}{3}L^{3/2}}F_1(\xi_{i_1}^{(1)})}{e^{-\frac{2}{3}L^{3/2}}F_1(\eta_{i_1}^{(1)})} \cdot L^{3/2(1-n_1)}e^{-\frac43(n_1-1)L^{3/2}}\prod_{\ell=1}^{m-1}(1-z_\ell)^{n_\ell}(1-z_\ell^{-1})^{n_{\ell+1}}
\end{split}
\end{equation}
where 
in the integrand the notations $\bxi^{(\ell)}=(\xi_1^{(\ell)},\cdots,\xi_{n_\ell}^{(\ell)}),\bseta=(\eta_1^{(\ell)},\cdots,\eta_{n_\ell}^{(\ell)})$ 
are functions of $u_{i_\ell}^{(\ell)}$ and $v_{i_\ell}^{(\ell)}$ determined by  \eqref{eq:change_variables1} and \eqref{eq:change_variables2}.  Moreover, 
\begin{equation}
    \label{eq:aux_12_28_03}
    \begin{split}
    &16\pi L^{3/2} e^{\frac{4}{3}L^{3/2}} \frac{\partial}{\partial \beta_k}\dD_{\bn}(\bh;\bz)\\
=&16\pi 
            \prod_{\ell=2}^{m} \prod_{i_\ell=1}^{n_\ell} \left(\frac{1}{1-z_{\ell-1}}\int_{\hat\Gamma_{\ell,\LL}^\inn}\ddbar{u_{i_\ell}^{(\ell)}}{}-\frac{z_{\ell-1}}{1-z_{\ell-1}}\int_{\hat\Gamma_{\ell,\LL}^\out}\ddbar{u_{i_\ell}^{(\ell)}}{}\right)\prod_{i_1=1}^{n_1}\int_{\hat\Gamma_{\LL}}\ddbar{u_{i_1}^{(1)}}{}\\
	    &\cdot \prod_{\ell=2}^{m} \prod_{i_\ell=1}^{n_\ell} \left(\frac{1}{1-z_{\ell-1}}\int_{\hat\Gamma_{\ell,\RR}^\inn}\ddbar{v_{i_\ell}^{(\ell)}}{}-\frac{z_{\ell- 
            1}}{1-z_{\ell-1}}\int_{\hat\Gamma_{\ell,\RR}^\out}\ddbar{v_{i_\ell}^{(\ell)}}{}\right)\prod_{i_1=1}^{n_1}\int_{L^{3/4}\hat\Gamma_{\RR}}\ddbar{v_{i_1}^{(1)}}{}\\ 
 &\quad L^{n_1/2}\rC(\bseta^{(1)};\bxi^{(1)})\cdot \prod_{\ell=1}^{m} L^{(n_\ell+n_{\ell+1})/2} \rC(\bxi^{(\ell)}\bunion \bseta^{(\ell+1)};\bseta^{(\ell)}\bunion \bxi^{(\ell+1)}) \cdot H_k\\
 &\quad \cdot \prod_{\ell=2}^{m}\prod_{i_\ell=1}^{n_\ell}\frac{f_\ell(u_{i_\ell}^{(\ell)})}{f_\ell(v_{i_\ell}^{(\ell)})} \cdot \prod_{i_1=1}^{n_1} \frac{e^{\frac{2}{3}L^{3/2}}F_1(\xi_{i_1}^{(1)})}{e^{-\frac{2}{3}L^{3/2}}F_1(\eta_{i_1}^{(1)})} \cdot L^{3/2(1-n_1)}e^{-\frac43(n_1-1)L^{3/2}}\prod_{\ell=1}^{m-1}(1-z_\ell)^{n_\ell}(1-z_\ell^{-1})^{n_{\ell+1}}
    \end{split}
\end{equation}
where 
\begin{equation}
    H_k=\begin{dcases}
        \sum_{i_m=1}^{n_m} (u_{i_m}^{(m)} - v_{i_m}^{(m)}), & k=m,\\
        \sum_{i_k=1}^{n_k} (u_{i_k}^{(k)} - v_{i_k}^{(k)}) - \sum_{i_{k+1}=1}^{n_{k+1}} (u_{i_{k+1}}^{(k+1)} - v_{i_{k+1}}^{(k+1)}), & 2\le k\le m-1,\\
        \sum_{i_1=1}^{n_1}(-2+ L^{-3/4}(u_{i_1}^{(1)} -v_{i_1}^{(1)})) -\sum_{i_2=1}^{n_2} (u_{i_2}^{(2)} - v_{i_2}^{(2)}), & k=1.
    \end{dcases}
\end{equation}
\subsubsection{Some inequalities}
We will need some inequalities to bound the integrand and apply the dominated convergence theorem.

Note that our choice of contours in \eqref{eq:contour_change1} and \eqref{eq:contour_change2} implies that
\begin{equation}
\label{eq:aux_11_28_02}
\dist(C_{\ell,\LL}^\inn \cup C_{\ell,\LL}^\out \cup C_{\ell+1,\RR}^\inn \cup C_{\ell+1,\RR}^\out; C_{\ell,\RR}^\inn \cup C_{\ell,\RR}^\out \cup C_{\ell+1,\LL}^\inn \cup C_{\ell+1,\LL}^\out) \ge c_1\sqrt{L},\quad 2\le \ell \le m-1,
\end{equation}
for some positive constant $c_1$, and
\begin{equation}
\label{eq:aux_11_28_03}
\dist(C_{1,\LL}  \cup C_{2,\RR}^\inn \cup C_{2,\RR}^\out; C_{1,\RR} \cup C_{2,\LL}^\inn \cup C_{2,\LL}^\out) \ge c_2\sqrt{L}
\end{equation}
for some positive constant $c_2$. Therefore Lemma \ref{lm:bounds_Cauchy_det} implies
\begin{equation}
\label{eq:bounds_Cauchy1}
|L^{n_1/2}\rC(\bseta^{(1)};\bxi^{(1)})| \le (1/c_2)^{n_1}n_1^{n_1/2},
\end{equation}
and
\begin{equation}
\label{eq:bounds_Cauchy2}
|L^{(n_\ell+n_{\ell+1})/2} \rC(\bxi^{(\ell)}\bunion \bseta^{(\ell+1)};\bseta^{(\ell)}\bunion \bxi^{(\ell+1)})| \le (\sqrt{2}/c_1)^{n_\ell+n_{\ell+1}}n_\ell^{n_\ell/2}n_{\ell+1}^{n_{\ell+1}/2},\quad \ell=1,\cdots,m.
\end{equation}

\bigskip
We also need the following bound for the function $F_1$. Recall $F_1(\zeta)=e^{-\frac{1}{3}t_1\zeta^3+x_1\zeta^2+ h_1\zeta}$ defined in \eqref{eq:def_F}. Also note that the parameters are scaled as in \eqref{eq:parameter_scale}.
\begin{lm}
\label{lm:bound_F1}
For all $u\in\hat\Sigma_{\LL}$, we have
\begin{equation}
\label{eq:aux_1_12_01}
|e^{2/3L^{3/2}-(u^2 + \alpha_1-\beta_1+\frac13\tau_1)}F_1(-\sqrt{L}+L^{-1/4}u)|\le e^{c_3L^{-3/4}}
\end{equation}
for sufficiently large $L$, where $c_3$ and $c_4$ are constants independent on $u$ and $L$.
When $v\in \ii\realR$, we have
\begin{equation}
\label{eq:aux_1_12_03}
|e^{-2/3L^{3/2}}F_1(\sqrt{L}+L^{-1/4}v)| \ge e^{-(1-c'_3)v^2 + \alpha_1+\beta_1-\frac13\tau_1}
\end{equation}
for sufficiently large $L$, and when $v=L^{3/4}( \pm\ii\cot \frac{2\pi}{5} +r e^{\pm\ii \pi/5})$ and $r\ge 0$ we have
\begin{equation}
\label{eq:aux_1_12_06}
|F_1(\sqrt{L} + L^{-1/4}v)| \ge e^{ 2/3L^{3/2}+c'_4L^{3/2}(1+r)}
\end{equation}
for sufficiently large $L$. Here the constants $c'_3\in(0,1)$ and $c'_4>0$ are independent of $v$ and $L$. Moreover, if $u\in\hat\Sigma_{\LL} $ is fixed, we have
\begin{equation}
\label{eq:aux_1_12_02}
\lim_{L\to\infty} e^{2/3L^{3/2}}F_1(-\sqrt{L}+L^{-1/4}u)=e^{u^2 + \alpha_1-\beta_1+\frac13\tau_1},
\end{equation}
and similarly, if $v\in\ii\realR$ is fixed, we have
\begin{equation}
\label{eq:aux_1_12_04}
\lim_{L\to\infty}e^{-2/3L^{3/2}}F_1(\sqrt{L}+L^{-1/4}v)=e^{ -v^2 + \alpha_1+\beta_1-\frac13\tau_1 }.
\end{equation}
\end{lm}
\begin{proof}[Proof of Lemma \ref{lm:bound_F1}]
The proof is based on a direct calculation and the fact that the exponent decays or grows super-exponentially along the contour $\hat\Sigma_\LL$ or $\hat\Sigma_\RR$ respectively. 

We first prove  \eqref{eq:aux_1_12_01} and \eqref{eq:aux_1_12_02}. By inserting \eqref{eq:parameter_scale} and simplifying the exponent, we have
\begin{equation}
\label{eq:aux_11_28}
\begin{split}
&e^{2/3L^{3/2}-(u^2 + \alpha_1-\beta_1+\frac13\tau_1)}F_1(-\sqrt{L}+L^{-1/4}u)\\
&=\exp\left(L^{-3/4}\left( -\frac13(1+\tau_1L^{-3/2})u^3 +(\tau_1+\alpha_1)L^{-3/4}u^2 +(\beta_1-\tau_1-2\alpha_1)u \right)\right).
\end{split}
\end{equation}
Recall the contour $\hat\Sigma_\LL=\{re^{\pm \ii 2\pi/3}: r\ge 0\}$. Along this contour, it is easy to see that $\Re(u^3)$ grows to $+\infty$ as a cubic function. Thus the real part of $-\frac13(1+\tau_1L^{-3/2})u^3 +(\tau_1+\alpha_1)L^{-3/4}u^2 +(\beta_1-\tau_1-2\alpha_1)u$ is uniformly bounded  by a constant. This constant can be chosen independent of $L$. Thus the right hand side of \eqref{eq:aux_11_28} is bounded by a constant $c_3$ independent of $L$. This implies \eqref{eq:aux_1_12_01}. The equation \eqref{eq:aux_11_28} also implies \eqref{eq:aux_1_12_02} for fixed $u\in \hat\Sigma_\LL$.

Now we consider the other cases. If $v\in \ii\realR$, we similarly obtain,
\begin{equation}
\label{eq:aux_1_12_05}
e^{-2/3L^{3/2}}F_1(\sqrt{L}+L^{-1/4}v) = e^{-\left(1+(\tau_1-\alpha_1)L^{-3/2}\right)v^2 + \alpha_1+\beta_1-\frac13\tau_1 +vL^{-3/4}(-\frac13v^2-\tau_1+2\alpha_1+\beta_1-\frac13\tau_1v^2 L^{-3/2})}.
\end{equation}
Note that the last term in the exponent $vL^{-3/4}(-\frac13v^2-\tau_1+2\alpha_1+\beta_1-\frac13\tau_1v^2 L^{-3/2})\in\ii\realR$. Thus
\begin{equation}
|e^{-2/3L^{3/2}}F_1( \sqrt{L}+L^{-1/4}v)| = e^{-\left(1+(\tau_1-\alpha_1)L^{-3/2}\right)v^2 + \alpha_1+\beta_1-\frac13\tau_1} \ge e^{-(1-c'_3)v^2 + \alpha_1+\beta_1-\frac13\tau_1}
\end{equation}
for all $v\in\ii\realR$ and some positive constant $c'_3$ when $L$ is sufficiently large. This implies \eqref{eq:aux_1_12_03}. We also note that \eqref{eq:aux_1_12_04} follows from \eqref{eq:aux_1_12_05}.

It remains to show \eqref{eq:aux_1_12_06}. Note that the two different cases of $v$ lead to the same formula since the conjugation of a function doesn't change its norm. We only consider the case when $v$ is in the upper half plane $v=L^{3/4}( \ii\cot \frac{2\pi}{5} +r e^{ \ii \pi/5})$. In this case, by a direct computation, we have
\begin{equation}
\begin{split}
&|e^{-2/3L^{3/2}}F_1( \sqrt{L}+L^{-1/4}v)|\\
&=e^{[L^{3/2}(\frac23+\cot^2\frac{2\pi}{5})+O(1)] +r[L^{3/2}(2\cot\frac{2\pi}{5}\sin\frac{\pi}{5} +\cot^2\frac{2\pi}{5}\cos\frac{\pi}{5})+O(1)] +r^2\cdot O(1) +r^3[-\frac{1}{3}L^{3/2}\cos\frac{3\pi}{5}+O(1)]}
\end{split}
\end{equation}
where $O(1)$ are constant terms which are independent of both $r$ and $L$. \eqref{eq:aux_1_12_06} follows immediately.

\end{proof}

\subsubsection{Proof of Lemma \ref{lm:convergence_ptwise} and Lemma \ref{lm:uniform_bound}}

We first prove Lemma \ref{lm:convergence_ptwise}.

Assume $U^{(\ell)}=(u_{1}^{(\ell)},\cdots,u_{n_\ell}^{(\ell)})$ and $V^{(\ell)}=(v_{1}^{(\ell)},\cdots,v_{n_\ell}^{(\ell)})$, $1\le \ell \le m$ are fixed at the moment.

Recall the definition of the Cauchy determinants in \eqref{eq:def_rC} and \eqref{eq:id_rC}. It is easy to check that
\begin{equation}
L^{\frac{1}{2}(n_\ell+n_{\ell+1})} \rC(\bxi^{(\ell)}\bunion \bseta^{(\ell+1)};\bseta^{(\ell)}\bunion \bxi^{(\ell+1)})
=
\rC(U^{(\ell)}\bunion V^{(\ell+1)};V^{(\ell)}\bunion U^{(\ell+1)}),\qquad \ell\ge 2,
\end{equation}
where we set $n_{\ell+1}=0$ and hence $U^{(m+1)}$ and $V^{(m+1)}$ are both empty vectors.
When $\ell=1$, noting that all $\xi_{i_1}^{(1)}$ are around $-\sqrt{L}$ and all $\eta_{i_1}^{(1)}$ are around $\sqrt{L}$ when $L$ becomes large, $1\le i_1\le n_1$, we have the asymptotic
\begin{equation}
\lim_{L\to\infty}L^{(n_1+n_2)/2}\rC(\bxi^{(1)}\bunion \bseta^{(2)};\bseta^{(1)}\bunion \bxi^{(2)})=
\begin{dcases}
\rC(-1\bunion V^{(2)};1\bunion U^{(2)}),& \ell=1, \text{ and }n_1=1,\\
0,& \ell=1, \text{ and }n_1\ge 2,
\end{dcases}
\end{equation}
and similarly
\begin{equation}
\lim_{L\to\infty}L^{n_1/2}\rC(\bseta^{(1)};\bxi^{(1)})=
\begin{dcases}
1/2,& n_1=1,\\
0,& n_1\ge 2.
\end{dcases}
\end{equation}
 
Finally, inserting the parameters \eqref{eq:parameter_scale} and recalling \eqref{eq:def_F} and \eqref{eq:def_f}, we have
\begin{equation}
F_\ell(\xi_{i_\ell}^{(\ell)})=f_\ell(u_{i_\ell}^{(\ell)}),\quad F_\ell(\eta_{i_\ell}^{(\ell)})=f_\ell(v_{i_\ell}^{(\ell)})
\end{equation}
for $1\le i_\ell\le n_\ell$ and $2\le \ell\le m$, and using Lemma \ref{lm:bound_F1} we have
\begin{equation}
\lim_{L\to\infty}e^{\frac{2}{3}L^{3/2}}F_1(\xi_{i_1}^{(1)}) = f_1(-1)\cdot e^{(u_{i_1}^{(1)})^2}, \quad \lim_{L\to\infty}\frac{1}{e^{-\frac{2}{3}L^{3/2}}F_1(\eta_{i_1}^{(1)})} = \frac{1}{f_1(1)}\cdot e^{(v_{i_1}^{(1)})^2}.
\end{equation}

Therefore, if we are able to take the limit inside the integral, we see that the large $L$ limit of \eqref{eq:prelimit_dD} is $0$ for $n_1>1$, and matches \eqref{eq:def_pD} for $n_1=1$. 
Thus its limit is $\pD^{\bn}(\bbeta;\bz)$, and \eqref{eq:aux_12_28_01} holds. For \eqref{eq:aux_12_28_02}, we need the limit of the extra factor 
\begin{equation}
    \lim_{L\to\infty} H_k=\tilde H_k:=\begin{dcases}
        \sum_{i_m=1}^{n_m} (u_{i_m}^{(m)} - v_{i_m}^{(m)}), & k=m,\\
        \sum_{i_k=1}^{n_k} (u_{i_k}^{(k)} - v_{i_k}^{(k)}) - \sum_{i_{k+1}=1}^{n_{k+1}} (u_{i_{k+1}}^{(k+1)} - v_{i_{k+1}}^{(k+1)}), & 2\le k\le m-1,\\
        -2n_1-\sum_{i_2=1}^{n_2} (u_{i_2}^{(2)} - v_{i_2}^{(2)}), & k=1.
    \end{dcases}
\end{equation}
For the only nontrivial case $n_1=1$, this limit matches the factor if we take the $\beta_k$ derivative of the function  $\pD_{\bn}(\bbeta;\bz;(\alpha_1,\tau_1),\cdots,(\alpha_m,\tau_m))$:
\begin{equation}
    \begin{split}
    \frac{\partial}{\partial \beta_k}  \prod_{\ell=2}^{m} \prod_{i_\ell=1}^{n_\ell} \frac{ f_{\ell}(u_{i_\ell}^{(\ell)})}{ f_{\ell}(v_{i_\ell}^{(\ell)})} \cdot \frac{f_1(-1)}{f_1(1)} =  \prod_{\ell=2}^{m} \prod_{i_\ell=1}^{n_\ell} \frac{ f_{\ell}(u_{i_\ell}^{(\ell)})}{ f_{\ell}(v_{i_\ell}^{(\ell)})} \cdot \frac{f_1(-1)}{f_1(1)} \cdot \tilde H_k.
\end{split}
\end{equation}
Lemma \ref{lm:convergence_ptwise} follows immediately.

\bigskip

It remains to justify that we can take the large $L$ limit inside the integral. We only need to find a uniform bound which is integrable so that the dominated convergence theorem applies. In fact, applying Lemma \ref{lm:bounds_Cauchy_det} with \eqref{eq:aux_11_28_02} and \eqref{eq:aux_11_28_03}, and Lemma \ref{lm:bound_F1}, we have the following uniform bound for the right hand side of \eqref{eq:prelimit_dD}
\begin{equation}
\label{eq:aux_1_13_01}
\begin{split}
&16\pi c^{|\bn|} \cdot \prod_{\ell=1}^{m-1}  \frac{(1+|z_\ell|^{n_{\ell+1}})}{|1-z_\ell|^{n_{\ell+1}-n_\ell}|z_\ell|^{n_{\ell+1}}} \cdot \prod_{\ell=1}^m n_\ell^{n_\ell} \cdot \prod_{\ell=2}^m \prod_{i_\ell=1}^{n_\ell}\int_{\hat\Gamma_{\ell,\LL}^{\inn}\cup \hat\Gamma_{\ell,\LL}^\out} |f_\ell(u_{i_\ell}^{(\ell)})|\frac{|\rd u_{i_\ell}^{(\ell)}|}{2\pi}\int_{\hat\Gamma_{\ell,\RR}^{\inn}\cup \hat\Gamma_{\ell,\RR}^\out} \frac{1}{|f_\ell(u_{i_\ell}^{(\ell)})|} \frac{|\rd v_{i_\ell}^{(\ell)}|}{2\pi}\\
&\cdot \prod_{i_1=1}^{n_1} \int_{\hat\Gamma_{\LL}} \left|e^{(u_{i_1}^{(1)})^2} \right| \frac{|\rd u_{i_1}^{(1)}|}{2\pi} \left(\int_{-\ii L^{3/4}\cot \frac{2\pi}{5}}^{\ii L^{3/4}\cot \frac{2\pi}{5}} \left| e^{(1-c_3')(v_{i_1}^{(1)})^2}\right| \frac{|\rd v_{i_1}^{(1)}|}{2\pi} +2\int_{0}^\infty e^{-c'_4L^{3/2}(1+r_{i_1})}\frac{L^{3/4}\rd r_{i_1}}{2\pi}\right)
\end{split}
\end{equation}
which is bounded by
\begin{equation}
\label{eq:aux_11_28_04}
C^{n_1+\cdots+n_m}\cdot \prod_{\ell=1}^{m-1} \frac{(1+|z_\ell|^{n_{\ell+1}})}{|1-z_\ell|^{n_{\ell+1}-n_\ell}|z_\ell|^{n_{\ell+1}}} \cdot \prod_{\ell=1}^m n_\ell^{n_\ell}
\end{equation}
since each integral in \eqref{eq:aux_1_13_01} is uniformly bounded. The uniform bound of \eqref{eq:aux_12_28_03} is similar. Note that  $|H_k| \le  \prod_{\ell=1}^m \prod_{i_\ell=1}^{n_\ell}(2+|u_{i_\ell}^{(\ell)}|)(2+|v_{i_\ell}^{(\ell)}|)$, therefore \eqref{eq:aux_12_28_03} is uniformly bounded by \eqref{eq:aux_1_13_01}, except that each integrand needs to be multiplied by a linear factor. Such an expression still gives the bound \eqref{eq:aux_11_28_04} with a different constant $C$. This completes the proof.
\bigskip

Finally, we prove Lemma \ref{lm:uniform_bound}. Note the equations \eqref{eq:prelimit_dD} and \eqref{eq:aux_12_28_03},  the uniform bound \eqref{eq:aux_11_28_04} for both right hand sides of \eqref{eq:prelimit_dD} and \eqref{eq:aux_12_28_03}. Lemma \ref{lm:uniform_bound} follows immediately.

\section{Proof of Theorem \ref{thm:main}}
\label{sec:proof_mainthm}

In this section, we prove Theorem \ref{thm:main} using Proposition \ref{prop:main}. We split the proof of the two parts into two different subsections.

\subsection{Proof of Theorem \ref{thm:main} (a)}
\label{sec:proof_mainthm_a}

Recall the notation $\HH$ in \eqref{eq:def_HH}. Theorem \ref{thm:main} (a) is equivalent to the following convergence for arbitrary $m\ge 1$ distinct points  $(\alpha_\ell,\tau_\ell)\in \realR^2$ and arbitrary real numbers $\beta_\ell\in\realR$, $1\le \ell\le m$,
\begin{equation}
\label{eq:need_to_show}
\begin{split}
&\prob \left( \bigcap_{\ell=1}^m \left\{\HH(\alpha_\ell,\tau_\ell) \ge \beta_\ell \right\} \, \Big| \, \left\{\HH(\hat \alpha, \hat\tau) \ge \hat \beta\right\}\right) \\
&\to \hat\FT(\beta_1-\hat\beta,\cdots,\beta_m-\hat\beta; (\alpha_1-\hat\alpha,\tau_1-\hat\tau),\cdots,(\alpha_m-\hat\alpha,\tau_m-\hat\tau))
\end{split}
\end{equation}
as $L\to\infty$, where the function $\hat\FT$ is the joint probability functions of $\cH$ defined in Definition \ref{def:cH_by_tail}. Without loss of generality, we assume that $(\alpha_1,\tau_1),\cdots,(\alpha_m,\tau_m)$ are ordered in the following way, $(\alpha_1,\tau_1) \prec \cdots \prec (\alpha_m,\tau_m)$ since both sides are invariant under the permutation of the indices. Now we consider the following three cases.

If $m=1$ and $(\alpha_1,\tau_1)=(\hat \alpha, \hat \tau)$, then the left hand side of \eqref{eq:need_to_show} becomes
\begin{equation}
    \begin{split}
\frac{\prob\left( \HH(\alpha_1,\tau_1) \ge \max\{\beta_1,\hat\beta\} \right)}{\prob\left(\HH(\alpha_1, \tau_1) \ge \hat \beta\right)} \to e^{-2\max\{\beta_1-\hat\beta,0\}} &= \FT(\max\{\beta_1-\hat\beta,0\};(0,0))\\
&= \hat\FT(\beta_1-\beta;(0,0))
    \end{split}
\end{equation}
as $L\to\infty$, where the function $\FT$ and $\hat\FT$ are defined in Definitions \ref{defn:FT} and \ref{def:cH_by_tail} respectively. The convergence in the above equation follows from the estimate \eqref{eq:tail}.

If $m\ge 2$ and $(\alpha_k,\tau_k)= (\hat\alpha,\hat \tau)$ for some $1\le k\le m$, and $\beta_k\ge \hat\beta$, the left hand side of  \eqref{eq:need_to_show} becomes
\begin{equation}
\label{eq:mge2_case1}
\frac{\prob\left(  \bigcap_{\ell=1}^m \left\{\HH(\alpha_\ell,\tau_\ell) \ge \beta_\ell \right\} \right) }{\prob\left(  \HH(\hat \alpha, \hat\tau) \ge \hat \beta\right) }
\to \frac{\FT(\bbeta;(\alpha_1,\tau_1),\cdots,(\alpha_m,\tau_m))}{\FT(\hat\beta;(\hat\alpha,\hat\tau))}= e^{-\frac{2}{3}\hat\tau +2\hat\beta}\FT(\bbeta;(\alpha_1,\tau_1),\cdots,(\alpha_m,\tau_m))
\end{equation}
as $L\to\infty$. Here the convergence follows from Proposition \ref{prop:main} and the second equation follows from the formula \eqref{eq:def_FT_1}. Note that the right hand side of \eqref{eq:mge2_case1} equals to, by using Proposition \ref{prop:shift_parameters},
\begin{equation}
\label{eq:mge2_case1b}
\FT(\beta_1-\hat\beta,\cdots,\beta_m-\hat\beta;(\alpha_1-\hat\alpha,\tau_1-\hat\tau),\cdots,(\alpha_m-\hat\alpha,\tau_m-\hat\tau))
\end{equation}
which matches the right hand side of \eqref{eq:need_to_show} since $\beta_k-\hat\beta\ge 0$. Thus \eqref{eq:need_to_show} holds. Note that in this argument we assumed that $\beta_k\ge \hat\beta$. If $\beta_k<\hat\beta$, then we need to replace $\beta_k$ on the left hand side of \eqref{eq:mge2_case1} by $\hat\beta$ since the event $\HH(\alpha_k,\tau_k)\ge \beta_k$ is trivial conditioned on $\HH(\hat\alpha=\alpha_k, \hat\tau=\tau_k) \ge \hat\beta$. We then end at the limit \eqref{eq:mge2_case1b} with $\beta_k$ replaced by $\hat\beta$, which is $\FT(\beta_1-\hat\beta,\cdots,0=\max\{0,\beta_k-\hat\beta\}, \cdots, \beta_m-\hat\beta; (\alpha_1-\hat\alpha,\tau_1-\hat\tau),\cdots,(\alpha_m-\hat\alpha,\tau_m-\hat\tau))$ which still matches the right hand side of \eqref{eq:need_to_show} (see Definition \ref{def:cH_by_tail} case (i)).

For the last case, we assume that $m\ge 1$ and $(\alpha_\ell,\tau_\ell)\ne (\hat\alpha,\hat\tau)$ for all $1\le \ell \le m$. For this case, the left hand side of \eqref{eq:need_to_show} becomes
\begin{equation}
\label{eq:mge1case3}
\frac{\prob\left(  \bigcap_{\ell=1}^m \left\{\HH(\alpha_\ell,\tau_\ell) \ge \beta_\ell \right\} \bigcap \left\{\HH(\hat \alpha, \hat\tau) \ge \hat \beta\right\}\right) }{\prob\left(  \HH(\hat \alpha, \hat\tau) \ge \hat \beta\right)}.
\end{equation}
This expression has been considered in the second case if we replace the $m$ points $(\alpha_\ell,\tau_\ell)$, $\ell=1,\cdots,m$, by the following new sequence of points
\begin{equation}
(\alpha_1,\tau_1),\cdots, (\alpha_{k-1},\tau_{k-1}), (\hat \alpha,\hat\tau), (\alpha_k,\tau_k),\cdots, (\alpha_m,\tau_m)
\end{equation}
and replace the vector $\bbeta$ by $(\beta_1,\cdots,\beta_{k-1},\hat\beta,\beta_{k+1},\cdots,\beta_m)$. Here $k$ is the index such that $(\alpha_{k-1},\tau_{k-1})\prec (\hat\alpha,\hat\tau) \prec (\alpha_k,\tau_k)$. The argument of the second case implies that \eqref{eq:mge1case3} converges to 
\begin{equation}
\FT(\cdots,\beta_{k-1}-\hat\beta, 0, \beta_{k}-\hat\beta,\cdots; \cdots,(\alpha_{k-1}-\hat\alpha,\tau_{k-1}-\hat\tau), (0,0),(\alpha_{k+1}-\hat\alpha,\tau_{k+1}-\hat\tau),\cdots)
\end{equation}
here we suppressed the irrelevant coordinates in the function $\FT$ for notation simplification. Note that this matches the right hand side of \eqref{eq:need_to_show} by Definition \ref{def:cH_by_tail} case (ii). Thus \eqref{eq:need_to_show} holds.

In conclusion, for all the three cases we proved \eqref{eq:need_to_show}. Theorem \ref{thm:main} (a) follows immediately.

\subsection{Proof of Theorem \ref{thm:main} (b)}
\label{sec:proof_mainthm_b}

In this subsection we prove the second part of Theorem \ref{thm:main}. 

Suppose $(\alpha_1,\tau_1)\prec \cdots \prec (\alpha_m,\tau_m)$ are $m\ge 2$ points on $\realR^2$ and $\beta_1,\cdots,\beta_m\in\realR$. Assume that $(\alpha_k,\tau_k)=(\hat \alpha, \hat\tau)$ and $\beta_k=\hat\beta$ for some $1\le k\le m$.  We need to prove that
\begin{equation}
    \label{eq:main2}
     \lim_{L\to\infty} \prob\left( \bigcap_{\substack{1\le \ell \le m\\ \ell \ne k}} \left\{\rH(x_\ell,t_\ell) \ge h_\ell\right\} \Big| \rH(\hat x,\hat t)= \hat h \right)
    = \prob\left( \bigcap_{\substack{1\le \ell \le m\\ \ell\ne k}} \left\{ \cH_0(\alpha_\ell-\hat\alpha,\tau_\ell-\hat\tau) \ge \beta_\ell - \hat\beta \right\} \right)
\end{equation}
where 
\begin{equation}
    \label{eq:parameter_scale2}
x_\ell = \alpha_\ell L^{-1},\quad t_\ell =1+\tau_\ell L^{-3/2},\quad h_\ell =L +\beta_\ell L^{-1/2}
\end{equation}
for $1\le \ell \le m$.

Now we write the left hand side of \eqref{eq:main2} as 
\begin{equation}
    \label{eq:aux_12_28_05}
    \begin{split}
   \lim_{L\to\infty} \frac{\frac{\rd }{\rd \hat h} \prob\left( \bigcap_{\ell=1}^m \left\{\rH(x_\ell,t_\ell) \ge h_\ell\right\} \right)}{\frac{\rd}{\rd \hat h} \prob\left( \rH(\hat x,\hat t)\ge \hat h\right)} &= \lim_{L\to\infty} \frac{\frac{\partial }{\partial \beta_k} \prob\left( \bigcap_{\ell=1}^m \left\{\rH(x_\ell,t_\ell) \ge h_\ell\right\} \right) }{\frac{\partial}{\partial \beta_k} \prob\left( \rH(x_k, t_k)\ge h_k\right)} \\
   &=\lim_{L\to\infty} \frac{\frac{\partial }{\partial \beta_k} \prob\left( \bigcap_{\ell=1}^m \left\{\HH(\alpha_\ell,\tau_\ell) \ge \beta_\ell\right\} \right) }{\frac{\partial}{\partial \beta_k} \prob\left( \HH(\alpha_k,\tau_k)\ge \beta_k\right)} 
    \end{split}
\end{equation}
where $\HH(\alpha,\tau) = L^{1/2}(\rH(\alpha L^{-1},1+\tau L^{-3/2})-L)$ is defined in \eqref{eq:def_HH}.
Now we apply Proposition \ref{prop:main} and further write \eqref{eq:aux_12_28_05} as 
\begin{equation}
    \begin{split}
    &\frac{\frac{\partial }{\partial \beta_k} \FT(\beta_1,\cdots,\beta_m;(\alpha_1,\tau_1),\cdots,(\alpha_m,\tau_m)) }{\frac{\partial}{\partial \beta_k} \FT(\beta_k;(\alpha_k,\tau_k))}\\
    &=\frac{\frac{\partial }{\partial \beta_k} \FT(\beta_1-\beta,\cdots,\beta_m-\beta;(\alpha_1-\alpha,\tau_1-\tau),\cdots,(\alpha_m-\alpha,\tau_m-\tau)) }{\frac{\partial}{\partial \beta_k} \FT(\beta_k-\beta;(\alpha_k-\alpha,\tau_k-\tau))}
    \end{split}
\end{equation}
for any $\alpha,\tau,\beta\in\realR$, where we also used the property of $\FT$ when we shift the parameters as described in Proposition \ref{prop:shift_parameters}. Assume $\beta<\beta_k$. The above expression can be further written as, using Definition \ref{def:cH_by_tail},
\begin{equation}
    \begin{split}
&\frac{\frac{\partial}{\partial \beta_k}\prob\left(\bigcap_{1\le \ell \le m} \left\{\cH(\alpha_\ell-\alpha,\tau_\ell-\tau)\ge \beta_\ell-\beta\right\}\right)}{\frac{\partial}{\partial \beta_k}\prob\left(\cH(\alpha_k-\alpha,\tau_k-\tau)\ge \beta_k-\beta\right)}\\
&= \prob\left(\bigcap_{\substack{1\le \ell \le m\\ \ell\ne k}}\left\{\cH(\alpha_\ell-\alpha,\tau_\ell-\tau)\ge \beta_\ell-\beta\right\} \Big| \cH(\alpha_k-\alpha,\tau_k-\tau)=\beta_k-\beta\right)\\
&= \prob\left(\bigcap_{\substack{1\le \ell \le m\\ \ell\ne k}}\left\{\cH(\alpha_\ell-\alpha,\tau_\ell-\tau)-\cH(\alpha_k-\alpha,\tau_k-\tau)\ge \beta_\ell-\beta_k\right\} \Big| \cH(\alpha_k-\alpha,\tau_k-\tau)=\beta_k-\beta\right).
    \end{split}
\end{equation}
Now we take $\alpha=\alpha_k=\hat\alpha$ and $\tau=\tau_k=\hat\tau$, and note that $\cH_0$ is independent of $\cH(0,0)$ (see Proposition \ref{prop:properties_cH}). We see the equation above equals to the right hand side of \eqref{eq:main2}. This completes the proof.
\section{Proof of Proposition \ref{prop:crossover}} \label{sec:crossover}

In this section, we prove Proposition \ref{prop:crossover}. Since the difference between the two fields $\cH$ and $\cH_0$ is an exponential random variable $\cH(0,0)$ which doesn't depend on $\lambda$, it is sufficient to prove the proposition for one field. We prove it for the field $\cH$ in two subsections for the negative and positive time regimes respectively.

\subsection{Brownian limit in the negative time regime}

In the negative time regime, Proposition \ref{prop:crossover} follows from the following two propositions.

\begin{prop}
    \label{prop:negative_regime_limit}
    Assume $m\ge 2$ is an integer, and $\rx_1,\cdots,\rx_{m-1},\rh_1,\cdots,\rh_{m-1} \in \realR$ and $\rt_1<\cdots<\rt_{m-1}<0$ are fixed. Then
    \begin{equation}
        \label{eq:negative_regime_limit}
        \prob\left( \bigcap_{\ell=1}^{m-1} \left\{ \frac{1}{\sqrt{2\lambda}} \left(\cH\left( \frac{\lambda^{1/2}\rx_\ell}{\sqrt{2}},\lambda \rt_\ell\right) -\lambda \rt_\ell \right) \ge \rh_\ell \right\} \right)\to\prob\left(\bigcap_{\ell=1}^{m-1} \left\{ \min\{\bB_1(-\rt_\ell)+\rx_\ell,\bB_2(-\rt_\ell)-\rx_\ell\}\ge \rh_\ell\right\}\right)
    \end{equation}
    as $\lambda\to\infty$.
\end{prop}

\begin{prop}   
    \label{prop:converegence_fields_extension}
    Let Y be a random field on $\realR\times (-T,0)$ with the property that for every positive integer $d$ and real numbers $x_1,\cdots,x_d$, the cumulative distribution function $\prob\left(\cap_{\ell=1}^d \{Y(x_\ell,t_\ell) \le h_\ell\}\right)$ is continuous in the variables $h_\ell$ and $t_\ell$ for all $1\le \ell \le d$. If a sequence of random fields $Y_n$ on $\realR\times (-T,0)$ satisfies $\prob\left(\cap_{\ell=1}^d \{Y_n(x_\ell,t_\ell) \le h_\ell\}\right) \to \prob\left(\cap_{\ell=1}^d \{Y(x_\ell,t_\ell) \le h_\ell\}\right)$ as $n\to\infty$ for all $d\ge 1$, all ordered time parameters $t_1<\cdots<t_d$ and all other real parameters $x_\ell,h_\ell$, $1\le \ell\le d$, then the fields $Y_n(x,t)\to Y(x,t)$ in the sense of convergence of finite-dimensional distributions as $n\to\infty$.
\end{prop}
Proposition \ref{prop:converegence_fields_extension} was first proved in \cite[Lemma 3.6]{Liu-Wang22}, also see \cite[Lemma 2.1]{baik2024pinchedup}, where the time parameter is stated within the interval $(0,T)$ instead of $(-T,0)$. However, the statement actually holds for any time interval without changing the proof. We also note that although the assumptions of Proposition \ref{prop:converegence_fields_extension} are about the joint cumulative distribution functions $\prob\left(\cap_{\ell=1}^d \{Y_n(x_\ell,t_\ell) \le h_\ell\}\right) \to \prob\left(\cap_{\ell=1}^d \{Y(x_\ell,t_\ell) \le h_\ell\}\right)$, it can be changed to $\prob\left(\cap_{\ell=1}^d \{Y_n(x_\ell,t_\ell) \ge h_\ell\}\right) \to \prob\left(\cap_{\ell=1}^d \{Y(x_\ell,t_\ell) \ge h_\ell\}\right)$ by replacing $Y_n$ and $Y$ by $-Y_n$ and $-Y$ respectively. Therefore the two propositions above implies Proposition \ref{prop:crossover} (a).

\bigskip

Below we prove Proposition \ref{prop:negative_regime_limit}. We write 
\begin{equation}
    \label{eq:parameter_negative_regime1}
    \alpha_\ell = \frac{\lambda^{1/2}\rx_\ell}{\sqrt{2}}, \quad \tau_\ell= \lambda \rt_\ell,\quad \beta_\ell = \lambda \rt_\ell + \sqrt{2\lambda}\rh_\ell,\qquad 1\le \ell \le m,
\end{equation}
where we set
\begin{equation}
    \label{eq:parameter_negative_regime2}
    \rx_m=0,\quad \rt_m=0,\quad \rh_m=0.
\end{equation}
Then the left hand side of \eqref{eq:negative_regime_limit} equals to, using Definitions \ref{def:cH_by_tail} and \ref{defn:FT},
\begin{equation}
    \label{eq:aux_12_29_04}
    \begin{split}
    &\FT(\beta_1,\cdots,\beta_m;(\alpha_1,\tau_1),\cdots,(\alpha_m,\tau_m))\\
    &= (-1)^m \oint_{>1}\cdots \oint_{>1} \sum_{\substack{n_\ell \ge 1\\ 2\le \ell\le m}} \frac{1}{ (n_2!\cdots n_{m-1}!)^2} \cD_{\bn}(\bbeta;\bz) \prod_{\ell=1}^{m-1}\frac{\rd z_\ell}{2\pi\ii z_\ell(1-z_\ell)}.
    \end{split}
\end{equation}
Here the notations in the formula are the same as in Definition \ref{defn:FT}, especially $\bn=(n_1=1,n_2,\cdots,n_m)$, and $\cD_{\bn}$ is defined in \eqref{eq:def_cdnz}. It turns out if the parameters satisfy \eqref{eq:parameter_negative_regime1} and \eqref{eq:parameter_negative_regime2}, the main contribution of the above summation comes from the term $n_2=\cdots=n_m=1$ and all other terms converges to zero, as $\lambda\to\infty$. The asymptotic analysis is very similar to that in \cite[Section 3]{Liu-Wang22} where the leading contribution comes from the term $n_1=\cdots=n_m=1$, and that in Section \ref{sec:proof_proposition} of this paper, hence we only provide the main steps of the proof and skip the details.

The main technical part of the proof is the following two lemmas, both assuming $\bz$ is in a compact set of $(\complexC\setminus\{0,1\})^{m-1}$.

\begin{lm}
    \label{lm:point_wise_negative}
When $n_2=\cdots=n_m=1$, we have 
\begin{equation}
    \label{eq:limit_n_1}
   \begin{split}
    &\lim_{\lambda\to\infty}\cD_{\bn}(\bbeta;\bz) \\
    &= (-1)^m\prod_{\ell=1}^{m-1} (1-z_\ell)(1-z_\ell^{-1}) \\
    &\quad \cdot \prod_{\ell=2}^m \left(\frac{1}{1-z_{\ell-1}}  \int_{C_{\ell,\LL}^\inn} \frac{\rd \xi_\ell}{2\pi\ii} -\frac{z_{\ell-1}}{1-z_{\ell-1}} \int_{C_{\ell,\LL}^\out} \frac{\rd \xi_\ell}{2\pi\ii} \right)\left(\frac{1}{1-z_{\ell-1}} \int_{C_{\ell,\RR}^\inn} \frac{\rd \eta_\ell}{2\pi\ii} -\frac{z_{\ell-1}}{1-z_{\ell-1}} \int_{C_{\ell,\RR}^\out} \frac{\rd \eta_\ell}{2\pi\ii} \right)\\
    &\quad \quad \frac{\prod_{\ell=2}^m e^{\frac12(\rt_{\ell}-\rt_{\ell-1})\xi_\ell^2 +(-\rx_{\ell}+\rx_{\ell-1}+\rh_{\ell}-\rh_{\ell-1})\xi_\ell}}{\xi_2\prod_{\ell=2}^{m-1} (\xi_{\ell+1} -\xi_{\ell})}
    \cdot \frac{\prod_{\ell=2}^m e^{\frac12(\rt_{\ell}-\rt_{\ell-1})\eta_\ell^2 +(-\rx_{\ell}+\rx_{\ell-1}-\rh_{\ell}+\rh_{\ell-1})\eta_\ell}}{(-\eta_2)\prod_{\ell=2}^{m-1} (-\eta_{\ell+1} +\eta_{\ell})}
   \end{split}
\end{equation}
where $C_{m,\LL}^\inn,\cdots,C_{2,\LL}^\inn$ and $C_{2,\LL}^\out,\cdots,C_{m,\LL}^\out$ are ordered contours on the left half plane $\Re(\xi)<0$, from left to right, with orientations from $\infty e^{-\ii2\pi/3}$ to $\infty e^{\ii 2\pi/3}$, and similarly $C_{m,\RR}^\inn,\cdots,C_{2,\RR}^\inn$ and $C_{2,\RR}^\out,\cdots,C_{m,\RR}^\out$ are ordered contours on the right half plane $\Re(\eta)>0$, from right to left, with orientations from $\infty e^{-\ii \pi/3}$ to $\infty e^{\ii  \pi/3}$. Moreover,  if $n_\ell\ge 2$ for some $2\le \ell\le m$, we have 
\begin{equation}
    \lim_{\lambda\to\infty}\cD_{\bn}(\bbeta;\bz)=0.
\end{equation}
\end{lm}
\begin{lm}
    \label{lm:uniform_bound_negative}
    The following uniform bound holds for $\cD_{\bn}(\bbeta;\bz)$ for sufficiently large $\lambda$,
    \begin{equation}
      |\cD_{\bn}(\bbeta;\bz)| \le e^{-\frac43\left( \sum_{\ell=2}^m (n_\ell-1)(\rt_\ell-\rt_{\ell-1}) \right)\lambda +o(\lambda)} \cdot \prod_{\ell=2}^{m} \frac{(1+|z_{\ell-1}|)^{2n_\ell}}{|z_{\ell-1}|^{n_\ell} |1-z_{\ell-1}|^{n_\ell -n_{\ell-1}}}  C^{\sum_{\ell=2}^{m}n_\ell} \prod_{\ell=2}^{m}n_\ell^{n_\ell}
    \end{equation}
    where $o(\lambda)$ could be chosen as $0$ when $n_2=\cdots=n_m=1$.
\end{lm}

We emphasize that the contours on the right half plane all have angles $\pm\pi/3$ instead of $\pm\pi/5$. With these angles, the integrand decays super-exponentially fast since $\rt_\ell-\rt_{\ell-1}>0$ for all $\ell=2,\cdots,m$.

We postpone the proof of these two lemmas to the end of this subsection and prove Proposition \ref{prop:crossover} (a) first. Assuming Lemma \ref{lm:point_wise_negative} and Lemma \ref{lm:uniform_bound_negative}, we apply the dominated convergence theorem in \eqref{eq:aux_12_29_04} and have 
\begin{equation}
    \label{eq:aux_12_29_05}
    \begin{split}
   &\lim_{\lambda\to\infty} \FT(\beta_1,\cdots,\beta_m;(\alpha_1,\tau_1),\cdots,(\alpha_m,\tau_m))\\
   &= \prod_{\ell=1}^{m-1}\oint_{>1}\frac{\rd z_\ell}{2\pi\ii z_\ell(1-z_\ell)} (1-z_\ell)(1-z_\ell^{-1}) \\
   &\quad \cdot \prod_{\ell=2}^m \left(\frac{1}{1-z_{\ell-1}} \cdot \int_{C_{\ell,\LL}^\inn} \frac{\rd \xi_\ell}{2\pi\ii} -\frac{z_{\ell-1}}{1-z_{\ell-1}} \int_{C_{\ell,\LL}^\out} \frac{\rd \xi_\ell}{2\pi\ii} \right)\left(\frac{1}{1-z_{\ell-1}} \int_{C_{\ell,\RR}^\inn} \frac{\rd \eta_\ell}{2\pi\ii} -\frac{z_{\ell-1}}{1-z_{\ell-1}} \int_{C_{\ell,\RR}^\out} \frac{\rd \eta_\ell}{2\pi\ii} \right)\\
   &\quad \quad \frac{\prod_{\ell=2}^m e^{\frac{\rt_{\ell}-\rt_{\ell-1}}{2}\xi_\ell^2 +(-\rx_{\ell}+\rx_{\ell-1}+\rh_{\ell}-\rh_{\ell-1})\xi_\ell}}{ \xi_2 \prod_{\ell=2}^{m-1} (\xi_{\ell+1} -\xi_{\ell})}
   \cdot \frac{\prod_{\ell=2}^m e^{\frac{\rt_{\ell}-\rt_{\ell-1}}{2}\eta_\ell^2 +(-\rx_{\ell}+\rx_{\ell-1}-\rh_{\ell}+\rh_{\ell-1})\eta_\ell}}{(-\eta_2)\prod_{\ell=2}^{m-1} (-\eta_{\ell+1} +\eta_{\ell})}\\
   &=\prod_{\ell=2}^m\int_{C_{\ell,\LL}^\out} \frac{\rd \xi_\ell}{2\pi\ii}\frac{\prod_{\ell=2}^m e^{\frac{\rt_{\ell}-\rt_{\ell-1}}{2}\xi_\ell^2 +(-\rx_{\ell}+\rx_{\ell-1}+\rh_{\ell}-\rh_{\ell-1})\xi_\ell}}{ \xi_2 \prod_{\ell=2}^{m-1} ( \xi_{\ell+1} -\xi_{\ell})} \cdot \prod_{\ell=2}^m\int_{C_{\ell,\RR}^\out} \frac{\rd \eta_\ell}{2\pi\ii}\frac{\prod_{\ell=2}^m e^{\frac{\rt_{\ell}-\rt_{\ell-1}}{2}\eta_\ell^2 +(-\rx_{\ell}+\rx_{\ell-1}-\rh_{\ell}+\rh_{\ell-1})\eta_\ell}}{(-\eta_2)\prod_{\ell=2}^{m-1} (-\eta_{\ell+1} +\eta_{\ell})}
    \end{split}
\end{equation}
where we evaluated the $z_\ell$ integrals in the last step and only the terms with integrals along the contours $C_{\ell,\diamond}^\out$ survive.
Now we claim that
\begin{equation}
\label{eq:aux_1_13_05}
    \prod_{\ell=2}^m\int_{C_{\ell,\LL}^\out} \frac{\rd \xi_\ell}{2\pi\ii}\frac{\prod_{\ell=2}^m e^{\frac{\rt_{\ell}-\rt_{\ell-1}}{2}\xi_\ell^2 +(-\rx_{\ell}+\rx_{\ell-1}+\rh_{\ell}-\rh_{\ell-1})\xi_\ell}}{ \xi_2 \prod_{\ell=2}^{m-1} ( \xi_{\ell+1} -\xi_{\ell})} =\prob\left(\bigcap_{\ell=1}^{m-1} \left\{ \bB_1(-\rt_\ell)\ge \rh_\ell-\rx_\ell\right\}\right)
\end{equation}
and 
\begin{equation}
\label{eq:aux_1_13_06}
    \prod_{\ell=2}^m\int_{C_{\ell,\RR}^\out} \frac{\rd \eta_\ell}{2\pi\ii}\frac{\prod_{\ell=2}^m e^{\frac{\rt_{\ell}-\rt_{\ell-1}}{2}\eta_\ell^2 +(-\rx_{\ell}+\rx_{\ell-1}-\rh_{\ell}+\rh_{\ell-1})\eta_\ell}}{(-\eta_2)\prod_{\ell=2}^{m-1} (-\eta_{\ell+1} +\eta_{\ell})}=\prob\left(\bigcap_{\ell=1}^{m-1} \left\{ \bB_2(-\rt_\ell)\ge \rh_\ell+\rx_\ell\right\}\right).
\end{equation}
These two identities, combing with \eqref{eq:aux_12_29_05}, give the desired result \eqref{eq:negative_regime_limit}.

The two identities are equivalent (by changing the variables $\xi\to -\eta$). Hence, we only prove the first one. It is equivalent to the following identity, using the fact that $\rx_m=\rh_m=0$,
\begin{equation}
    \prod_{\ell=2}^m\int_{C_{\ell,\LL}^\out} \frac{\rd \xi_\ell}{2\pi\ii}\frac{\prod_{\ell=2}^m e^{\frac{\rt_{\ell}-\rt_{\ell-1}}{2}\xi_\ell^2 +(\hat\rh_\ell -\hat\rh_{\ell-1})\xi_\ell}}{ \xi_2 \prod_{\ell=2}^{m-1} ( \xi_{\ell+1} -\xi_{\ell})} =\prob\left(\bigcap_{\ell=1}^{m-1} \left\{ \bB_1(-\rt_\ell)\ge \hat\rh_\ell\right\}\right),\quad \rt_m=\hat\rh_m=0.
\end{equation}

We deform the contours to vertical lines and shift them to the right half plane while keeping their order. Note that this deformation does not affect the integral on the left hand side since $\rt_\ell -\rt_{\ell-1}>0$. Now $C_{2,\LL}^\out,\cdots,C_{m,\LL}^\out$ are vertical lines on the right half plane ordered from left to right and their orientations are all upward. Note that both sides go to $0$ if any $\hat\rh_\ell\to \infty$ since the coefficient of $\hat\rh_\ell$ in the exponent on the left hand side is $\xi_\ell-\xi_{\ell+1}$ (when $2\le \ell \le m-1$) or $-\xi_2$ (when $\ell=1$) which always has a negative real part. Hence it is sufficient to show 
\begin{equation}
    \label{eq:aux_12_29_06}
    \frac{\partial^{m-1}}{\partial \hat\rh_1\cdots\partial \hat\rh_{m-1}}
    \prod_{\ell=2}^m\int_{C_{\ell,\LL}^\out} \frac{\rd \xi_\ell}{2\pi\ii}\frac{\prod_{\ell=2}^m e^{\frac{\rt_{\ell}-\rt_{\ell-1}}{2}\xi_\ell^2 +(\hat\rh_\ell -\hat\rh_{\ell-1})\xi_\ell}}{ \xi_2 \prod_{\ell=2}^{m-1} ( \xi_{\ell+1} -\xi_{\ell})} =\frac{\partial^{m-1}}{\partial \hat\rh_1\cdots\partial \hat\rh_{m-1}}\prob\left(\bigcap_{\ell=1}^{m-1} \left\{ \bB_1(-\rt_\ell)\ge \hat\rh_\ell\right\}\right).
\end{equation}
A direct calculation implies the left hand side of \eqref{eq:aux_12_29_06} equals to 
\begin{equation}
    (-1)^{m-1}\prod_{\ell=2}^m\int_{C_{\ell,\LL}^\out} \frac{\rd \xi_\ell}{2\pi\ii} \prod_{\ell=2}^m e^{\frac{\rt_{\ell}-\rt_{\ell-1}}{2}\xi_\ell^2 +(\hat\rh_\ell -\hat\rh_{\ell-1})\xi_\ell} =(-1)^{m-1}\prod_{\ell=2}^m \frac{1}{\sqrt{2\pi}} \frac{1}{\sqrt{\rt_\ell-\rt_{\ell-1}}}e^{-\frac{(\hat\rh_\ell -\hat\rh_{\ell-1})^2}{2(\rt_\ell -\rt_{\ell-1})}}
\end{equation}
which matches the right hand side of \eqref{eq:aux_12_29_06}. This completes the proof.

\bigskip

The remaining part of this subsection is to prove Lemma \ref{lm:point_wise_negative} and Lemma \ref{lm:uniform_bound_negative}. As we mentioned before, due to the similarity of the arguments with \cite{Liu-Wang22} and Section \ref{sec:proof_proposition} in this paper, we only provide the main ideas of the proof and skip details. We also note that Lemma \ref{lm:uniform_bound_negative} implies the second part of Lemma \ref{lm:point_wise_negative} since $t_\ell-t_{\ell-1}>0$ and $n_\ell\ge 1$ for each $2\le \ell \le m$.

Recall the formula $\cD_{\bn}(\bbeta;\bz)$ in \eqref{eq:def_cdnz}. We choose the contours 
\begin{equation}
    \Gamma_{\ell,\LL}^{\star} = -1 + \frac{1}{\sqrt{2\lambda}}C_{\ell,\LL}^{\star},\quad \star \in \{\inn,\out\},\quad 2\le \ell\le m,
\end{equation}
and 
\begin{equation}
    \Gamma_{\ell,\RR}^{\star} = 1 + \frac{1}{\sqrt{2\lambda}}C_{\ell,\RR}^{\star},\quad \star \in \{\inn,\out\},\quad 2\le \ell\le m.
\end{equation}
Here we emphasize that the contours $\Gamma_{\ell,\RR}^{\star}$ in the formula of $\cD_{\bn}(\bbeta;\bz)$ have the angles $\pm\pi/5$ initially. We are able to change the angles to $\pm\pi/3$ since the cubic term $(\tau_\ell-\tau_{\ell-1})v^3$ in the exponent of $1/f_\ell(v)$ is nonzero by our assumptions that the times are distinct and ordered. We remind that the main reason we chose $\pm\pi/5$ in \eqref{eq:def_cdnz} is to guarantee the convergence of the integral when the times become equal. See the footnote after the equation \eqref{eq:aux_02}.

We also changes the variables accordingly
\begin{equation}
    u_{i_\ell}^{(\ell)} = -1 + \frac{1}{\sqrt{2\lambda}} \xi_{i_\ell}^{(\ell)},\quad v_{i_\ell}^{(\ell)} = 1 + \frac{1}{\sqrt{2\lambda}} \eta_{i_\ell}^{(\ell)}
\end{equation}
for $u_{i_\ell}^{(\ell)}\in \Gamma_{\ell,\LL}^{\inn}\cup \Gamma_{\ell,\LL}^{\out}$ and $v_{i_\ell}^{(\ell)} \in \Gamma_{\ell,\RR}^{\inn}\cup\Gamma_{\ell,\RR}^{\out}$. If we fix the variables $\xi_{i_\ell}^{(\ell)}, \eta_{i_\ell}^{(\ell)}$, it is direct to compute
\begin{equation}
    \label{eq:aux_12_29_02}
    \begin{split}
    &-\frac{\tau_\ell - \tau_{\ell-1}}{3}(u_{i_\ell}^{(\ell)})^3 + (\alpha_\ell -\alpha_{\ell-1}) (u_{i_\ell}^{(\ell)})^2 + (\beta_\ell -\beta_{\ell-1}) u_{i_\ell}^{(\ell)} \\
    &= -\frac{2}{3}(\rt_\ell -\rt_{\ell-1})\lambda + \frac{\sqrt{\lambda}}{\sqrt{2}}(\rx_\ell - \rx_{\ell-1}) -\sqrt{2\lambda} (\rh_\ell -\rh_{\ell-1}) \\
    &\quad + \frac12(\rt_\ell -\rt_{\ell-1})(\xi_{i_\ell}^{(\ell)})^2 + (-\rx_\ell+\rx_{\ell-1}+\rh_\ell -\rh_{\ell-1})\xi_{i_\ell}^{(\ell)} + O(\lambda^{-1/2})
    \end{split}
\end{equation}
and 
\begin{equation}
    \label{eq:aux_12_29_03}
    \begin{split}
    &-\frac{\tau_\ell - \tau_{\ell-1}}{3}(v_{i_\ell}^{(\ell)})^3 + (\alpha_\ell -\alpha_{\ell-1}) (v_{i_\ell}^{(\ell)})^2 + (\beta_\ell -\beta_{\ell-1}) v_{i_\ell}^{(\ell)} \\
    &= \frac{2}{3}(\rt_\ell -\rt_{\ell-1})\lambda + \frac{\sqrt{\lambda}}{\sqrt{2}}(\rx_\ell - \rx_{\ell-1}) +\sqrt{2\lambda} (\rh_\ell -\rh_{\ell-1}) \\
    &\quad - \frac12(\rt_\ell -\rt_{\ell-1})(\eta_{i_\ell}^{(\ell)})^2 + (\rx_\ell-\rx_{\ell-1}+\rh_\ell -\rh_{\ell-1})\eta_{i_\ell}^{(\ell)} + O(\lambda^{-1/2}).
    \end{split}
\end{equation}
This implies, if $n_2=\cdots=n_m=1$, we have 
\begin{equation}
    \prod_{\ell=2}^{m} \prod_{i_\ell=1}^{n_\ell} \frac{f_{\ell}(u_{i_\ell}^{(\ell)})}{f_{\ell}(v_{i_\ell}^{(\ell)})} \cdot \frac{f_1(-1)}{f_1(1)} \approx \prod_{\ell=2}^m \frac{e^{\frac12(\rt_\ell -\rt_{\ell-1})(\xi_{1}^{(\ell)})^2 + (-\rx_\ell+\rx_{\ell-1}+\rh_\ell -\rh_{\ell-1})\xi_{1}^{(\ell)}  }}{e^{- \frac12(\rt_\ell -\rt_{\ell-1})(\eta_{1}^{(\ell)})^2 + (\rx_\ell-\rx_{\ell-1}+\rh_\ell -\rh_{\ell-1})\eta_{1}^{(\ell)}}}.
\end{equation}
Moreover, if $n_\ell \ge 2$ for any $2\le \ell\le m$, then we have 
\begin{equation}
    \label{eq:estimate_exponential_ftn}
    \prod_{\ell=2}^{m} \prod_{i_\ell=1}^{n_\ell} \frac{f_{\ell}(u_{i_\ell}^{(\ell)})}{f_{\ell}(v_{i_\ell}^{(\ell)})} \cdot \frac{f_1(-1)}{f_1(1)} \approx e^{-\frac43\lambda\left(\rt_1+\sum_{\ell=2}^mn_\ell(\rt_\ell-\rt_{\ell-1})\right)+o(\lambda)} = e^{-\frac43\lambda\left(\sum_{\ell=2}^m(n_\ell-1)(\rt_\ell-\rt_{\ell-1})\right)+o(\lambda)},
\end{equation}
which decays exponentially as $\lambda\to\infty$. So intuitively we know that the main contribution comes from the term when $n_2=\cdots=n_m=1$.

When $n_2=\cdots=n_m=1$, we have, after a simple computation,
\begin{equation}
    \label{eq:aux_12_29_01}
    \rC(-1\bunion V^{(2)};1\bunion U^{(2)}) = \frac{-(u_1^{(2)}-1)(v_1^{(2)}+1)}{2(v_1^{(2)}-u_1^{(2)})\cdot (u_1^{(2)}+1)(v_1^{(2)}-1)}\approx - \frac{2\lambda}{\xi_1^{(2)}\cdot (-\eta_1^{(2)})},
\end{equation}
and 
\begin{equation}
    \begin{split}
    \rC(U^{(\ell)}\bunion V^{(\ell+1)}; V^{(\ell)}\bunion U^{(\ell+1)})&= -\frac{(u_1^{(\ell)}-v_{1}^{(\ell+1)})(v_1^{(\ell)}-u_{1}^{(\ell+1)})}{(u_1^{(\ell)}-v_{1}^{(\ell)})(u_1^{(\ell+1)}-v_{1}^{(\ell+1)})(u_1^{(\ell)}-u_{1}^{(\ell+1)})(v_1^{(\ell)}-v_{1}^{(\ell+1)})}\\
    &\approx -\frac{2\lambda}{(-\xi_1^{(\ell)}+\xi_{1}^{(\ell+1)})(-\eta_1^{(\ell+1)}+\eta_{1}^{(\ell)})},
    \end{split}
\end{equation}
and
\begin{equation}
    \rC(U^{(m)}; V^{(m)})= \frac{1}{u_1^{(m)}-v_1^{(m)}}\approx -\frac{1}{2}.
\end{equation}
Inserting all these estimates in \eqref{eq:def_cdnz} we get \eqref{eq:limit_n_1}.

For Lemma \ref{lm:uniform_bound_negative}, we apply Proposition \ref{lm:bounds_Cauchy_det} and get
\begin{equation}
    \begin{split}
        |\rC(-1\bunion V^{(2)};1\bunion U^{(2)})|\le 2^{\frac{1+n_2}{2}}n_2^{n_2/2} (c\lambda)^{\frac{n_2+1}{2}},\quad |\rC(U^{(m)};V^{(m)})| \le n_m^{n_m/2} (c')^{\frac{n_m}{2}}
    \end{split}
\end{equation}
and 
\begin{equation}
    \begin{split}
        |\rC(U^{(\ell)}\bunion V^{(\ell+1)}; V^{(\ell)}\bunion U^{(\ell+1)})|\le 2^{(n_\ell+n_{\ell+1})/2}n_\ell^{n_\ell/2}n_{\ell+1}^{n_{\ell+1}/2} (c\lambda)^{\frac{n_\ell+n_{\ell+1}}{2}}
    \end{split}
\end{equation}
for $2\le \ell \le m-1$, where $c>0$ is a constant such that the distance between the $\Gamma$ contours are at least $(c\lambda)^{-1/2}$, and $c'$ is a constant such that the distance between $\Gamma_{m,\LL}^\out$ and $\Gamma_{m,\RR}^\out$ is at least $(c')^{-1/2}$. Finally 
\begin{equation}
    \left|\prod_{\ell=2}^m \prod_{i_\ell=1}^{n_\ell} \frac{\rd u_{i_\ell}^{(\ell)}}{2\pi\ii}\frac{\rd v_{i_\ell}^{(\ell)}}{2\pi\ii}\right|
    = \frac{1}{(2\lambda)^{n_2+\cdots+n_m}} \left|\prod_{i_\ell=1}^{n_\ell} \frac{\rd \xi_{i_\ell}^{(\ell)}}{2\pi\ii}\frac{\rd \eta_{i_\ell}^{(\ell)}}{2\pi\ii}\right|.
\end{equation}

Inserting these bounds to \eqref{eq:def_cdnz}, and noting \eqref{eq:estimate_exponential_ftn}, we get Lemma \ref{lm:uniform_bound_negative}.

\subsection{KPZ fixed point limit in the positive time regime}

In the positive time regime, Proposition \ref{prop:crossover} follows from the following proposition.

\begin{prop}
    \label{prop:positive_regime_limit}
    Assume $m\ge 2$ is an integer, $(\rx_\ell,\rt_\ell)$, $2\le \ell \le m$, are $m-1$ points on $\realR\times (0,\infty)$ satisfying $(\rx_2,\rt_2)\prec \cdots \prec (\rx_m,\rt_m)$, and $\rh_2,\cdots,\rh_{m} \in \realR$. Then
    \begin{equation}
        \label{eq:positive_regime_limit}
        \prob\left( \bigcap_{\ell=2}^{m} \left\{ \lambda^{-1/3} \cH\left( \lambda^{2/3}\rx_\ell, \lambda \rt_\ell\right) \ge \rh_\ell \right\} \right)\to\prob\left(\bigcap_{\ell=2}^{m} \left\{ \rH(\rx_\ell,\rt_\ell)\ge \rh_\ell\right\}\right)
    \end{equation}
    as $\lambda\to\infty$, where $\rH$ is the KPZ fixed point with the narrow wedge initial condition.
\end{prop}

Below we prove this proposition.

Denote 
\begin{equation}
    \label{eq:aux_12_29_08}
    \rx_1=\rt_1=\rh_1=0.
\end{equation}
Recall Definition \ref{def:cH_by_tail} and Proposition \ref{prop:FT_alt}. We write the left hand side of \eqref{eq:positive_regime_limit} as 
\begin{equation}
    \label{eq:aux_12_29_07}
    \begin{split}
    &\FT(\beta_1,\cdots,\beta_m;(\alpha_1,\tau_1),\cdots,(\alpha_m,\tau_m))\\
    &= (-1)^m \oint_{>1}\cdots \oint_{>1} \sum_{\substack{n_\ell \ge 1\\ 2\le \ell\le m}} \frac{1}{ (n_2!\cdots n_{m-1}!)^2} \tilde\cD_{\bn}(\bbeta;\tilde\bz) \prod_{\ell=2}^{m-1}\frac{\rd z_\ell}{2\pi\ii z_\ell(1-z_\ell)},
    \end{split}
\end{equation}
where  
\begin{equation}
    \alpha_\ell =\lambda^{2/3} \rx_\ell, \quad \tau_\ell =\lambda \rt_\ell,\quad \beta_\ell = \lambda^{1/3} \rh_\ell
\end{equation}
for $1\le \ell \le m$, and the function $\tilde\cD_{\bn}(\bbeta;\tilde\bz)$ and the notations are the same as in Proposition \ref{prop:FT_alt}. Note $n_1=1$ is fixed and $\tilde\bz=(z_2,\cdots,z_{m-1})$. We copy the formula below for convenience of the readers. Note that we dropped the factor $f_1(-1)/f_1(1)=1$ by \eqref{eq:aux_12_29_08}. 
\begin{equation}
    \label{eq:def_cdnz_alt2}
            \begin{split}
                 &\tilde\cD_{\bn}(\bbeta;\tilde\bz)\\
                 &=2\prod_{\ell=2}^{m-1}(1-z_{\ell})^{n_{\ell}}(1-z_{\ell}^{-1})^{n_{\ell+1}}
                 \cdot \prod_{\ell=3}^{m} \prod_{i_\ell=1}^{n_\ell} \left(\frac{1}{1-z_{\ell-1}}\int_{\Gamma_{\ell,\LL}^\inn}\ddbar{u_{i_\ell}^{(\ell)}}{}-\frac{z_{\ell-1}}{1-z_{\ell-1}}\int_{\Gamma_{\ell,\LL}^\out}\ddbar{u_{i_\ell}^{(\ell)}}{}\right)\prod_{i_2=1}^{n_2}\int_{\Gamma_{2,\LL}^{\out}}\ddbar{u_{i_2}^{(2)}}{}\\
                 &\quad  \prod_{\ell=3}^{m} \prod_{i_\ell=1}^{n_\ell} \left(\frac{1}{1-z_{\ell-1}}\int_{\Gamma_{\ell,\RR}^\inn}\ddbar{v_{i_\ell}^{(\ell)}}{}-\frac{z_{\ell-1}}{1-z_{\ell-1}}\int_{\Gamma_{\ell,\RR}^\out}\ddbar{v_{i_\ell}^{(\ell)}}{}\right)\prod_{i_2=1}^{n_2}\int_{\Gamma_{2,\RR}^{\out}}\ddbar{v_{i_2}^{(2)}}{}\cdot \prod_{\ell=2}^{m} \prod_{i_\ell=1}^{n_\ell} \frac{f_{\ell}(u_{i_\ell}^{(\ell)})}{f_{\ell}(v_{i_\ell}^{(\ell)})} \\
            & \quad \cdot \rC(-1\bunion V^{(2)}; 1\bunion U^{(2)}) 
            \cdot \prod_{\ell=2}^{m-1} \rC(U^{(\ell)}\bunion V^{(\ell+1)}; V^{(\ell)}\bunion U^{(\ell+1)})
            \cdot \rC(U^{(m)};V^{(m)}).
            \end{split}
\end{equation}

We will apply the steepest descent method to the formula \eqref{eq:def_cdnz_alt2} of $\tilde\cD_{\bn}(\bbeta;\tilde\bz)$ when $\lambda\to\infty$. This will be done by shrinking the contours 
\begin{equation}
    \label{eq:shrink_contours}
\Gamma_{\ell,\diamond}^{\star} \to \lambda^{-1/3} \Gamma_{\ell,\diamond}^{\star},\quad 3\le \ell \le m,\quad \diamond \in \{\LL,\RR\}, \quad \star \in \{\inn,\out\}
\end{equation}
and 
\begin{equation}
    \label{eq:shrink_contours2}
\Gamma_{2,\diamond}^{\out} \to \lambda^{-1/3} \Gamma_{2,\diamond}^{\out}, \quad \diamond \in \{\LL,\RR\}.
\end{equation}
Recall that we initially require the contours $\Gamma_{\ell,\diamond}^\inn$ to be nested and outside of the points $\pm 1$. However, since the only factor that might generate poles is  $\rC(-1\bunion V^{(2)}; 1\bunion U^{(2)})$ while $u_{i_2}^{(2)}\in \Gamma_{2,\LL}^\out$ and $v_{i_2}^{(2)}\in\Gamma_{2,\RR}^\out$ which are already closer to the origin (see Figure \ref{fig:contours1} for an illustration), the deformation of the contours will not encounter any poles. 

We change variables accordingly
  \begin{equation}
    \label{eq:change_variables_cD}
    u_{i_\ell}^{(\ell)} = \lambda^{-1/3} \xi_{i_\ell}^{(\ell)}, \quad v_{i_\ell}^{(\ell)} = \lambda^{-1/3} \eta_{i_\ell}^{(\ell)}
\end{equation}
where $\xi_{i_\ell}^{(\ell)} \in \Gamma_{\ell,\LL}^\star$ and $\eta_{i_\ell}^{(\ell)} \in \Gamma_{\ell,\RR}^\star$ for all $1\le i_\ell \le n_\ell$ and $2\le \ell\le m$. Note the simple identity
\begin{equation}
    f_\ell(w)=e^{-\frac{1}{3}(\tau_\ell -\tau_{\ell-1})w^3 +(\alpha_\ell - \alpha_{\ell-1})w^2 +(\beta_\ell -\beta_{\ell-1})w} = e^{-\frac{1}{3}(\rt_\ell-\rt_{\ell-1})\zeta^3 + (\rx_\ell-\rx_{\ell-1})\zeta^2 +(\rh_\ell -\rh_{\ell-1})\zeta} :=\tilde F_\ell(\zeta),
\end{equation}
when  $2\le\ell\le m$ and $w=\lambda^{-1/3}\zeta$.  Thus we obtain 
\begin{equation}
            \begin{split}
                 &\tilde\cD_{\bn} (\bbeta;\tilde\bz)\\
                 &=2\prod_{\ell=2}^{m-1}(1-z_{\ell})^{n_{\ell}}(1-z_{\ell}^{-1})^{n_{\ell+1}}  \prod_{\ell=3}^{m} \prod_{i_\ell=1}^{n_\ell} \left(\frac{1}{1-z_{\ell-1}}\int_{ \Gamma_{\ell,\LL}^\inn}\ddbar{\xi_{i_\ell}^{(\ell)}}{}-\frac{z_{\ell-1}}{1-z_{\ell-1}}\int_{ \Gamma_{\ell,\LL}^\out}\ddbar{\xi_{i_\ell}^{(\ell)}}{}\right)\cdot \prod_{i_2=1}^{n_2}\int_{ \Gamma_{\ell,\LL}^\out}\ddbar{\xi_{i_2}^{(2)}}{}\\
                 &\quad  \prod_{\ell=3}^{m} \prod_{i_\ell=1}^{n_\ell}\left(\frac{1}{1-z_{\ell-1}}\int_{ \Gamma_{\ell,\RR}^\inn}\ddbar{\eta_{i_\ell}^{(\ell)}}{}-\frac{z_{\ell-1}}{1-z_{\ell-1}}\int_{ \Gamma_{\ell,\RR}^\out}\ddbar{\eta_{i_\ell}^{(\ell)}}{}\right) \cdot \prod_{i_2=1}^{n_2}\int_{ \Gamma_{\ell,\RR}^\out}\ddbar{\eta_{i_2}^{(2)}}{}\\
                 &\quad \prod_{\ell=2}^{m} \prod_{i_\ell=1}^{n_\ell} \frac{\tilde F_{\ell}(\xi_{i_\ell}^{(\ell)})}{\tilde F_{\ell}(\eta_{i_\ell}^{(\ell)})}\cdot \lambda^{-n_2/3}\rC(-1\bunion V^{(2)}; 1\bunion U^{(2)})
                 \cdot \prod_{\ell=2}^{m-1} \rC(\bxi^{(\ell)}\bunion \bseta^{(\ell+1)}; \bseta^{(\ell)}\bunion \bxi^{(\ell+1)})
                 \cdot \rC(\bxi^{(m)};\bseta^{(m)}).
            \end{split}
        \end{equation}

        We also note that
\begin{equation}
    \lambda^{-n_2/3}\rC(-1\bunion V^{(2)}; 1\bunion U^{(2)}) \to -\frac12\cdot \rC(\bseta^{(2)};\bxi^{(2)})
\end{equation}
as $\lambda\to\infty$. 
    Thus, assuming the integrand is uniformly bounded, we obtain the large $\lambda$ limit
    \begin{equation}
        \begin{split}
             &\lim_{\lambda\to\infty}\tilde\cD_{\bn} (\bbeta;\tilde\bz)\\
             &=-\prod_{\ell=2}^{m-1}(1-z_{\ell})^{n_{\ell}}(1-z_{\ell}^{-1})^{n_{\ell+1}}  \prod_{\ell=3}^{m} \prod_{i_\ell=1}^{n_\ell} \left(\frac{1}{1-z_{\ell-1}}\int_{ \Gamma_{\ell,\LL}^\inn}\ddbar{\xi_{i_\ell}^{(\ell)}}{}-\frac{z_{\ell-1}}{1-z_{\ell-1}}\int_{ \Gamma_{\ell,\LL}^\out}\ddbar{\xi_{i_\ell}^{(\ell)}}{}\right)\cdot \prod_{i_2=1}^{n_2}\int_{ \Gamma_{\ell,\LL}^\out}\ddbar{\xi_{i_2}^{(2)}}{}\\
             &\quad  \prod_{\ell=3}^{m} \prod_{i_\ell=1}^{n_\ell}\left(\frac{1}{1-z_{\ell-1}}\int_{ \Gamma_{\ell,\RR}^\inn}\ddbar{\eta_{i_\ell}^{(\ell)}}{}-\frac{z_{\ell-1}}{1-z_{\ell-1}}\int_{ \Gamma_{\ell,\RR}^\out}\ddbar{\eta_{i_\ell}^{(\ell)}}{}\right) \cdot \prod_{i_2=1}^{n_2}\int_{ \Gamma_{\ell,\RR}^\out}\ddbar{\eta_{i_2}^{(2)}}{}\\
             &\quad \prod_{\ell=2}^{m} \prod_{i_\ell=1}^{n_\ell} \frac{\tilde F_{\ell}(\xi_{i_\ell}^{(\ell)})}{\tilde F_{\ell}(\eta_{i_\ell}^{(\ell)})}\cdot \rC(\bseta^{(2)};\bxi^{(2)})
             \cdot \prod_{\ell=2}^{m-1} \rC(\bxi^{(\ell)}\bunion \bseta^{(\ell+1)}; \bseta^{(\ell)}\bunion \bxi^{(\ell+1)})
             \cdot \rC(\bxi^{(m)};\bseta^{(m)}).
        \end{split}
    \end{equation}

Inserting it to \eqref{eq:aux_12_29_07} and comparing it with Proposition \ref{prop:KPZ_tail_prob}, we find
\begin{equation}
    \lim_{\lambda\to\infty}\FT(\beta_1,\cdots,\beta_m;(\alpha_1,\tau_1),\cdots,(\alpha_m,\tau_m))=\prob\left(\bigcap_{\ell=2}^{m} \left\{ \rH(\rx_\ell,\rt_\ell)\ge \rh_\ell\right\}\right).
\end{equation}
 We thus obtain Proposition \ref{prop:positive_regime_limit}.
      
It remains to justify that we can take the large $\lambda$ limit within the integral and summation. In fact, we can obtain the following uniform bound for $\tilde\cD_{\bn} (\bbeta;\tilde\bz)$ due to the super-exponentially growing/decaying property of the functions $\tilde F_\ell$,
\begin{equation}
    \left|\tilde\cD_{\bn} (\bbeta;\tilde\bz)\right| \le  C^{n_2+\cdots+n_m}\prod_{\ell=2}^{m-1} \frac{(1+|z_\ell|)^{n_{\ell+1}}}{|1-z_\ell|^{n_{\ell+1}-n_\ell}|z_\ell|^{n_{\ell+1}}}\cdot \prod_{\ell=2}^m n_\ell^{n_\ell}.
\end{equation}
The proof of the above bound is almost identical to the proof of Lemma \ref{lm:uniform_bound} hence we omit the details. Using this bound, we see that the dominated convergence theorem applies and we can take the limit within the integral in \eqref{eq:def_cdnz_alt2}.

\section{Proof of Proposition \ref{prop:properties_cH}}
\label{sec:proof_property}

The first part of Proposition \ref{prop:properties_cH} follows from the definition of $\cH$ in Definition \ref{def:cH_by_tail}, especially the equation \eqref{eq:cH_00}.

For part (b), it is sufficient to show the two cases when $(0,0)\prec (\alpha,\tau)$ and $(\alpha,\tau)=(0,0)$ by the symmetry of the formula. When $(0,0)\prec (\alpha,\tau)$, the property follows from  Definition \ref{def:cH_by_tail} and the identity \eqref{eq:aux_12_31_02}. When $(\alpha,\tau)=(0,0)$, it follows from part (a) and a direct verification.

Part (c) is a direct corollary of Lemma \ref{lm:joint_density}.

The proof of (d) requires some calculations. We first write down the formula of the joint tail probability function of $\cH_0(\alpha_1,0),\cdots, \cH(\alpha_{k-1},0)$ and $\cH(\alpha_{k+1},0),\cdots, \cH_0(\alpha_m,0)$, where $\alpha_1<\cdots<\alpha_{k-1}<0$ and $0<\alpha_{k+1}<\cdots<\alpha_m$. Using Proposition \ref{prop:FT_alt} and Proposition \ref{prop:finite_dimensional_cH_0}, we have
\begin{equation}
\begin{split}
\prob \left(\bigcap_{\substack{1\le \ell\le m\\ \ell\ne k}}  \left\{\cH_0(\alpha_\ell,0)\ge \beta_\ell\right\}\right)= (-1)^m \oint_{>1}\cdots \oint_{>1} \sum_{\substack{n_\ell \ge 1\\ 2\le \ell\le m}} \frac{1}{ (n_2!\cdots n_{m-1}!)^2} \tilde\cD_{\bn}'(\bbeta;\tilde\bz) \prod_{\ell=2}^{m-1}\frac{\rd z_\ell}{2\pi\ii z_\ell(1-z_\ell)},
    \end{split}
\end{equation}
where $\tilde\bz=(z_2,\cdots,z_{m-1})$, and
\begin{equation}
    \label{eq:aux_1_13_02}
            \begin{split}
                 &\tilde\cD_{\bn}'(\bbeta;\tilde\bz)\\
                 &=-e^{-2\beta_1}\prod_{\ell=2}^{m-1}(1-z_{\ell})^{n_{\ell}}(1-z_{\ell}^{-1})^{n_{\ell+1}}
                 \cdot \prod_{\ell=3}^{m} \prod_{i_\ell=1}^{n_\ell} \left(\frac{1}{1-z_{\ell-1}}\int_{\Gamma_{\ell,\LL}^\inn}\ddbar{u_{i_\ell}^{(\ell)}}{}-\frac{z_{\ell-1}}{1-z_{\ell-1}}\int_{\Gamma_{\ell,\LL}^\out}\ddbar{u_{i_\ell}^{(\ell)}}{}\right)\prod_{i_2=1}^{n_2}\int_{\Gamma_{2,\LL}^{\out}}\ddbar{u_{i_2}^{(2)}}{}\\
                 &\quad  \prod_{\ell=3}^{m} \prod_{i_\ell=1}^{n_\ell} \left(\frac{1}{1-z_{\ell-1}}\int_{\Gamma_{\ell,\RR}^\inn}\ddbar{v_{i_\ell}^{(\ell)}}{}-\frac{z_{\ell-1}}{1-z_{\ell-1}}\int_{\Gamma_{\ell,\RR}^\out}\ddbar{v_{i_\ell}^{(\ell)}}{}\right)\prod_{i_2=1}^{n_2}\int_{\Gamma_{2,\RR}^{\out}}\ddbar{v_{i_2}^{(2)}}{}\cdot \prod_{\ell=2}^{m} \prod_{i_\ell=1}^{n_\ell} \frac{f_{\ell}(u_{i_\ell}^{(\ell)})}{f_{\ell}(v_{i_\ell}^{(\ell)})} \\ 
            & \quad \cdot \rC(-1\bunion V^{(2)}; 1\bunion U^{(2)}) 
            \cdot \prod_{\ell=2}^{m-1} \rC(U^{(\ell)}\bunion V^{(\ell+1)}; V^{(\ell)}\bunion U^{(\ell+1)})
            \cdot \rC(U^{(m)};V^{(m)}) \\
            &\quad \cdot \left(\sum_{i_k=1}^{n_k}(u_{i_k}^{(k)}-v_{i_k}^{(k)})-\sum_{i_{k+1}=1}^{n_{k+1}}(u_{i_{k+1}}^{(k+1)}-v_{i_{k+1}}^{(k+1)})\right).
            \end{split}
\end{equation}
Here we set $\alpha_k=\beta_k=0$ for notation convention in the above formula. The contours are the same as in Definition \ref{defn:FT}, also see Figure \ref{fig:contours1} for an illustration. The functions
\begin{equation}
f_\ell(w) = (\alpha_\ell-\alpha_{\ell-1})w^2 + (\beta_\ell-\beta_{\ell-1})w,\quad 2\le \ell \le m
\end{equation}
which are defined in \eqref{eq:def_f} with the parameters $\tau_\ell=0$.

We need to simplify the formula first. 
\begin{lm}
\label{lm:aux_1_13_1}
The contribution of the integral in \eqref{eq:aux_1_13_02} is zero if any of the $v_{i_\ell}^{(\ell)}$ contours is chosen as $\Gamma_{\ell,\RR}^\inn$. Moreover, $\tilde\cD_{\bn}'(\bbeta;\tilde\bz)=0$ unless $n_1=\cdots=n_m=1$.
\end{lm}
\begin{proof}[Proof of Lemma \ref{lm:aux_1_13_1}]
The key observation is that for any $\star\in\{\inn,\out\}$
\begin{equation}
\label{eq:aux_2_13_03}
\int_{R+\Gamma_{\ell,\RR}^\star} g(v_{i_\ell}^{(\ell)})e^{-f_\ell (v_{i_\ell}^{(\ell)})} \ddbar{v_{i_\ell}^{(\ell)}}{} \to 0
\end{equation}
as $R\to+\infty$ as long as the $g$ function is analytic to the right of $R+\Gamma_{\ell,\RR}^\star$ and it grows slower than $e^{(1-\epsilon)f_\ell(v)}$ for some $\epsilon>0$ when $\mathrm{Re}(v)\to\infty$. Thus, if we choose some $\Gamma_{\ell,\RR}^\inn$ when we expand the integral in \eqref{eq:aux_1_13_02}, and assume that $\ell$ is the largest index such that  $\Gamma_{\ell,\RR}^\inn$ is chosen, then we can shift the integral $\int_{\Gamma_{\ell,\RR}^\inn} $ to right by $+\infty$ and the integral vanishes. This proves the first half of the lemma. 

For the second half, we prove it inductively. If $n_2>1$, we deform all the $v_{i_2}^{(2)}$ contours as in \eqref{eq:aux_2_13_03} hence only the residues during the deformation survives. Moreover, the only residue comes from the factor $v_{i_2}^{(2)}=1$ for each $i_2$. However, note the Cauchy determinant $\rC(U^{(2)}\bunion V^{(3)}; V^{(2)}\bunion U^{(3)})$ which contains a factor $v_{1}^{(2)}-v_2^{(2)}=0$. Thus, the residue vanishes in this case and we proved the case when $n_2>1$. Now we assume $n_2=\cdots=n_{\ell-1}=1$ but $n_\ell>1$. Our earlier argument implies that we could deform all the $v_1^{(2)},\cdots, v_{1}^{(\ell-1)}$ contours to infinity and only the residue $v_{1}^{(2)}=\cdots=v_1^{(\ell-1)}=1$ survives. We similarly deform the $v_{i_\ell}^{(\ell)}$ contours to infinity as in \eqref{eq:aux_2_13_03} and only the residue at $v_{i_\ell}^{(\ell)}= v_{1}^{(\ell-1)}=1$ survives. The residue is zero due to the Cauchy factor $\rC(U^{(\ell)}\bunion V^{(\ell+1)}; V^{(\ell)}\bunion U^{(\ell+1)})$. This finishes the induction.
\end{proof}

Applying Lemma \ref{lm:aux_1_13_1}, we only need to evaluate $\tilde\cD_{\bn}'(\bbeta;\tilde\bz)$ when $n_1=\cdots=n_m=1$ and all the $v$ contours are chosen as $\Sigma_{\ell,\RR}^\out$. In this case, by consider the $z_\ell$ integrals and letting $z_\ell$ contours to infinity, we find that all the $u$ contours can only be chosen as $\Sigma_{\ell,\LL}^\out$. Thus we obtain, after simplifying the notations,
\begin{equation}
\begin{split}
&\prob \left(\bigcap_{\substack{1\le \ell\le m\\ \ell\ne k}}  \left\{\cH_0(\alpha_\ell,0)\ge \beta_\ell\right\}\right)= (-1)^{m-1}e^{-2\beta_1} \prod_{\ell=2}^m \int_{\Gamma_{\ell,\LL}^\out} \ddbar{u_\ell}{} \int_{\Gamma_{\ell,\RR}^\out} \ddbar{v_\ell}{} \prod_{\ell=2}^m \frac{f_\ell(u_\ell)}{f_\ell(v_\ell)}\\
&\quad \det\begin{bmatrix}
\frac{1}{-1-1}&\frac{1}{-1-u_2}\\ \frac{1}{v_2-1}&\frac{1}{v_2-u_2} 
\end{bmatrix}\cdot \prod_{\ell=2}^{m-1}
\det\begin{bmatrix}
    \frac{1}{u_\ell-v_\ell} & \frac{1}{u_\ell-u_{\ell+1}} \\ \frac{1}{v_{\ell+1}-v_\ell} & \frac{1}{v_{\ell+1}-u_{\ell+1}}
\end{bmatrix}
\cdot \frac{1}{u_m-v_m} \cdot \left(u_k-v_k-u_{k+1}+v_{k+1}\right).
    \end{split}
\end{equation}
By further deforming the contours $\Sigma_{\ell,\RR}^\out$ to infinity and evaluating the residues at $1$, we get
\begin{equation}
\label{eq:aux:1_13_04}
\begin{split}
&\prob \left(\bigcap_{\substack{1\le \ell\le m\\ \ell\ne k}}  \left\{\cH_0(\alpha_\ell,0)\ge \beta_\ell\right\}\right)\\
&= (-1)^{m-1}e^{\alpha_1-\beta_1-\alpha_m-\beta_m} \prod_{\ell=2}^m \int_{\Gamma_{\ell,\LL}^\out} \ddbar{u_\ell}{} \left[\prod_{\ell=2}^m f_\ell(u_\ell)\right] \cdot \left[ \frac{1}{-1-u_2} \cdot \prod_{\substack{2\le \ell \le m-1\\ \ell\ne k}}\frac{1}{u_\ell-u_{\ell+1}} \cdot  \frac{1}{u_m-1}\right].
    \end{split}
\end{equation}
Now we change variables $u_\ell=-1+\hat u_\ell$ for $2\le \ell \le k$ and $u_\ell=1+\hat u_\ell$ for $\ell=k+1,\cdots,m$. \eqref{eq:aux:1_13_04} is the product of the following two terms
\begin{equation}
 \prod_{\ell=2}^k \int_{1+\Gamma_{\ell,\LL}^\out}  \frac{\prod_{\ell=2}^k e^{(\alpha_\ell-\alpha_{\ell-1})\hat u_\ell^2 +(\beta_\ell-2\alpha_\ell-(\beta_{\ell-1}-2\alpha_{\ell-1}))\hat u_\ell}}{\hat u_2\prod_{\ell=2}^{k-1}(\hat u_{\ell+1}-\hat u_\ell)}=\prob\left(\bigcap_{\ell=1}^{k-1} \left\{ \bB_1(-2\alpha_\ell)\ge \beta_\ell-2\alpha_\ell\right\}\right)
\end{equation}
by \eqref{eq:aux_1_13_05}, and
\begin{equation}
 \prod_{\ell=k+1}^m \int_{-1+\Gamma_{\ell,\LL}^\out}  \frac{\prod_{\ell=k+1}^m e^{(\alpha_\ell-\alpha_{\ell-1})\hat u_\ell^2 +(\beta_\ell+2\alpha_\ell-(\beta_{\ell-1}+2\alpha_{\ell-1}))\hat u_\ell}}{(-\hat u_m)\prod_{\ell=k+1}^{m-1}(\hat u_{\ell+1}-\hat u_\ell)}=\prob\left(\bigcap_{\ell=k+1}^{m} \left\{ \bB_2(2\alpha_\ell)\ge \beta_\ell+2\alpha_\ell\right\}\right)
\end{equation}
by \eqref{eq:aux_1_13_06} and a shift of index, where $\bB_1$ and $\bB_2$ are two independent Brownian motions. Thus we obtain
\begin{equation}
\prob \left(\bigcap_{\substack{1\le \ell\le m\\ \ell\ne k}}  \left\{\cH_0(\alpha_\ell,0)\ge \beta_\ell\right\}\right) = \prob\left(\bigcap_{\substack{1\le \ell\le m\\ \ell\ne k}} \left\{ \Bts(2\alpha_\ell)-|2\alpha_\ell|\ge \beta_\ell \right\}\right)
\end{equation}
and we proved the last part of Proposition \ref{prop:properties_cH}.

\end{document}